\documentclass[12pt]{amsart}
\usepackage{amssymb,amsfonts,latexsym,amsmath,mathdots,comment,tikz-cd}
\usepackage[mathscr]{eucal}
\usepackage{bbm}
\usepackage[vcentermath]{youngtab}
\usepackage{color}
\usepackage[all]{xy}
\usepackage{amscd}

\numberwithin{equation}{section}
\numberwithin{table}{section}
\numberwithin{figure}{section}

\def\naive{{na\"ive}}

\def\Dokovic{{\DH}okovi{\'c}}

\def\cf{cf.~}

\newcommand{\rb}[1]{\raisebox{1.5ex}[0pt]{#1}}

\newcommand{\hsp}[1]{{\hbox{\hspace{#1}}}}

\newcommand{\mystack}[2]{\ensuremath{ \substack{ \hbox{\tiny{${#1}$}} \\ \hbox{\tiny{${#2}$}} }} }

\newcounter{letcnt} 

\def\a{\alpha}  
\def\b{\beta}  
\def\d{\delta}  
\def\e{\varepsilon}  

\def\z{\zeta}

\def\s{\sigma}

\def\w{\omega}

\def\fa{\mathfrak{a}} 
  
\def\tAd{\mathrm{Ad}} \def\tad{\mathrm{ad}}
 
\def\tAut{\mathrm{Aut}}

 \def\sB{\mathscr{B}}
 
\def\fb{\mathfrak{b}} 
 
\def\bC{\mathbb C} \def\cC{\mathcal C}

\def\tcodim{\mathrm{codim}}

 \def\tdim{\mathrm{dim}}
 \def\sE{\mathscr{E}}
 \def\ttE{\mathtt{E}}
\def\tEnd{\mathrm{End}} 

\def\be{\mathbf{e}}  
 
\def\texp{\mathrm{exp}}

\def\tGL{\mathrm{GL}}

\def\fg{{\mathfrak{g}}} 

\def\fgl{\mathfrak{gl}}

 \def\sH{\mathscr{H}}

\def\fh{\mathfrak{h}} 
\def\bh{\mathbf{h}}
\def\bi{\mathbf{i}}

 \def\tIm{\mathrm{Im}}

\def\fk{\mathfrak{k}}

 \def\tker{\mathrm{ker}}
\def\cL{\mathcal L} 
 
\def\fl{\mathfrak{l}}

\def\cM{\mathcal M}

 \def\cN{\mathcal N}
 \def\sN{\mathscr{N}}
\def\fn{\mathfrak{n}} 
\def\bn{\mathbf{n}}

\def\tNilp{\mathrm{Nilp}}

 \def\cO{\mathcal O}
\def\tO{\mathrm{O}}

\def\bP{\mathbb P} \def\cP{\mathcal P}

\def\fp{\mathfrak{p}} 
 
\def\tprim{\mathrm{prim}}

\def\bQ{\mathbb Q}

\def\bR{\mathbb R} \def\cR{\mathcal R}
 \def\sR{\mathscr{R}}

 \def\cS{\mathcal S}
\def\sS{\mathscr{S}}
  
\def\ttS{\mathtt{S}}

\def\tss{\mathrm{ss}}
\def\tSL{\mathrm{SL}} \def\tSO{\mathrm{SO}}
\def\tSp{\mathrm{Sp}}

 \def\tspan{\mathrm{span}}
\def\fsl{\mathfrak{sl}} \def\fso{\mathfrak{so}} 
\def\fsp{\mathfrak{sp}} \def\fsu{\mathfrak{su}}

\def\ft{\mathfrak{t}}

\def\cW{\mathcal W} \def\sW{\mathscr{W}}

\def\by{\mathbf{y}}
   \def\bZ{\mathbb Z}
\def\ttZ{\mathtt{Z}} \def\sZ{\mathscr{Z}}
\def\fz{\mathfrak{z}} 
 \def\bz{\mathbf{z}}
\def\half{\tfrac{1}{2}}

\def\one{\mathbbm{1}}
\def\tand{\quad\hbox{and}\quad}

\def\sb{{\hbox{\tiny{$\bullet$}}}}

\def\inj{\hookrightarrow}

\def\op{\oplus}
\def\ot{\otimes}

\newcounter{numcnt}
\newenvironment{numlist}{
   \begin{list}{{\small{(\arabic{numcnt})}}}
   {\usecounter{numcnt} 
    \setlength{\itemsep}{3pt}
    \setlength{\leftmargin}{25pt} 
    \setlength{\labelwidth}{20pt} 
    \setlength{\listparindent}{20pt} }
   }
   {\end{list}}
\newcounter{cnt}

\newcounter{acnt}
\newenvironment{a_list}{ 
  \begin{list}{{(\alph{acnt})}}
   {\usecounter{acnt} \setlength{\itemsep}{3pt}
    \setlength{\leftmargin}{25pt} 
    \setlength{\labelwidth}{20pt}
    \setlength{\listparindent}{20pt} }
   }
   {\end{list}}

\newcounter{Acnt}

\newcounter{icnt}

\newenvironment{i_list}{ 
  \begin{list}{{(\roman{icnt})}}
   {\usecounter{icnt} 
    \setlength{\itemsep}{3pt} 
    \setlength{\parsep}{0pt} 
    \setlength{\topsep}{0pt} 
    \setlength{\leftmargin}{25pt} 
    \setlength{\labelwidth}{20pt}
    \setlength{\listparindent}{20pt} }
   }
   {\end{list}}
\newcounter{Icnt}

\newcounter{exam_cnt}

\newcounter{mccnt}

\newenvironment{bcirclist}{ 
  \begin{list}{\boldmath$\circ$\unboldmath}
   {\usecounter{cnt} \setlength{\itemsep}{2pt}
    \setlength{\leftmargin}{15pt} \setlength{\labelwidth}{20pt}
    \setlength{\listparindent}{20pt} }
   }
   {\end{list}}
\newenvironment{blist}{ 
  \begin{list}{$\bullet$}
   {\usecounter{cnt} \setlength{\itemsep}{2pt}
    \setlength{\leftmargin}{20pt} \setlength{\labelwidth}{20pt}
    \setlength{\listparindent}{20pt} }
   }
   {\end{list}}

\newtheorem{corollary}[equation]{Corollary}
\newtheorem{lemma}[equation]{Lemma}
\newtheorem{proposition}[equation]{Proposition}
\newtheorem{theorem}[equation]{Theorem}

\theoremstyle{definition}

\newtheorem*{boldQ*}{Question}
\newtheorem*{boldP*}{Problem}

\theoremstyle{definition}

\theoremstyle{remark}
\newtheorem*{assume*}{Assume}
\newtheorem*{answer*}{Answer}
\newtheorem{claim}[equation]{Claim}
\newtheorem*{claim*}{Claim}

\newtheorem{definition}[equation]{Definition}
\newtheorem*{definition*}{Definition}
\newtheorem{example}[equation]{Example}
\newtheorem*{example*}{Example}
\newtheorem*{hint*}{Hint}
\newtheorem*{notation*}{Notation}
\newtheorem{remark}[equation]{Remark}
\newtheorem*{remark*}{Remark}
\newtheorem*{remarks*}{Remarks}
\newtheorem*{fact*}{Fact}
\newtheorem*{emphL*}{Lemma}

\newtheorem*{emphQ*}{Question}
\newtheorem*{emphA*}{Answer}

\def\uz{\underline{z}}
\def\uy{\underline{y}}

\def\fgR{\fg_\bR}
\def\cC{\mathcal{C}}
\def\cO{\mathcal{O}}
\def\PP{\mathbb{P}}
\def\be{\underline{\epsilon}}
\def\bfD{\mathbf{\Delta}}
\def\bfN{\mathbf{N}}

\def\blank{\hbox{\ }}
\def\rmI{\mathrm{I}}
\def\rmII{\mathrm{II}}
\def\rmIII{\mathrm{III}}
\def\rmIV{\mathrm{IV}}
\def\schmid{MR0382272}
\def\Schmid{MR0382272}
\def\BP{MR3133298}
\def\CK{MR664326}
\def\CK2{MR1042802}
\def\DL{DtoCK}
\def\GGK{MR2918237}
\def\HP{HayPearl}
\def\KP{KP2012}
\def\CKSdeg{MR840721}
\def\CKS{MR840721}
\def\CKpmhs{MR664326}
\def\CKextn{MR0432925}
\def\CoMc{MR1251060}

\def\beq{\begin{equation}}
\def\eeq{\end{equation}}

\begin{document}

\title[Polarized relations on horizontal $\tSL(2)$'s]{Polarized relations on horizontal $\tSL(2)$'s}

\author[Kerr]{M. Kerr}
\email{matkerr@math.wustl.edu}
\address{Department of Mathematics, Washington University in St. Louis, Campus Box 1146, One Brookings Drive, St. Louis, MO 63130-4899}
\author[Pearlstein]{G. Pearlstein}
\email{gpearl@math.tamu.edu}
\address{Mathematics Department, Mail stop 3368, Texas A\&M University, College Station, TX  77843}
\author[Robles]{C. Robles}
\email{robles@math.duke.edu}
\address{Mathematics Department, Duke University, Box 90320, Durham, NC  27708-0320} 
\thanks{The authors gratefully acknowledge the National Science Foundation partial support though the DMS grants 1361147 (Kerr), 1361120 (Pearlstein and Robles), and 1309238 (Robles).}

\date{\today}

\begin{abstract}
We introduce a relation on real conjugacy classes of $\mathrm{SL}(2)$-orbits in a Mumford-Tate domain $D$. The relation answers the question \emph{when is one $\bR$--split polarized mixed Hodge structure more singular/degenerate than another?}  The relation is compatible with natural partial orders on the sets of nilpotent orbits in the corresponding Lie algebra and boundary orbits in the compact dual.

A generalization of the $\mathrm{SL}(2)$-orbit theorem to such domains leads to an algorithm for computing this relation.  The relation is then worked out in several examples and special cases, including period domains, Hermitian symmetric domains, and complete flag domains.  

Although the above relation is not in general a partial order, it leads (via cubical sets) to a poset of equivalence classes of multivariable nilpotent orbits on $D$.  The elements of this poset encode the possible degeneracy relations amongst the polarized mixed Hodge structures that arise in a several-variable degeneration of Hodge structure. We conclude with an example illustrating a link to mirror symmetry for Calabi-Yau VHS.
\end{abstract}

\subjclass[2010]
{
}

\bibliographystyle{amsalpha}

\maketitle

\setcounter{tocdepth}{1}

\tableofcontents

\section{Introduction} \label{S:intro}
\subsection{Objective} \label{S:1.1}

The purpose of this article is to use representation theory to better understand the constraints on several-variable degenerations of Hodge structure, and hence (via the period map) on degenerations of algebraic varieties along a local normal crossing divisor.  Polarizable nilpotent cones $\sigma = \bR_{>0}\langle N_1,\ldots ,N_r \rangle$ in a reductive Lie algebra $\fg_{\bR}$ are the basic objects underlying such degenerations:  any unipotent variation of Hodge structure over $(\Delta^*)^r$ is approximated on the universal cover by a (Hodge-theoretic) nilpotent orbit $\theta(\underline{z})=e^{\sum z_j N_j} F^{\bullet}_{\infty}$, where the $N_j \in \fg_{\bQ}$ are the monodromy logarithms.  It is these cones that we would like to somehow classify, for polarized variations with arbitrary Hodge numbers and Mumford-Tate group $G$.

We recall that $G\leq \mathrm{Aut}(V,Q)$ is the reductive, connected $\bQ$-algebraic group fixing all Hodge tensors of a polarized Hodge structure $(V,Q,\varphi)$ on a $\bQ$-vector space $V$.  However, the Lie group $G(\bR)$ of real points need not be topologically connected; and we shall primarily work with the identity connected component $G(\bR)^+$, of which the (connected) Mumford-Tate domain $D:=G(\bR)^+.\varphi \cong G(\bR)^+/G^0(\bR)$ is an orbit.  The domain $D$ is an analytic open subset of its compact dual $\check{D}=G(\bC).F_{\varphi}^{\bullet}$.  Given a period map $\Phi: (\Delta^*)^r \to \Gamma \backslash D$ with generic M-T (Mumford-Tate) group $G$, the approximating nilpotent orbit $\theta(\underline{z})$ has the property that $\theta(\underline{z})\in D$ when all $\text{Im}(z_j)\gg 0$ ($\theta$ is polarized), and $F^{\bullet}_{\infty} \in \check{D}$ satisfies $N_j F^{\bullet}_{\infty}\subseteq F^{\bullet-1}_{\infty}$ ($\theta$ is horizontal).

Our present goal is to construct a ``combinatorially computable'' finite poset comprising suitable equivalence-classes of these $\{\theta\}$, in such a way as to render transparent the relation between the stratification of a cone $\sigma$ (in a given class) by its faces and the stratification of $\mathrm{Nilp}(\fg_{\bR})$ and $\partial D \subset \check{D}$ by $G(\bR)^+$-orbits. To obtain a reasonable classification, we shall jettison much of the rational structure, allowing the $N_j$ to be real so that we may act by $G(\bR)^+$ on the set of all such $\{\theta\}$. Unfortunately, for $r>1$, what remains is still a ``wild'' problem -- for instance, this action typically has infinitely many orbits.  In order to find some structure in the situation, we are therefore led to study equivalence classes modulo the action of $G(\bR)^+$ on each face of $\sigma$  \emph{individually}.  The full strength of the multivariate $\mathrm{SL}(2)$-orbit theory, adapted to M-T domains, must be brought to bear to determine how the resulting equivalence-classes of faces may fit together -- that is, which limiting MHS-types can admissibly degenerate into which.  

In order to describe these goals (and our results) more precisely, we shall introduce the main objects of study $\Psi_D$, $\mathbf{N}_D$, and $\mathbf{\Delta}_D$ in the next subsection. The initial stimulus for this paper was to relate work by the third author on $\Psi_D$ \cite{SL2} to work of the first two authors on $\mathbf{\Delta}_D$ \cite{MR3331177} and Hodge-theoretic boundary components \cite{KP2012}, and to the study of the partial order on $\mathbf{\Delta}_D$ for adjoint domains in \cite{KR1}.

Before going further, we briefly address why we work in the more general setting of Mumford-Tate domains \cite{MR2918237}, rather than sticking to period domains (i.e. to the case $G=\tAut(V,Q)$). While long familiar in the setting of Shimura varieties (cf. \cite{K-sv,FL}), period maps into such subdomains are increasingly common in algebraic geometry, whether in the context of ``occult'' period maps arising from cyclic covers (e.g. \cite{ACT,LPZ}), or from other motivations related to arithmetic or exceptional groups (e.g. \cite{dSKP,Katz,Yun}). In these and related settings, it is important to be able to compute the restrictions imposed by the M-T group on the LMHS types and on their degeneracies into one another.  

Also significant is that the definitions and results are simply more natural in the representation-theoretic language.  For instance, while the relations $\leq,\preceq$ on the set $\Psi_D$ are not partial orders in general (or for most period domains), they actually yield a \emph{linear} order when $D$ is Hermitian symmetric ($\S$\ref{S:HS}), and for ``complete flag'' domains $\leq$ (but not $\preceq$) yields a partial order ($\S$\ref{S:fd}).  Moreover, while some of the results for period domains can be stated in terms of Hodge numbers (e.g. Theorem \ref{T:PDpr}), we are unaware of such a formulation for the secondary poset in $\S$\ref{S:eg}.

\subsection{Definitions} 

Let $D = G(\bR)^+/ G^0(\bR)$ be a Mumford--Tate domain parameterizing weight $k$, $Q$--polarized Hodge structures $F$ on $V$.  (Henceforth we shall drop the superscript bullets on these Hodge flags.)  By the 1-variable case of Schmid's Nilpotent Orbit Theorem \cite{\schmid}, a period map $\Phi:\Delta^* \to \Gamma\backslash D$ (or rather, its lift $\tilde{\Phi}:\mathfrak{H}\to D$) is asymptotically approximated by a nilpotent orbit
\begin{equation} \label{E:noi}
  z \ \mapsto \ e^{zN} F \, .
\end{equation}
Here $F$ is a point in the compact dual $\check D = G(\bC)/P$ of $D$, $N$ is a nilpotent element of the Lie algebra $\fg_\bR$ of $G(\bR)^+$, $z\in \bC$, and $e^{zN} F \in D$ for $\tIm\,z \gg0$.    The ``\naive~ limit'' 
\[
  F_\infty \ := \ \lim_{\tIm\,z \to \infty} e^{zN} F
\]
of the nilpotent orbit lies in the analytic closure $\overline D$ of $D$ in the compact dual.

\subsubsection{Polarized mixed Hodge structures} 

Recall that, given $F \in \check D$ and a nilpotent $N \in \fg_\bR$, the map \eqref{E:noi} is a nilpotent orbit on $D$ if and only if $(F,N)$ is a \emph{polarized mixed Hodge structure} (PMHS) on $D$, cf.~\cite[(6.16)]{\schmid} and \cite[(3.13)]{\CKSdeg}.  Associated to $N$ is a monodromy weight filtration $W=W(N)[-k]$ on $V$, see \S\ref{S:JMf}.  Then $(F,N)$ is a PMHS if and only if $NF^p \subset F^{p-1}$ ($\forall p$), $F$ induces a weight-$m$ HS on each $\mathrm{Gr}^W_m V$, and $Q_\ell (u,v)=Q (u,N^{\ell} v)$ polarizes each primitive subspace $(\mathrm{Gr}^W_{k+\ell} V)_{\text{prim}}$ ($\forall \ell\geq 0$).

\subsubsection{Horizontal $\tSL(2)$'s} 
A particularly nice class of nilpotent orbits are those arising from horizontal $\tSL(2)$s.  Schmid's $\tSL(2)$--Orbit Theorem \cite{\schmid} asserts that every one-variable nilpotent orbit is asymptotically approximated by a horizontal $\tSL(2)$--orbit.  (The several-variable $\tSL(2)$--Orbit Theorem is due to Cattani, Kaplan and Schmid \cite{\CKSdeg}.)  
The horizontal $\tSL(2)$--orbits on $D$ are the nilpotent orbits on $D$ with the property that the PMHS is $\bR$--split (which is to say, the associated actual MHS $(F,W(N)[-k])$ is $\bR$--split).  

Set 
\[
  \tilde B_\bR(D) \ := \ 
  \left\{ (F,N) \in \check D \times \tNilp(\fg_\bR) \ | \ (F,N)
  \hbox{ is an $\bR$--split PMHS on $D$}\right\}\,.
\] 
Let $\tNilp(\fg_\bR) \subset \fg_\bR$ denote the set of nilpotent endomorphisms.  Then we have maps 
\begin{equation} \label{E:diagram1}
\begin{tikzcd}
  & \tilde B_\bR(D) \arrow[rd, "\phi_\infty"] \arrow[ld,"\pi"'] & \\
  \tNilp(\fg_\bR) & & \check D\,,
\end{tikzcd}
\end{equation}
where $\pi$ is the projection $(F,N) \mapsto N$ onto the second factor, and
\[
(F_{\infty}=\ )\  \phi_\infty(F,N) \ := \ \lim_{y \to + \infty} \exp(\bi y\,N)F 
  \ \in \ \overline D 
\]
is the \naive~limit map.  (Here and throughout, we denote $\sqrt{-1}$ by a boldface $\mathbf{i}$.) In the case of horizontal $\tSL(2)$'s, the latter is related to the PMHS by

\begin{proposition}[{\cite[\S5.1]{MR3331177}}] \label{P:sorb}
When the nilpotent orbit is an $\tSL(2)$--orbit, $F$ and $F_\infty$ lie in the same $G(\bR)^+$--orbit.
\end{proposition}

\subsubsection{Congruence classes} 

The group $G(\bR)^+$ acts on both $\overline D$ and $\tNilp(\fg_\bR)$ (via the adjoint action in the second case\footnote{In an unfortunate clash of nomenclature, these orbits are also called ``nilpotent orbits'' in representation theory.}).  The Schmid and Cattani--Kaplan--Schmid orbit theorems, and their r\^ole in the analysis of degenerations of Hodge structure, lead us to consider the set
\begin{eqnarray*}
  \Psi_D & := & \left.\{
  \hbox{$G(\bR)^+$--conjugacy classes of pairs 
  $(F,N) \in \check D \times \tNilp(\fg_\bR)$}\right.\\
  & & \left.\hbox{ such that $(F,N)$ is an $\bR$--split PMHS on $D$}\right\}
\end{eqnarray*}
introduced in \cite{SL2}.  Setting
\begin{eqnarray*}
  \bfD & := & \left\{ 
  \hbox{$G(\bR)^+$--orbits in the analytic closure 
  $\overline D$ of $D$ in the compact dual $\check D$}\right\}\,,\\
  \bfN & := & \left\{
  \hbox{$G(\bR)^+$--conjugacy classes in $\tNilp(\fg_\bR)$}\right\}\,,
\end{eqnarray*}
the maps of \eqref{E:diagram1} descend to well-defined maps 
\begin{equation} \label{E:diagram2}
\begin{tikzcd}
  & \Psi_D \arrow[rd, "\phi_\infty"] \arrow[ld,"\pi"'] & \\
  \bfN & & \bfD
\end{tikzcd}
\end{equation}
on the quotients.

Given $[F,N]\in\Psi_D$, we say that $\phi_\infty([F,N]) \in \Delta$ is the \emph{boundary orbit polarized by $[F,N]$} and that $N$ is a \emph{polarizing nilpotent}.  Let 
\[
  \bfD_D \ := \ \phi_\infty(\Psi_D) \ \subset \ \bfD
\] 
denote the \emph{polarizable boundary strata}, and 
\[
  \bfN_D \ := \ \pi(\Psi_D) \ \subset \ \bfN
\] 
the (conjugacy classes of) polarizing nilpotents.  Thus we obtain a subdiagram
\begin{equation} \label{E:diagram3}
\begin{tikzcd}
  & \Psi_D 
  \arrow[rd, two heads, "\phi_\infty"] \arrow[ld, two heads, "\pi"'] & \\
  \bfN_D & & \bfD_D
\end{tikzcd}
\end{equation}
of \eqref{E:diagram2}.  Surprisingly, we have

\begin{theorem} \label{T:0}
The map $\phi_\infty : \Psi_D \to \bfD_D$ is a bijection.
\end{theorem}

\noindent Theorem \ref{T:0} is proved in \S\ref{S:T0}.  In contrast, the map $\pi : \Psi_D \to \bfN_D$ generally fails to be injective; see Example \ref{eg:borel7}.  

\subsection{Motivation}

On each of $\Psi_D$, $\bfD$ and $\bfN$, we introduce ``relations'', which for the latter two sets are partial orders given by ``containment in closure''.  Schmid's several-variable Nilpotent Orbit Theorem leads to a notion of a ``polarizable relation.''  Very roughly, these are the relations that are ``Hodge--theoretically realizable.''

It is these polarizable relations which are our main object of study, along with the relationships between $\Psi_D$, $\bfD$ and $\bfN$, especially as encoded in Hodge-theoretically natural maps \eqref{E:diagram2} preserving the relations.  Our efforts are motivated by the expectation that the polarizable relations will reflect the boundary structure of a partial compactification $\overline{\Gamma\backslash D}$.  Given an extension $\overline\cM \to \overline{\Gamma\backslash D}$ of a period map, this would in turn provide some information on the boundary $\overline\cM \backslash \cM$ of the moduli space.

\begin{example}[Period domain for $\bh = (g,g)$] \label{eg:curve0}
A familiar classical example is the period domain parameterizing weight 1 Hodge structures.  In this case $\Psi_D$ consists of $g+1$ elements, the relations are all polarized and define a linear order.  In particular, we may enumerate the elements $[F_a,N_a] \in \Psi_D$, $0 \le a \le g$, so that the linear order may be visualized as 
\[
  [F_0,N_0] \ \to \ [F_1,N_1] \ \to \ [F_2,N_2] \ 
  \to \cdots \to \ [F_g,N_g] \,,
\]
where each arrow $\to$ represents a generating relation $<$ of the linear order.  Specializing to $g=2$, we have $[F_0,N_0] \to [F_1,N_1]\to [F_2,N_2]$.  Geometrically, these polarized relations are realized by degenerations of the form
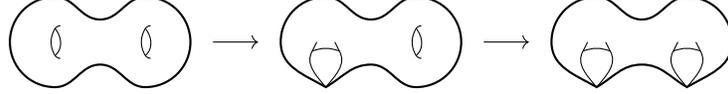
\begin{center}
\begin{tikzpicture}[scale=0.6]
\draw[thick] (0,0) to [out=90,in=180] (1,1)
  to [out=0,in=180] (2,0.5)
  to [out=0,in=180] (3,1)
  to [out=0,in=90] (4,0)
  to [out=270,in=0] (3,-1)
  to [out=180,in=0] (2,-0.5)
  to [out=180,in=0] (1,-1)
  to [out=180,in=270] (0,0);
\draw (1,0.3) to [out=240,in=120] (1,-0.3);
\draw (1,0.3) to [out=60,in=50] (1.1,0.35);
\draw (1,-0.3) to [out=300,in=310] (1.1,-0.35);
\draw (1,0.3) to [out=310,in=40] (1,-0.3);
\draw (3,0.3) to [out=240,in=120] (3,-0.3);
\draw (3,0.3) to [out=60,in=50] (3.1,0.35);
\draw (3,-0.3) to [out=300,in=310] (3.1,-0.35);
\draw (3,0.3) to [out=310,in=40] (3,-0.3);
\draw[->] (4.5,0) -- (5.5,0);
\draw[thick] (6,0) to [out=90,in=180] (7,1)
  to [out=0,in=180] (8,0.5)
  to [out=0,in=180] (9,1)
  to [out=0,in=90] (10,0)
  to [out=270,in=0] (9,-1)
  to [out=180,in=0] (8,-0.5)
  to [out=180,in=30] (7,-1)
  to [out=150,in=270] (6,0);
\draw (7,-1) to [out=130,in=240] (6.7,-0.2)
  to [out=60,in=240] (6.8,0);
\draw (7,-1) to [out=50,in=300] (7.3,-0.2)
  to [out=120,in=300] (7.2,0);
\draw (6.7,-0.2) to [out=20,in=160] (7.3,-0.2);
\draw (9,0.3) to [out=240,in=120] (9,-0.3);
\draw (9,0.3) to [out=60,in=50] (9.1,0.35);
\draw (9,-0.3) to [out=300,in=310] (9.1,-0.35);
\draw (9,0.3) to [out=310,in=40] (9,-0.3);
\draw[->] (10.5,0) -- (11.5,0);
\draw[thick] (12,0) to [out=90,in=180] (13,1)
  to [out=0,in=180] (14,0.5)
  to [out=0,in=180] (15,1)
  to [out=0,in=90] (16,0)
  to [out=270,in=30] (15,-1)
  to [out=150,in=0] (14,-0.5)
  to [out=180,in=30] (13,-1)
  to [out=150,in=270] (12,0);
\draw (13,-1) to [out=130,in=240] (12.7,-0.2)
  to [out=60,in=240] (12.8,0);
\draw (13,-1) to [out=50,in=300] (13.3,-0.2)
  to [out=120,in=300] (13.2,0);
\draw (12.7,-0.2) to [out=20,in=160] (13.3,-0.2);
\draw (15,-1) to [out=130,in=240] (14.7,-0.2)
  to [out=60,in=240] (14.8,0);
\draw (15,-1) to [out=50,in=300] (15.3,-0.2)
  to [out=120,in=300] (15.2,0);
\draw (14.7,-0.2) to [out=20,in=160] (15.3,-0.2);
\end{tikzpicture}
\end{center}
Weight one Hodge structures are discussed in greater detail in Examples \ref{eg:curve1} and \ref{eg:curve2}.
\end{example}

\subsection{Relations} \label{iS:R}

We may define partial orders on $\bfD$ and $\bfN$ by ``inclusion in closure.'' That is, given $\cN \in \bfN$, let $\overline\cN$ denote the analytic closure of $\cN$ in $\tNilp(\fg_\bR)$; likewise, given $\cO \in \bfD$, let $\overline \cO$ denote the analytic closure of $\cO$ in $\check D$.  Given $\cN_i \in \bfN$, we write
\[
  \cN_1 \ \le \ \cN_2 \quad\hbox{if}\quad \cN_1 \subset \overline\cN_2  \,.
\]
We give $\bfD$ the ``opposite'' partial order: given $\cO_i \in \mathbf{\Delta}$, we write  
\[
  \cO_1 \ \le \ \cO_2 \quad\hbox{if}\quad \cO_2 \subset \overline\cO_1 \,.
\]
(This is the choice that will be compatible with inclusions of nilpotent cones.)  As subsets, both $\bfD_D \subset \bfD$ and $\bfN_D \subset \bfN$ inherit partial orders. 

In \S\ref{S:po} we define a relation (also denoted $\le$) on $\Psi_D$.  In general transitivity fails for this relation, so that it is not a partial order (Examples \ref{eg:CY0} and \ref{eg:CY2}).  Nonetheless, the maps of \eqref{E:diagram2} preserve the relations.  

\begin{theorem} \label{T:1}
The surjections $\phi_\infty : \Psi_D \to \bfD_D$ and $\pi : \Psi_D \to \bfN_D$ preserve the relations $\le$.
\end{theorem}

\noindent The theorem is proved in \S\ref{S:T1}.


\subsection{Polarized relations} \label{iS:PR}

In \S\ref{S:nc+pr} we introduce the notion of a ``polarized relation''; these are the relations in the partial orders that may be realized Hodge theoretically.  Geometrically, a polarized relation $\preceq$ arises as follows (for example):  consider a variation of Hodge structure $\Phi : \Delta^* \times \Delta^* \to \Gamma \backslash D$ defined over a product of punctured discs.  Then Schmid's Nilpotent Orbit Theorem associates to the limits $\lim_{z\to 0} \Phi(w,z)$ and $\lim_{w,z\to 0} \Phi(w,z)$ two conjugacy classes $[F_1,N_1] , [F_2,N_2] \in \Psi_D$ with polarized relation $[F_1,N_1] \preceq [F_2,N_2]$.  More generally, suppose that $\s$ is a nilpotent cone underlying a (possibly several variable) nilpotent orbit $e^{\bC \sigma}F$.  Let $\Gamma_\s$ denote the faces of $\s$, and define a partial order on $\Gamma_\s$ by declaring $\s_1 \le \s_2$ if $\s_1 \subset \overline\s_2$.  Then we construct a commutative diagram associated to the nilpotent orbit:
\begin{equation} \label{E:cd}
\begin{tikzcd}
   & \Gamma_\s 
   \arrow[d,"\psi^\circ"] \arrow[dl,"\pi^\circ"'] \arrow[dr,"\phi_\infty^\circ"]
   & \\
   \bfN_D 
   &
   \Psi_D \arrow[l, two heads, "\pi"'] \arrow[r, two heads, "\phi_\infty"]
   &
   \bfD_D \,.
\end{tikzcd}
\end{equation}

\begin{theorem} \label{T:2}
The map $\psi^\circ: \Gamma_\s \to \Psi_D$ preserves the relations $\le$.
\end{theorem}

\noindent Theorem \ref{T:2} is proved in \S\ref{S:T2}.  From Theorems \ref{T:1} and \ref{T:2}, and the commutativity of \eqref{E:cd}, we obtain

\begin{corollary} \label{C:3}
The maps $\phi_\infty^\circ : \Gamma_\s \to \bfD_D$ and $\pi^\circ : \Gamma_\s \to  \bfN_D$ are morphisms of posets.
\end{corollary}

\noindent We note that, in general, the full nilpotent orbit is needed to define $\phi_{\infty}^{\circ}$ and $\psi^{\circ}$: the cone $\sigma$ by itself gives only $\pi^{\circ}$.  (The issue is that the associated \emph{boundary component} $\tilde{B}(\sigma)$ may have multiple connected components; see $\S$\ref{S:nc+pr} and also Example \ref{eg:borel7}.)

\begin{definition}
A relation in any one of $\Psi_D$, $\bfD_D$ or $\bfN_D$ is \emph{polarized} if it is the image of a relation on $\Gamma_\s$.   
\end{definition}

The key computational tool used to identify polarizable relations is Theorem \ref{T:4}: any polarized relation $\preceq$ may be realized by commuting horizontal $\tSL(2)$'s.  This result relies, in turn, on the multivariable $\mathrm{SL}(2)$-orbit theorem.  This theorem is proved by Cattani, Kaplan and Schmid \cite{MR840721} in the case that $D$ is a period domain.  We extend their result to arbitrary Mumford-Tate domains in \S\ref{S:CKSMTD}.

\begin{remark}
The polarized relations $\preceq$ form a partial order only in very special cases; in general, transitivity fails; see Examples \ref{eg:CY0} and  \ref{eg:CY2}, and Remark \ref{R:fd}.  The special cases include: (i) Hermitian symmetric $D$ (Examples \ref{eg:curve2} and \ref{eg:K32}, and Theorem \ref{T:hs}), (ii) period domains with contact horizontal distribution (Example \ref{eg:H3}).  In both cases all relations are polarized.
\end{remark}

\subsubsection{Period domains} 
In \S\ref{S:pd} we consider the case that $D$ is a period domain.  The main result here is a simple, combinatorial characterization of polarized relations in terms of the possible Hodge substructures on the primitive cohomology (Theorem \ref{T:PDpr}).  A number of examples are worked out here.  (In this section only, we work modulo the full automorphism group $\mathrm{Aut}(V_{\bR},Q)$, rather than the connected component.  Of course, this only makes a difference for even weight.) 

\subsubsection{The classical case} 
In \S\ref{S:HS} we study the ``classical case'' that $D$ is Hermitian symmetric.  (This includes Example \ref{eg:curve0} above.)  

\begin{theorem} \label{T:hs}
If $D$ is Hermitian symmetric, then \emph{(i)} the relation $\le$ on $\Psi_D$ is a linear order, and \emph{(ii)} all relations are polarized.
\end{theorem}

\noindent The theorem is proved in \S\ref{S:hs}.  Moreover, in this case there exists a single nilpotent cone $\s$ that realizes every polarizable relation on $\Psi_D$ (Remark \ref{R:1cone}).

\subsubsection{The case of minimal isotropy} 

In \S\ref{S:fd} we turn to the case that the isotropy subgroup $G^0(\bR)$ of $D$ is a torus.  As discussed there, this case may be viewed as ``maximally nonclassical.''  

\begin{theorem}
Suppose that $G^0(\bR)$ of $D$ is a torus.  Then $\Psi_D$ is indexed by the subsets of the simple roots $\sS$ of $\fg_\bC$, and the relation $[F_1,N_1] \leq [F_2,N_2] \in \Psi_D$ holds if and only if $\sS_1 \subseteq \sS_2$.  In particular, $\leq$ yields a partial order on $\Psi_D$.  Moreover, a relation is polarized if and only if the corresponding subsets $\sS_1 \subset \sS_2 \subset \sS$ have the property that the elements of $\sS_1$ are strongly orthogonal to the elements of $\sS_2\backslash\sS_1$.  Accordingly, $\preceq$ is also a partial order on $\Psi_D$.
\end{theorem}

\noindent The theorem is proved in \S\ref{S:fd} (cf.~Proposition \ref{P:fd} and Corollary \ref{C:fd}).
\subsection{Examples} 

\begin{example}[Period domain for $\bh = (2,m,2)$] \label{eg:H0}
An interesting \emph{nonclassical} case that has much in common with the classical case is the period domain parameterizing weight two polarized Hodge structures with Hodge numbers $\bh = (2,m,2)$.  This period domain is ``nearly classical'' in the sense that the horizontal subbundle $T^hD \subset TD$ has corank 1 (while $T^h D = TD$ holds in the classical case).  In this case the relations on $\bar{\Psi}_D$ define a partial (but nonlinear) order, and the maps of \eqref{E:diagram3} are isomorphisms of posets.  Moreover the relations are all polarizable.  Each set consists of six elements, which we enumerate $0,\mathrm{I},\mathrm{II},\ldots,\mathrm{V}$, where ``$0$'' corresponds to $D$ (i.e. pure HS) and ``$\mathrm{V}$'' to Hodge-Tate LMHS.  The partial order may be visualized as below, where each arrow ``$\to$'' (suggesting degeneration) represents a generating relation ``$<$'' (suggesting inclusions of cones):
\begin{equation} \label{E:H2diag}
\begin{tikzpicture}[baseline=(current  bounding  box.center)]
  \node (0) at (0,0) {0};
  \node (I) at (1,0) {I};
  \node (II) at (2,1) {II};
  \node (III) at (2,-1) {III};
  \node (IV) at (3,0) {IV};
  \node (V) at (4.5,0) {V};
  \draw[->] (0) -- (I);
  \draw[->] (I) -- (II);
  \draw[->] (I) -- (III);
  \draw[->] (II) -- (IV);
  \draw[->] (III) -- (IV);
  \draw[->] (IV) -- (V);
\end{tikzpicture}
\end{equation}
So, for example, here we have $\rmI < \rmII$ and $\rmII < \rmIV$ (so that $\rmI < \rmIV$ by transitivity), but there is no relation between $\rmII$ and $\rmIII$.  (These period domains parameterize Hodge structures of Horikawa surfaces.  Green, Griffiths and Laza \cite{GGL} have identified geometric degenerations realizing each of the arrows in \eqref{E:H2diag}.  As in Example \ref{eg:curve0}, the algebraic varieties become successively more singular as we move to the right.)  To juxtapose with the classical case:
\begin{a_list}
\item[(a${}^\prime$)] The partial order is nonlinear (Example \ref{eg:H1}).
\item[(b${}^\prime$)] All the relations are polarized (Example \ref{eg:H2}).
\item[(c${}^\prime$)] There does exist a single nilpotent cone with the property that every polarized relation on $\overline{\Psi}_D$ is realized by some face of $\s$ \cite[\S5.3]{BPR}.
\end{a_list}
This example is discussed in greater detail in Examples \ref{eg:H1}, \ref{eg:H2} and \ref{eg:H3}. 
\end{example}

In general, the structure of the polarizable relations on $\Psi_D$, $\bfD_D$ and $\bfN_D$ is not a simple as Examples \ref{eg:curve0} and \ref{eg:H0} may suggest.  The following example (which is a special case of Examples \ref{eg:CY1} and \ref{eg:CY2}) hints at the more complicated structures that may arise.

\begin{example}[Period domain for $\bh = (1,2,2,1)$] \label{eg:CY0}
As in the two examples above the maps of \eqref{E:diagram3} are bijections.  However, as we will discuss below, the relation on $\Psi_D$ is not a partial order.  The set $\Psi_D$ consists of eight elements, which we denote 
\[
  \Psi_D \ = \ \{ \rmI_0 \,,\, \rmI_1 \,,\, \rmI_2 \,,\, \rmII_0 \,,\, 
  \rmII_1 \,,\, \rmIII_0 \,,\, \rmIV_1 \,,\, \rmIV_2 \}
\]
in order to be consistent with the notation of Examples \ref{eg:CY1} and \ref{eg:CY2}.  The polarized relations $\prec$ on $\Psi_D$ are indicated below by the arrows $\to$.
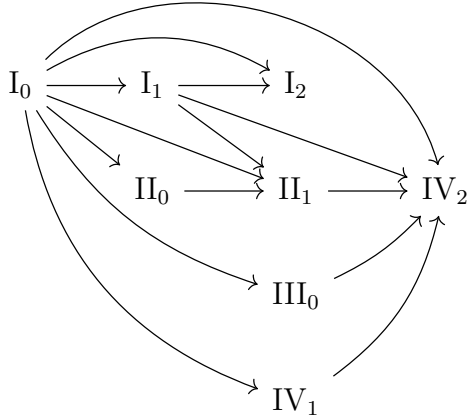
\begin{center}
\begin{tikzcd}
  \rmI_0 \arrow[r] \arrow[rr, bend left]
  \arrow[rd] \arrow[rrd]
  \arrow[rrrd, bend left=60]
  \arrow[rrdd, bend right=20] \arrow[rrddd, bend right]
  & \rmI_1 \arrow[r] \arrow[rd] \arrow[rrd]
  & \rmI_2 & \\
  & \rmII_0 \arrow[r]
  & \rmII_1 \arrow[r]
  & \rmIV_2 \\
  &
  & \rmIII_0 \arrow[ru , bend right=10]
  & \\
  &
  & \rmIV_1 \arrow[ruu, bend right=20]
  &
\end{tikzcd}
\end{center}
Notice that the polarized relations are not transitive (and so fail to constitute a partial order):  $\rmII_0 \prec \rmII_1 \prec \rmIV_2$, but $\rmII_0 \not\prec \rmIV_2$.  The remaining (unpolarized) relations are 
\begin{eqnarray*}
  \rmI_1 & < & \rmIV_1 \\
  \rmI_2 & < & \rmIII_0 \,,\ \rmIV_2 \\
  \rmII_0 & < & \rmIII_0 \,,\ \rmIV_1 \,,\ \rmIV_2 \,. 
\end{eqnarray*}
Note that the relation $<$ is not transitive: $\rmI_1 < \rmI_2$ and $\rmI_2 < \rmIII_0$, but $\rmI_1 \not < \rmIII_0$.
\end{example}

\subsection{Secondary poset} 
The \emph{secondary poset} $\tilde{\Psi}_D^{\text{pol}}$ of equivalence classes of multivariable nilpotent orbits is constructed in \S\ref{S:cubes}.  This, finally, is the object promised in $\S$\ref{S:1.1}, which classifies how admissible degeneracies may be assembled into cones.

We first define a partially ordered set $\tilde{\Psi}_D$ whose elements are morphisms (for any $r\in \mathbb{N}$) from the power sets $(\mathcal{P}\{1,\ldots,r\},\subseteq )$ to $(\Psi_D,\preceq)$  satisfying certain root-theoretic admissibility conditions, and which are ordered by an obvious notion of inclusion.  These elements are called \emph{admissible $n$-cubes}.  We then define two subposets $\tilde{\Psi}_D\supseteq \tilde{\Psi}^{\text{str}}_D \supseteq \tilde{\Psi}^{\text{pol}}_D$, with $\tilde{\Psi}^{\text{pol}}_D$ indexing the ``types'' of multivariable nilpotent orbits that really do occur.  (The reason for defining $\tilde{\Psi}_D$ at all is that it is straightforward to compute, whereas the two refinements are not.)  In \S\ref{S:G2} we compute these posets in the case where $G(\bR)$ is the simple, noncompact, exceptional real Lie group $G_2$ of rank two, the interesting case being that with Hodge numbers $(2,3,2)$. 

Finally, in \S\ref{S:matt} we describe how mirror symmetry can be used in some special cases to check that a given admissible $n$-cube in $\tilde{\Psi}_D$ belongs to $\tilde{\Psi}_D^{\text{pol}}$. Much to our surprise (and great interest), since the initial posting of this article, the classification of admissible 2-cubes for Calabi-Yau variations has been put to use in work on quantum gravity \cite{GPV,GLP}.

\subsection{Technical remarks} \label{S:tech}

\begin{i_list}
\item A given connected Mumford-Tate domain arises from a Hodge representation (of $G$) on a vector space $V$.  In the paper, we frequently pass to the adjoint representation (on $\fg$), which factors through $G^{\text{ad}}$. However, this affects neither the (connected) M-T domain nor its boundary components, cf. \cite[\S 1]{KP2012}.
\item Given a Mumford-Tate (algebraic) group $G$ and field $K\supseteq\bQ$, we write $G_K$ for the base-change (an algebraic group) and $G(K)$ for the group of $K$-valued points, which we interpret as a Lie group when $K=\bR$ or $\bC$.  While $G$ is always connected (i.e. irreducible), $G(\bR)$ need not be, as in the case of $\mathrm{SO}(p,q)$. On the other hand, if $G^0\leq G_{\bR}$ is the isotropy group of a HS $\varphi$, then $G^0(\bR)$ is always connected as a Lie group. (Being the centralizer of a torus in $G$, $G^0$ connected as an algebraic group. As $G^0(\bR)$ is compact, each of its elements is semisimple, hence -- by the algebraic connectedness -- contained in some real torus $T\leq G^0(\bR)$, which being compact is an $(S^1)^r$.)  
\item On the other hand, in \S\ref{S:pd}, for even-weight period domains, we use a group $G=\mathrm{O}(V,Q)$ which (with two irreducible components) is not even algebraically connected.  In this situation, the Lie group $G(\bR)\cong \mathrm{O}(p,q)$ has four components, and $G^0(\bR)$ has two.  Moreover, the domain $\tilde{D} = G(\bR).\varphi = D \amalg D'$ has two components.  The reason why the resulting ($G(\bR)$-)equivalence-classes are quotients of those for $D$, is that $\mathrm{SO}(V,Q)/\mathrm{SO}(V,Q)^+$ already gives identifications between $\Psi_D$ and $\Psi_{D'}$, etc.
\item We will make frequent use of an identification $\Psi_D \cong \mathcal{L}_{\varphi,\ft}/\mathcal{W}^0$ (cf. \eqref{E:FNvl}), which is stated and proved in \cite[Thm. 5.5]{SL2} under the assumption that the horizontal distribution on $D$ is bracket-generating. That this assumption is unnecessary may be seen at once in light of \cite[Prop. 3.10]{MR3217458}, which yields a (unique) subdomain $\mathsf{D}=\mathsf{G}(\bR)^+ .\varphi$ through $\varphi\in D$ which is maximal for the bracket-generating property and contains every horizontal $\mathrm{SL}(2)$ through $\varphi$. By \cite{SL2}, we have $\Psi_{\mathsf{D}} \cong \mathcal{L}_{\varphi,\ft}/\mathsf{W}^0$, where $\mathsf{W}^0$ is the Weyl group of $\mathsf{G}^0(\bR)$ and $\mathcal{L}$ is the same as for $D$.  Quotienting both sides by the (larger) Weyl group of $G^0(\bR)$ then yields the identification in the general case. 
\end{i_list}
The proofs in the paper make some use of representation theory; the necessary background is reviewed in the appendix.

\subsection*{Acknowledgements} 
Over the course of this work we have benefitted from conversations and correspondence with several colleagues; we would especially like to thank Eduardo Cattani, Mark Green, Phillip Griffiths, William McGovern, Thomas Grimm and Radu Laza.  We also thank the referee for comments which have helped us to improve the exposition.

\section{The relation on $\Psi_D$}  \label{S:po}

Recall that $\Psi_D$ is the finite set consisting of $G(\bR)^+$-conjugacy classes of pairs $(F,N)\in \check{D}\times \mathrm{Nilp}(\fg_{\bR})$ defining $\bR$-split PMHSs (or equivalently, of horizontal $\tSL(2)$-orbits). 
In this section, we shall use inclusions of Levi subalgebras in $\fg_{\bR}$ to define the relation alluded to in \S\ref{iS:R}, and prove Theorem \ref{T:1} establishing the compatibility of the relation with the partial orders on $\mathbf{N}_D$ and $\mathbf{\Delta}_D$.

\subsection{Parameterization of $\Psi_D$}  \label{S:Psi}

To define the relation on $\Psi_D$ we must first summarize the characterization of $\Psi_D$ given by Robles in \cite{SL2}.  The description is representation theoretic and involves the notions of Weyl groups, Levi subalgebras and distinguished grading elements; the reader wishing to review these notions will find definitions and some discussion in the appendix.

Fix a Hodge structure $\varphi \in D$.  From this point on, we will 
\begin{center}
\emph{assume that $G^0(\bR)$ is the stabilizer of $\varphi$ in $G(\bR)^+$.}
\end{center}
Let 
\begin{equation}\label{E:gphi}
  \fg_\bC \ = \ \bigoplus_{p \in\bZ} \fg^p_\varphi
\end{equation}
denote the induced weight zero Hodge decomposition.\footnote{Traditionally, $\fg^p_\varphi$ is denoted $\fg^{p,-p}_\varphi$.}  Recall that 
\begin{equation}\label{E:lb}
  [\fg^p_\varphi,\fg^q_\varphi] \ \subset \ \fg^{p+q}_\varphi \,.
\end{equation}
Moreover, 
\[
  \fg^0_{\varphi,\bR} \ := \ \fg^0_\varphi \,\cap\,\fg_\bR
\]
is a compact real form of $\fg^0_\varphi$, and the Lie algebra of $G^0(\bR)$.

Let $\ttE_\varphi' \in \tEnd(\fg_\bC)$ be the endomorphism acting on $\fg^p_\varphi$ by $p \one$.  Then \eqref{E:lb} implies that $\ttE_\varphi'$ is a derivation of $\fg_\bC$.  Since $\fg_\bC$ is semisimple, there exists a semisimple element $\ttE_\varphi \in \fg_\bC$ such that $\ttE_\varphi' = \tad(\ttE_\varphi)$.  Note that $\ttE_\varphi$ is a grading element (\S\ref{S:ge}).  Moreover, $\ttE_\varphi \in \bi \fg_\bR$ \cite[\S2.3]{MR3217458}.

Fix a compact Cartan subalgebra $\ft \subset \fg_\bR$ containing $\bi \ttE_\varphi$.  We have
\begin{equation} \label{E:tg00}
  \ft \ \subset \ \fg^0_{\varphi,\bR}\,.
\end{equation}
Given a Levi subalgebra $\fl_\bR \subset \fg_\bR$, recall the projection $\pi^\tss_\fl : \fl_\bC \to \fl_\bC^\tss$ onto the semisimple factor, \cf\eqref{SE:proj}.  Define\footnote{See the choice of conventions in \S\ref{S:dge}.}
\begin{equation} \label{E:dfncL}
  \cL_{\varphi,\ft} \ := \ 
  \left\{ \begin{array}{c}
  \hbox{$\varphi$--stable Levi subalgebras $\fl_\bR \subset \fg_\bR$ such
        that $\ft \subset \fl_\bR$ and}\\
  \hbox{$2\,\pi^\tss_\fl(\ttE_\varphi)$ is a distinguished 
        grading element of $\fl_\bC^\tss$}
  \end{array}\right\} \,.
\end{equation}
Note that 
\begin{equation} \label{E:Einl}
  \ttE_\varphi \ \in \ \fl_\bC \quad \hbox{for all} \quad \fl_\bR \in \cL_{\varphi,\ft}\, ,
\end{equation}
and that $\pi^\tss_\fl(\ttE_\varphi)$ is always a grading element of $\fl_\bC^\tss$.  Moreover, $\ft$ is always an element of $\cL_{\varphi,\ft}$ (cf. \ref{S:dge}).
Let $\sW^0$ denote the Weyl group of $\fg^{0}_\varphi$.  Then $\sW^0$ acts on $\cL_{\varphi,\ft}$ and 
\begin{equation}\label{E:FNvl}
  \Psi_D \ \simeq \ \Lambda_{\varphi,\ft} \ := \ \cL_{\varphi,\ft}/\sW^0 \,.
\end{equation}

Given $[\fl_\bR]\in \Lambda_{\varphi,\ft}$, the corresponding $[F,N] \in \Psi_D$ is described as follows.  (See \cite{SL2} for details.)  The Cartan subalgebra $\ft \subset \fg_\bR$ determines a \emph{Cartan decomposition} $\fg_\bR = \fk_\bR \op \fk_\bR^\perp$ with $\fk_\bR \supset \ft$ a maximal compact subalgebra.  In fact, 
\begin{equation}\label{E:k}
  \fk_\bC \ = \ \op\,\fg^{2p}_\varphi \tand
  \fk_\bC^\perp \ = \ \op\,\fg^{2p+1}_\varphi\,.
\end{equation}

\begin{definition}
A \emph{\Dokovic--Kostant--Sekiguchi triple\label{p:DKS}} (DKS--triple) is a standard triple (\S\ref{S:st}) $\{\overline\sE , \sZ , \sE\} \subset \fg_\bC$ such that $\overline\sZ = -\sZ \in \fk_\bC$ and $\overline\sE,\sE \in \fk^\perp_\bC$.
\end{definition}

\begin{lemma}[{\cite{SL2}}] \label{L:DKS}
Given $\fl_\bR \in \cL_{\varphi,\ft}$, there exists a DKS--triple $ \{ \overline\sE , \sZ , \sE\} \subset \fl_\bC^\tss$ with neutral element \begin{equation} \label{E:sZ0} \sZ = 2\,\pi^\tss_\fl(\ttE_\varphi) \in \bi \ft \subset \fg^0_\varphi \end{equation} and $\sE \in \fg^{-1}_\varphi$. 
\end{lemma}

\noindent Given a DKS--triple as in Lemma \ref{L:DKS}, set 
\begin{equation}\label{E:rho}
  \varrho \ := \ \exp\,\bi\tfrac{\pi}{4}(\sE + \overline\sE) \ \in \ L_\bC \,.
\end{equation}
Then the conjugacy class of $\Psi_D$ associated with $[\fl_\bR]\in \Lambda_{\varphi,\ft}$ by \eqref{E:FNvl} is represented by 
\begin{equation}\label{E:FN}
  (F,N) \ = \ \varrho^{-1}\cdot (\varphi,\sE) \ \in \ \tilde B_\bR(D) \,.
\end{equation}
Moreover, both $F = \varrho^{-1}\cdot\varphi$ and 
\begin{equation}\label{E:phiinf-ct}
  \phi_\infty(F,N) \ = \ \varrho\cdot\varphi \ \in \ \check D
\end{equation}
lie in the \emph{same} $G(\bR)^+$--orbit $\cO \in \bfD_D$.

Observe that 
\[
  [F_{\varphi},0] \ \in \ \Psi_D
\]  
is a well-defined element; we call this the \emph{trivial element}.  Note that it corresponds to the Cartan:
\begin{equation}\label{E:trive}
  [F_{\varphi},0] \quad \longleftrightarrow \quad [\fl_\bR] \,=\, [\ft] \, .
\end{equation}

\subsubsection{The diagonal Levi} \label{S:diagL}

There is a second Levi subalgebra $\tilde\fl_\bR \supset \fl_\bR$ that will be used to define the relation on $\Psi_D$.  The $\bR$--split PMHS $(F,W(N))$ induces a Deligne bigrading $\fg_\bC = \op \fg^{p,q}_{(F,N)}$.  To be precise, setting $\ttE' = \tAd_{\varrho}^{-1}(\ttE_\varphi)$ and $Y = \tAd_\varrho^{-1}(\sZ)$, we have
\begin{equation}\label{E:gpq}
  \fg^{p,q} \ = \ \left\{ \xi \in \fg_\bC \ 
  \left| \ [\ttE',\xi] = p\,\xi\,,\ [Y,\xi] = (p+q)\xi \right.\right\}\,,
\end{equation}
\cf\cite[(5.12)]{SL2}.  The Levi subalgebra $\fl_\bC$ is contained in the ``diagonal'' Levi subalgebra
\begin{equation}\label{E:diag}
  \tilde\fl_\bC \ := \ \op \, \fg^{p,p}_{(F,N)} \,,
\end{equation}
which is also defined over $\bR$.  Moreover, since \[ \tAd_{\rho^{-1}}[2\ttE_{\varphi}-\sZ,\tAd_{\rho} (\xi)] = [2\ttE' - Y,\xi] \] and $\xi\in\fl_{\bC} \iff \tAd_{\rho}(\xi)\in\fl_{\bC}\, ,$
\begin{equation}\label{E:tL}
  \tilde\fl_\bC \ = \ \left\{ \xi \in \fg_\bC \ | \ 
  [  2\ttE_\varphi-\sZ \,,\, \xi] = 0 \right\} \,. 
\end{equation}
That is, $2\ttE_\varphi-\sZ$ 
is an element of the centralizer of $\tilde \fl_\bC$ in $\fg_\bC$.  It is a general property of Levi subalgebras that they contain their centralizers.  That is, 
\begin{subequations}\label{SE:sZ}
\begin{equation}\label{E:sZ1}
  2\ttE_\varphi\,-\,\sZ \ \in \ \tilde \fz\,,
\end{equation}
where $\tilde\fz$ denotes the center of $\tilde\fl_\bC$.  Since $\sZ \subset \fl^\tss_\bC \subset \tilde\fl^\tss_\bC$, it follows that 
\begin{equation} \label{E:sZ2}
  \sZ \ = \ 2\pi^\tss_{\tilde\fl}(\ttE_\varphi) \,. 
\end{equation}
\end{subequations}

\subsubsection{Sub-Hodge structures} \label{S:subH}

This is a convenient point to record two remarks on the induced Hodge structure on $\fl$ that will be used in subsequent proofs.  (Identical remarks hold for the diagonal Levi $\tilde \fl$ of \eqref{E:diag}.)  See \cite[\S3.1.3 \& \S4.2]{SL2} and \cite[\S V.E]{GGR} for proofs and further discussion.  Let $L \subset G$ be the algebraic subgroup with Lie subalgebra $\fl \subset \fg$.  \smallskip

\textsc{(a)}
Note that $\varphi$ induces a real sub-Hodge structure on $\fl_\bR$, since $\ttE\in\fl_{\bR}$ stabilizes $\fl_{\bR}$.  The Hodge decomposition $\fl_\bC = \op\,\fl_\varphi^{p,-p}$ is given by $\fl_\varphi^{p,-p} = \fl_\bC \cap \fg_\varphi^{p,-p}$.  Its $L(\bR)^+$-orbit $D_{\fl}$ may be identified with the subdomain $L(\bR)^+\cdot\varphi\subset D$, with compact dual $\check D_\fl = L(\bC) \cdot F_{\varphi} = L(\bC) \cdot F \subset \check D$ (with $F$ as in \eqref{E:FN}).
 
The semisimple factor $\fl^\tss = [\fl,\fl]$ likewise inherits a real sub-Hodge structure (with $L^{\tss}(\bR)^+$-orbit $D_{\fl^{\tss}}$), as it is stabilized by $\fl_{\bR}\, (\ni\ttE_{\varphi})$.  Since $L^\tss(\bR)^+\cdot\varphi = L(\bR)^+\cdot\varphi$, we may identify $D_{\fl^\tss}$ with $D_\fl$ (as complex manifolds, but not as homogeneous manifolds). \smallskip

\textsc{(b)}
More generally, if $(F,W(N))$ is polarized mixed Hodge structure on $\fg_\bR$ with $N\in\fl_\bR$, and $F\in\check{D}_{\fl}$, then $( \fl_\bC \cap F \,,\, \fl \cap W(N))$ is a polarized mixed Hodge structure on $\fl_\bR$, which is $\bR$--split if $(F,W(N))$ is.  

As a nilpotent endomorphism, $N$ is necessarily contained in the semisimple factor $\fl^\tss_\bR$, and the mixed Hodge representation necessarily stabilizes $\fl^{\tss}_{\bC}\subset \fl_{\bC}$ (because $L(\bC)$ does).  So we obtain a polarized mixed Hodge structure on $\fl_\bR^\tss$.  In particular, if $[F,N] \in \Psi_D$, and we set $F' = \fl_\bC^\tss \cap F$, then $[F' , N] \in \Psi_{D_{\fl^\tss}}$.

Let $\sW_\fl$ denote the Weyl group of $\fl^\tss_\bC$.  Since $\fl_\bC$ is a Levi subgroup of $\fg_\bC$, we have $\sW_\fl \subset \sW$.  Set $\sW^0_\fl = \sW_\fl \cap \sW^0$.  Likewise, set $\ft' = \fl_\bR^\tss \cap \ft$.  Applying the characterization above we see that $\Psi_{D_{\fl^\tss}} \simeq \cL_{\left.\varphi\right|_\fl,\ft'}/\sW^0_\fl$.

\subsection{The relation on $\Psi_D$}  \label{S:po-Psi}

Given $[\fl]\in\Psi_D$, recall the diagonal Levi subalgebra $\tilde\fl$ of \eqref{E:diag}.  

\begin{definition} \label{dfn:PsiPO}
Write $[\fl_1] \le [\fl_2 ]$ if $\fl_1 \subset w(\tilde\fl_2)$ for some $w\in \sW^0$.
\end{definition}

\begin{remark} \label{R:trivr}
Recall the trivial element $[F_{\varphi},0]\in\Psi_D$ of \eqref{E:trive}.  It follows directly from \eqref{E:dfncL}, \eqref{E:trive} and Definition \ref{dfn:PsiPO} that 
\[
  [F_{\varphi},0] \ \le \ [F,N] \quad \hbox{for all}\quad [F,N] \ \in \ \Psi_D \,.
\]
We call these the \emph{trivial relations}.
\end{remark}



\subsection{Proof of Theorem \ref{T:1}} \label{S:T1}

Suppose that $[F_1,N_1] \le [F_2,N_2]  \in \Psi_D$.  Then without loss of generality, we may assume that the representatives $\fl_i$ of $[\fl_i] \in \Lambda_{\varphi,\ft}$ were chosen so that 
\[
  \fl_1 \ \subset \ \tilde \fl_2 \,.
\]
Let $L_1 \subset L_2 \subset G$ be the associated connected algebraic subgroups with Lie algebras $\fl_1 \subset \tilde\fl_2$.  

We also assume that $(F_i,N_i)$ are given by \eqref{E:FN}.  Set 
\[
  \cN_i \ := \  \pi([F_i,N_i]) \ = \ \tAd(G(\bR)^+) \cdot N_i \ \in \ \bfN_D
\]
and 
\[
  \cO_i \ := \ \phi_\infty([F_i,N_i]) \ = \ G(\bR)^+\cdot F_i 
  \ \in \ \bfD_D \,.
\]
We want to show that $\cO_1 \le \cO_2$ and $\cN_1 \le \cN_2$.

\begin{proof}[Proof of $\cO_1 \le \cO_2$]

By \eqref{E:Einl}, $\ttE_\varphi \in \fl_1 \subset \tilde\fl_2$.  As discussed in \S\ref{S:subH}, the restrictions of $\varphi$ to $\fl_1$ and $\tilde\fl_2$, respectively, are Hodge structures.  Their respective orbits (by $L_1(\bR)^+$ and $L_2(\bR)^+$) are naturally identified with 
\[
  D_1 \ := \ L_1(\bR)^+ \cdot \varphi \tand 
  D_2 \ := \ L_2(\bR)^+ \cdot \varphi \,.
\]
Note that 
\[
  F_1 \ \in \ \overline {D_1} \tand F_2 \ \in \ \overline{D_2} \,.
\]
Moreover, $L_1 \subset L_2 \subset G$ implies $D_1 \subset D_2$, so that 
\begin{equation}\label{E:F1}
  F_1 \ \in \ \overline{D_2} \,.
\end{equation}

It is clear from the definition \eqref{E:diag} of $\tilde\fl$ that the polarized mixed Hodge structure $( \tilde\fl_2\cap F_2 , \left.N_2\right|_{\tilde\fl_2})$ is Hodge--Tate.  It follows that the $L_2(\bR)^+$--orbit
\[
  C_2 \ := \ L_2(\bR)^+\cdot F_2 
\]
polarized by the mixed Hodge structure is the unique closed $L_2(\bR)^+$--orbit in the compact dual $\check D_2 = L_2(\bC)\cdot \varphi$ of $D_2$ (\cf\cite[Corollary 4.3]{MR3331177}), hence contained in the closure of all $L_2(\bR)^+$-orbits \cite{MR0251246}.  Then \eqref{E:F1} implies
\[
  C_2 \ \subset \ \overline{L_2(\bR)^+\cdot F_1}\,. 
\]
Whence 
\begin{eqnarray*}
  \cO_2 & = & G(\bR)^+\cdot F_2 \ = \ G(\bR)^+\cdot L_2(\bR)^+\cdot F_2 \\
  & = & G(\bR)^+ \cdot C_2 \ \subset \ 
  G(\bR)^+\cdot \overline{L_2(\bR)^+\cdot F_1}\\
  &  \subset & \overline{G(\bR)^+\cdot L_2(\bR)^+\cdot F_1} 
  \ = \ \overline{G(\bR)^+\cdot F_1} \ = \ \overline\cO_1\,. 
\end{eqnarray*}
\end{proof}

\begin{proof}[Proof of $\cN_1 \le \cN_2$]
The \Dokovic--Kostant--Sekiguchi correspondence \cite[Section 2.6]{SL2} preserves the closure order on orbits \cite{MR1422847, MR1104427}.  So it suffices to show that 
\begin{equation} \label{E:ord0}
  K_\bC \cdot \sE_1 \ \subset \ \overline{K_\bC \cdot \sE_2} \,.
\end{equation}
Let $\tilde\fl_{2,\bC} = \op \tilde\fl^p_{2,\varphi}$ be the $\ttE_\varphi$--eigenspace decomposition.  From \eqref{E:gphi} we see that $\tilde\fl^p_{2,\varphi} = \tilde\fl_{2,\bC} \cap \fg^p_\varphi$.  Then \eqref{E:k} implies $\tilde\fl_{2,\varphi}^0 \subset \fk_\bC$, and Lemma \ref{L:DKS} implies $\sE_i \in \tilde\fl_{2,\varphi}^{-1}$.  Let $L_2^0(\bR) \subset L_2(\bR)^+$ be the connected Lie group with Lie algebra $\tilde\fl_{2,\varphi}^0$.  We claim that 
\begin{equation} \label{E:ord1}
  L_2^0(\bC) \cdot \sE_1 \ \subset \
  \overline{L_2^0(\bC) \cdot \sE_2}.
\end{equation}
This implies \eqref{E:ord0} and will complete the proof.  Note that the Jacobi identity implies $[\tilde\fl_2^0 , \tilde\fl_2^{-1}] \subset \tilde\fl_2^{-1}$.  So $L_2^0(\bC)$ preserves $\tilde\fl_2^{-1}$.  Therefore, to prove \eqref{E:ord1}, it suffices to show that $L_2^0(\bC) \cdot \sE_2$ is Zariski dense in $\tilde\fl^{-1}_2$.  

Let $W(\sE_2)$ denote the Jacobson--Morosov filtration of $\tilde\fl_{2,\bC}$ (\S\ref{S:JMf}).  Then \eqref{E:tL} and \eqref{E:W(N)} imply $\tilde\fl_{2,\bC} \cap \fg^{\leq 0}_\varphi = W_0(\sE_2)$.  This implies that $\sE_2 : \tilde\fl_2^0 \to \tilde\fl_2^{-1}$ is a surjection.  It follows that the rank of the map
\[
  L_2^0(\bC) \ \to \ \tilde\fl_2^{-1} \quad\hbox{sending}\quad
  g \ \mapsto \ g\cdot \sE_2
\]
is equal to the dimension of $\tilde\fl_2^{-1}$.  Whence $\overline{L_2^0(\bC) \cdot \sE_2} = \tilde\fl_2^{-1}$ and \eqref{E:ord1} follows.
\end{proof}

\subsection{Proof of Theorem \ref{T:0}} \label{S:T0}

Let $[F_1,N_1] , [F_2,N_2] \in \Psi_D$ and assume that 
\begin{equation}\label{E:2}
  \phi_\infty([F_1,N_1]) \ = \ \phi_\infty([F_2,N_2]) \,. 
\end{equation}
Recall that $\phi_\infty : \Psi_D \to \bfD_D$ is induced by the map $\tilde B_\bR(D) \to \check D$, also denoted $\phi_\infty$, of \eqref{E:diagram1}.  Since the latter map is $G(\bR)^+$--equivariant, and \eqref{E:2} holds, we may assume without loss of generality that 
\begin{equation}\label{E:3}
  F_\infty \ := \ \phi_\infty(F_1,N_1) \ = \ \phi_\infty(F_2,N_2) \,.
\end{equation}

Let $\fg_\bC = \op\,\fg^{p,q}_{(F_i,N_i)}$ denote the Deligne bigradings of the $\bR$--split PMHS $(F_i,W(N_i))$.  Recall that 
\[
  F^a_\infty \ = \ \bigoplus_{q \le -a} \fg^{p,q}_{(F_i,N_i)} \,,
\]
\cf~ the proof of \cite[(3.12)]{\CKSdeg}.  Define $\fg^{p,q}_{\infty,i}:= \fg^{-q,-p}_{(F_i,N_i)}$.  Then $\fg_\bC = \op\,\fg^{p,q}_{\infty,i}$ is the unique bigrading of $\fg_\bC$ associated to the pair $(F_\infty,\ft)$ by Kerr and Pearlstein \cite[Lemma 3.2]{MR3331177}.  Therefore, $\fg^{p,q}_{(F_1,N_1)} = \fg^{p,q}_{(F_2,N_2)}$, so that 
\[
  \tilde \fl_1 \ = \tilde \fl_2 \,.
\]
Then Lemma \ref{L:tilde} implies $[\fl_1] = [\fl_2]$, establishing injectivity.

\begin{lemma} \label{L:tilde}
Let $\fl_1 , \fl_2 \in \cL_{\varphi,\ft}$.  If $\tilde\fl_1 = \tilde \fl_2$, then $[\fl_1] = [\fl_2]$.
\end{lemma} 

\begin{proof}[Proof of Lemma \ref{L:tilde}]
Set 
\begin{equation}\label{E:tli}
  \tilde\fl \ := \ \tilde\fl_1 \ = \ \tilde\fl_2 \,.
\end{equation}
Let $L(\bR)^+ \subset G(\bR)^+$ be the Lie subgroup with Lie algebra $\tilde\fl$, and let $L^0(\bR) = L(\bR) \cap G^0(\bR)$.  

From \eqref{E:sZ0} and \eqref{E:sZ2} we see that $\sZ_1 = \sZ_2$.  Set $\sZ = \sZ_1 = \sZ_2$.  By Lemma \ref{L:DKS}, there exist two DKS--triples $\{\overline\sE_i , \sZ , \sE_i\} \subset \fl_{i,\bC}$, where $i=1,2$, containing $\sZ$ as the neutral element.  Set $\varrho_i = \texp\,\bi\frac{\pi}{4}(\overline\sE_i+\sE_i) \in L(\bC)$.  Recall \eqref{E:FNvl} that $(F_i , N_i) = \varrho_i^{-1}(\varphi,\sE_i)$ represents the conjugacy class $[F_i , N_i] \in \Psi_D$ corresponding to $[\fl_i] \in \Lambda_{\varphi,\ft}$.  We will prove the lemma by showing that 
\begin{equation} \label{E:6}
  [F_1 , N_1] \ = \ [F_2,N_2] \,.
\end{equation}

The Cayley transform of the DKS--triple $\{ \overline\sE_i , \sZ , \sE_i \}$ is the triple $\{ N^+_i , Y_i , N_i \} = \tAd_{\varrho_i}^{-1}\{ \overline\sE_i , \sZ , \sE_i\} \subset \fl_{i,\bR}$, see \cite[\S2.7]{SL2}.  Note that 
\[
  N^+_1 \,-\, N_1 \ = \ \bi\,\sZ \ = \ N^+_2 \,-\, N_2 \,.
\]
Rao's Theorem \cite[Theorem 9.4.6]{\CoMc} asserts that the two Cayley triples $\{ N^+_i , Y_i , N_i\}$ are conjugate under $L^0(\bR)$.  That is, $N_1 = \tAd_gN_2$ for some $g \in L^0(\bR)$.  Since we are working modulo the action of $G(\bR)^+$ we may assume that 
\begin{equation} \label{E:Ni}
  N_1 \ = \ N_2 \ =: \ N\,.
\end{equation}
Then Kostant's Theorem asserts that $Y_1$ and $Y_2$ are conjugate under an element of the centralizer $Z(N) = \{ g \in G(\bR)^+ \ | \ \tAd_g(N) = N \}$ of $N$, \cf\cite[Theorem 3.6]{MR0114875} or \cite[Theorem 3.4.10]{\CoMc}.  So, without loss of generality we may assume that 
\begin{equation} \label{E:Yi}
  Y_1 \ = \ Y_2 \ =: \ Y \,.
\end{equation}
The nilnegative and neutral elements, $N_i$ and $Y_i$, uniquely determine the nilpositive element $N_i^+$ of the standard triple.  Whence \eqref{E:Ni} and \eqref{E:Yi} imply
\[
  N_1^+ \ = \ N_2^+ \ =: \ N^+ \,.
\]
One may check that $\overline\sE_i + \sE_i = N_i^+ + N_i$ \cite[(2.27)]{SL2}.  Consequently $\varrho_1 = \varrho_2$ and $F_1 = F_2$.
\end{proof}

An easy consequence of Lemma \ref{L:tilde} (proof left to the reader) is the following:

\begin{corollary}
The relation $<$ on $\Psi_D$ satisfies the antisymmetry property:  that is, if $[\fl_1]\leq [\fl_2]$ and $[\fl_1]\geq [\fl_2]$, then $[\fl_1]=[\fl_2]$.
\end{corollary}

As we shall see, the transitivity property can fail.

\section{Nilpotent cones and polarized relations} \label{S:nc+pr}

To a multivariable nilpotent orbit $e^{z_1 N_1 + \cdots + z_{\ell}N_{\ell}}F$ on $D$, one can associate $2^{\ell}$ single-variable nilpotent orbits in 1-to-1 correspondence with the faces of the underlying nilpotent cone $\sigma$.  In this section we will show that inclusions of these faces are compatible with $\leq$ on $\Psi_D$ (Theorem \ref{T:2}), leading to the notion of ``polarized relation'' $\preceq$ on $\Psi_D$ discussed in \S\ref{iS:PR}.  We will also show that all polarized relations are realized by multivariable $\tSL(2)$-orbits (Theorem \ref{T:4}), which leads to an algorithm for computing $\preceq$.

\subsection{Nilpotent cones} \label{S:ncs}

Given a nilpotent cone 
\begin{equation} \nonumber 
  \s \ = \ \tspan_{\bR_{>0}}\{N_1 , \ldots , N_\ell\} \ \subset \ \fg_\bR
\end{equation}
underlying a nilpotent orbit, let $\Gamma_\s$ denote the partially ordered set comprising the faces of $\s$, including both $\s$ and the trivial vertex $\{0\}$, with partial order $\tau' \le \tau$ if and only if $\tau'$ is a face of $\tau$.  Of course the $\{ \tau\}$ also underlie nilpotent orbits (see \eqref{E:ftinb} below). Recall 

\begin{theorem}[{\cite[Corollary 4.9]{SL2}}] \label{T:R}
Let $\s$ be any cone underlying a nilpotent orbit and choose a Hodge flag $F \in \check D$ so that $(F,W(\s))$ is an $\bR$--split polarized mixed Hodge structure \emph{(equivalently, $F \in \tilde B_\bR(\s)$)}.  Fix $N\in\s$.  Then $\sigma$ is contained in the orbit of $N$ under a subgroup $L^{0,0}(\bR)^+ \subset G(\bR)^+$ that stabilizes $F$.
\end{theorem}

\noindent Applying Theorem \ref{T:R} to the cone $\tau$, we see that $\tau$ is contained in a conjugacy class $\cN_\tau \in \bfN_D$.  In particular, we obtain a well-defined map 
\[
  \pi^\circ : \Gamma_\s \ \to \ \bfN_D \,.
\]

In order to define the map $\Gamma_\s \to \bfD_D$ we must choose a connected component $\tilde B_\bR(\s)^\circ$ of 
\[
  \tilde B_\bR(\s) \ := \ \left\{ F \in \check D \ | \ (F,W(\s)[-k])
  \hbox{ is an $\bR$--split PMHS}\right\}\,.
\]
This choice determines a connected component of $\tilde B_\bR(\tau)^\circ$ of $\tilde B_\bR(\tau)$ as follows.  Fix $F \in \tilde B_\bR(\s)^\circ$.  Define $J_\tau \subset \{ 1 , \ldots , \ell\}$ by 
\[
  \tau \ = \ \tspan_{\bR_{>0}}\{N_j \ | \ j \in J_\tau \} \,.
\]
Set 
\[
  \tilde B(\tau) \ := \ \left\{ F \in \check D \ | \ (F,W(\tau)[-k])
  \hbox{ is an PMHS}\right\}\,,
\]
and note that 
\begin{equation}\label{E:ftinb}
  F_\tau \ := \ \exp\Big( \bi \sum_{j \not\in J_\tau} N_j\Big) F
  \ \in \ \tilde B(\tau) \,.
\end{equation}
Recall the smooth map
\[
  \delta_\tau : \tilde B(\tau) \ \to \ \tilde B_\bR(\tau) 
\]
of \cite[(2.20) and (2.24)]{\CKSdeg}.  (As checked in \cite[\S 4]{KP2012}, this is compatible with the Mumford-Tate group $G$.) Define $\tilde B_\bR(\tau)^\circ$ to be the connected component containing $\d_{\tau}(F_\tau)$.  

Let $N_{\tau}\in\tau$ and $\tilde{F}_{\tau}\in\tilde{B}_{\bR}(\tau)^{\circ}$.  We define 
\begin{equation}\label{E:glim}
  \phi_\infty^\circ : \Gamma_\s \ \to \ \bfD_D
\end{equation} 
by $\phi^{\circ}_{\infty}(\tau):=G(\bR)^+\cdot \underset{\text{Im}(z)\to \infty}{\lim} e^{z N_{\tau}}\tilde{F}_{\tau}$, and
\begin{equation}\label{E:glob}
\psi^{\circ}: \Gamma_{\s}\to \Psi_D
\end{equation}
by $\psi^{\circ}(\tau):=[\tilde{F}_{\tau},N_{\tau}]$.  Both are independent of the choice of $N_{\tau}$, $\tilde{F}_{\tau}$.  For $\phi^{\circ}_{\infty}$, this is Remark 5.6 of \cite{MR3331177}, but we can easily see it for both, by invoking


\begin{theorem}[{\cite[\S5]{KP2012}}] \label{T:KP}
The connected component $\tilde B_\bR(\tau)^\circ$ is homogeneous under the action of a Lie subgroup $M_{B(\tau)}(\bR)^+ \subset G(\bR)^+$ that point-wise fixes the elements of $\tau$.
\end{theorem}

\noindent From Theorems \ref{T:R} and \ref{T:KP}, we see that any second choice of $(\tilde F_\tau',N_\tau')$ is necessarily of the form $(\tilde F_\tau',N_\tau') = gh \cdot(\tilde F_\tau,N_\tau) = (g\cdot \tilde F_\tau , h \cdot N_\tau)$ with $g \in M_{B(\tau)}(\bR)^+$ fixing $\tau$ point-wise, and $h \in L^0_\tau(\bR)$ fixing $\tilde F_\tau$; thus, $\psi^\circ(\tau) = [\tilde F_\tau,N_\tau] \in \Psi_D$ is well-defined.

Moreover, it is clear from the definitions that $\pi^\circ$ and $\phi^{\circ}_{\infty}$ factor through $\psi^\circ$; in particular, $\phi^{\circ}_{\infty}$ is well-defined.  We remark that by Proposition \ref{P:sorb}, we could also have defined $\phi^{\circ}_{\infty} (\tau) := G(\bR)^+\cdot \tilde{F}_{\tau}$.  We now obtain the commutative diagram \eqref{E:cd}.

\subsection{Polarized relations} 

Corollary \ref{C:3} suggests a refinement of the relations on $\bfD_D$, $\bfN_D$ and $\Psi_D$.
Given $[F,N] , [F',N'] \in \Psi_D$, we write $[F',N'] \preceq [F,N]$ if there exists a nilpotent cone $\s$ underlying a nilpotent orbit, a face $\s' \in \Gamma_\s$ and a choice of connected component $\tilde B(\s)^\circ$ such that $\psi^\circ (\s') = [F',N']$ and $\psi^\circ(\s) = [F,N]$.  It follows at once from the definition and Theorem \ref{T:2} that  
\[
  [F',N'] \ \preceq \ [F,N] \quad\hbox{implies}\quad [F',N'] \ \le \  [F,N] \,.
\]

\begin{remark} \label{R:trivp}
It also follows directly from the definition that the trivial relations (Remark \ref{R:trivr}) are all polarized.
\end{remark}

\noindent Given $\sigma$ as above, we shall also write $\cO' \preceq \cO$ if $\cO' = \phi^\circ_\infty(\s')$ and $\cO = \phi^\circ_\infty (\sigma)$, and $\cN' \preceq \cN$ when $\cN' = \pi^\circ(\s')$ and $\cN = \pi^\circ(\sigma)$.  As above, 
\[
  \cO' \ \preceq \ \cO \quad \hbox{implies}\quad \cO' \ \le \ \cO \,,
\]
and
\[
  \cN' \ \preceq \ \cN
  \quad\hbox{implies}\quad 
  \cN' \ \le \ \cN \,.
\]
\begin{definition}
On $\mathbf{\Delta}_D$, $\mathbf{N}_D$, or $\Psi_D$, we call $\preceq$ a \emph{polarized relation}, and say that the polarized relation is \emph{realized by $\s$}. More heuristically, one should think of the polarized relations $\preceq$ as \emph{those relations} $\leq$ \emph{which may be realized Hodge theoretically}.
\end{definition}
\noindent Unlike the usual relations on $\bfN_D$ resp. $\bfD_D$, the polarizable ones all ``come from $\Psi_D$'', a fact which shall be (together with Theorem \ref{T:4} and \S\ref{S:multvarSL2}) useful in the general Mumford-Tate domain case (where $\Psi_D$ is the more computationally accessible object).

\subsection{Polarized relations and commuting horizontal $\tSL(2)$'s} 
As we will now discuss, when identifying the polarized relations it suffices to consider those coming from commuting horizontal $\tSL(2)$'s (Theorem \ref{T:4}).

To every nilpotent orbit
\begin{equation} \label{E:no}
  (z_1,\ldots,z_\ell) \ \mapsto \ \exp( \sum_j z_j N_j ) F \ \in \ \check D
\end{equation}
Cattani, Kaplan and Schmid \cite{\CKSdeg} associate a canonical Lie group homomorphism
\begin{equation} \label{E:up}
   \upsilon : \underbrace{\tSL(2,\bC) \times \cdots \times \tSL(2,\bC)}_{\ell \hbox{ terms}} \ \to \ 
   G(\bC) 
\end{equation}
(which is defined over $\bR$).  The homomorphism $\upsilon$ determines a second nilpotent orbit
\begin{equation} \label{E:hat-no}
   (z_1,\ldots,z_\ell) \ \mapsto \ \exp(\sum_j z_j \hat N_j) \hat F 
   \ \in \ \check D 
\end{equation}
that asymptotically approximates the first.  The relationship between \eqref{E:up} and \eqref{E:hat-no} is that the $\hat N_j$ are the images $\upsilon_*\bn_j$ of nilnegative elements  in standard triples $\{ \bn_j^+ , \by_j , \bn_j\}$ spanning pair-wise commuting $\fsl(2)$s.  (See \S\ref{S:st} for the definitions of standard triples and nilnegative elements.)  Let 
\[
  \hat \sigma \ = \ \tspan_{\bR_{>0}}\{\hat N_1 , \ldots , \hat N_\ell\}
\]
be the associated nilpotent cone.  Note that the $\{\hat N_j\}$ depend on our choice of ordering of the $\{N_j\}$.  In particular, reindexing the $N_j$ may yield a different set of $\{\hat N_j\}$ and thus a different cone $\hat \s$.

\begin{theorem} \label{T:4}
If a polarized relation is realized by the cone $\s$, then it is also realized by a cone $\hat\s$.  That is, all polarized relations on $\bfD_D$, $\bfN_D$ and $\Psi_D$ may be realized by horizontal commuting $\tSL(2)$'s.
\end{theorem}

\noindent Theorem \ref{T:4} is proved in \S\ref{S:T4}. It immediately follows that all such relations are realized by $2$-variable $\tSL(2)$-orbits, by taking an appropriate slice of $\hat{\sigma}$.  

\subsection{Classification of horizontal $\tSL(2)$--orbits}

In order to use Theorem \ref{T:4} to study the polarized relations in $\bfN_D$ and $\bfD_D$ it is necessary to understand the cones $\hat\s$ that arise from two-variable $\tSL(2)$--orbits; that is, it is necessary to have a good understanding of the orbits in Cattani, Kaplan and Schmid's theorem.  The single-variable horizontal $\tSL_2(\bC)$--orbits of Schmid's theorem \cite{\schmid} are classified in \cite{SL2}; this classification is briefly reviewed in \S\ref{S:1varSL2}.  We then explain in \S\ref{S:multvarSL2} how to obtain a classification of the several-variable orbits by induction.

\subsubsection{Classification of one--variable $\tSL(2)$--orbits} \label{S:1varSL2}

Recall that the upper half-plane $\sH \subset \bC$ is homogeneous under the action of $\tSL(2,\bR)$.  There is a well--known equivalence 
\begin{equation}\label{E:wk}
  \tilde B_\bR(D) \quad \longleftrightarrow \quad
  \left\{\hbox{$\tSL_2(\bR)$--equivariant embeddings $\sH \inj D$}
  \right\}\,,
\end{equation}
\cf\cite{\CKpmhs, \CKSdeg, \CKextn, \schmid, MR1245714}; these embeddings are the the \emph{horizontal $\tSL(2)$--orbits} on $D$.  There is a natural action of $G(\bR)$ on the right-hand side of \eqref{E:wk}; let $\Upsilon_D$ denote the set of $G(\bR)$--conjugacy classes.  The equivalence \eqref{E:wk} is $G(\bR)$--equivariant, and we have a natural identification $\Psi_D \simeq \Upsilon_D$.  That is, $\Upsilon_D \simeq \Lambda_{\varphi,\ft}$.  Briefly, and recalling the notation of \S\S\ref{S:Psi}, \ref{S:st} \& \ref{S:hSL2}, a representative $\upsilon : \tSL(2,\bC) \to G(\bC)$ of the conjugacy class $[\upsilon] \in \Upsilon_D$ corresponding to $[\fl_\bR] \in \Lambda_{\varphi,\ft}$ may be described as follows.  (See \cite{SL2} for further discussion.)  First recall that $\upsilon$ is determined by the image of the differential $\upsilon_*$ at the identity $\one \in \tSL(2,\bC)$.  Second, as a linear map, the differential is determined by the image of the basis 
\[
  \bar{\mathbf{e}} \ := \ \half \left(\begin{array}{cc} -\bi & 1 \\
            1 & \bi \end{array}\right) \,,\quad
  \bz \ := \ \left(\begin{array}{cc} 0 & -\bi \\ 
           \bi & 0 \end{array}\right) \,,\quad
  \mathbf{e} \ := \ \half \left(\begin{array}{cc} \bi & 1 \\ 
            1 & -\bi \end{array}\right)
\]
of $\fsl(2,\bC)$.  The class $[\upsilon]\in\Psi_D$ corresponding to $[\fl_\bR]$ is given by 
\[
  \upsilon_*\{ \bar{\mathbf{e}},\bz,\mathbf{e}\} \ = \ 
  \{ \overline\sE  , \sZ , \sE \} \,,
\]
where the right-hand side is the DKS--triple of Lemma \ref{L:DKS}.

\subsubsection{Classification of several-variable $\tSL(2)$--orbits} \label{S:multvarSL2}

A simple inductive argument yields a classification of the several variable orbits.  (Note that we only need to classify the two-variable ones.) One proceeds as follows:  Fix a Hodge structure $\varphi$ and suppose that $\fa_1,\fa_2 \subset \fgR$ are two commuting $\fsl(2)$'s that are horizontal at $\varphi$.  The algebras $\fa_1$ and $\fa_2$ commute if and only if $\fa_2$ is contained in the trivial isotypic component $\tilde \Gamma \subset \fgR$ of $\fa_1$.  The trivial isotypic component $\tilde\Gamma$ is a reductive Lie algebra, and $\fa_2$ is contained in the semisimple factor $\Gamma = [\tilde\Gamma , \tilde\Gamma]$.  Moreover, $\Gamma$ inherits a polarized Hodge structure from $\fgR$ by $\Gamma^{p,-p}_\varphi := \fg^{p,-p}_\varphi \cap \Gamma_\bC$; the polarization on $\Gamma$ is just the restriction of that on $\fg$.  Clearly $\fa_2\subset\Gamma$ is horizontal with respect to this induced Hodge structure on $\Gamma$ if and only if $\fa_2\subset \fgR$ is horizontal with respect to the original Hodge structure $\varphi$ on $\fgR$.  To summarize, the inductive classification of an $n$--tuple $\{ \fa_1 , \ldots , \fa_n\}$ of commuting $\fsl(2)$'s that are  horizontal with respect to $\varphi$ proceeds as follows:
\begin{numlist}
\item 
Apply \cite{SL2} to classify all $\fsl(2)$'s $\fa_1 \subset \fgR$ that are horizontal with respect to the Hodge structure $\varphi$.
\item Let $\Gamma_1 := \Gamma\subset\fgR$ be the semisimple factor of the trivial isotypic component of $\fa_1$.  Apply \cite{SL2} to classify all $\fsl(2)$'s $\fa_2 \subset \Gamma_1$ that are horizontal with respect to the induced Hodge structure $\varphi_1$ on $\Gamma_1$.  At this point we have a classification (up to the adjoint action of $G_\bR$) of commuting, horizontal pairs $\{ \fa_1 , \fa_2\}$.
\item The inductive hypothesis: suppose $2\le k \le n$, $\{ \fa_1,\ldots,\fa_k\}$ is a $k$--tuple of commuting $\fsl(2)$'s.  Then $\fa_k \subset \Gamma_{k-1}$, where $\Gamma_{k-1}$ is (the semisimple factor of) the trivial isotypic component for the adjoint action of $\fa_{k-1}$ on $\Gamma_{k-2}$.  (Our convention is that $\Gamma_0 = \fgR$.)  As part of the inductive hypotheses we further suppose that $\Gamma^{p,-p}_{k-1} = \Gamma_{k-1,\bC} \cap \fg^{p,-p}_\varphi$ defines a polarized Hodge structure on $\Gamma_{k-1}$, where the polarization is the restriction of that on $\fgR$, and that $\fa_k$ is horizontal with respect to this Hodge structure.
\item The induction: let $\Gamma_k \subset \Gamma_{k-1}$ be the (semisimple factor of the) trivial isotypic component for the adjoint action of $\fa_k$ on $\Gamma_k$.  As above $\Gamma^{p,-p}_k = \Gamma_{k,\bC} \cap \fg^{p,-p}_\varphi$ defines a polarized Hodge structure on $\Gamma_k$.  Let $\fa_{k+1} \subset \Gamma_k$ be any $\fsl(2)$ that is horizontal with respect to this Hodge structure.  Then $\{ \mathfrak{a}_1 , \ldots , \mathfrak{a}_{k+1}\}$ is a $(k+1)$--tuple of commuting $\fsl(2)$'s that are horizontal with respect to the original Hodge structure $\varphi$.
\end{numlist}

\begin{remark}
Of course, in order for this algorithm to be useful it is necessary that we be able to compute $\Gamma\subset\fgR$.  This is generally straightforward in explicit examples (see \S\ref{S:G2}).  And it is in general understood how to determine at least the isomorphism class of $\Gamma$; see, for example, \cite[\S15]{MR926619}.
\end{remark}

\begin{remark}
The number $n$ of commuting horizontal $\fsl(2)$'s is bounded by the real rank of $\fgR$ \cite{BPR}.
\end{remark}

In order to give the $n$--variable nilpotent orbit $\exp(\sum z_k N_k) \cdot F$ corresponding to the tuple $\{\fa_1 , \ldots , \fa_n\}$, it suffices to describe the $N_k\in\fgR$ and $F \in \partial D\subset\check D$.  First, let $\upsilon_k : \tSL(2) \to G$ be the embedding of the horizontal $\tSL(2)$ with Lie algebra $\fa_k$.  Then 
\[
  N_k \ = \
  v_{k,\ast} \left( \begin{array}{cc} 0 & 0 \\ 1 & 0 \end{array} \right) 
  \ = \ \half \,v_{k,\ast} (\bar{\mathbf{e}} + \mathbf{e} - \bi\bz) \,.
\]
Set 
\[
  N_k^+ \ := \
  v_{k,\ast} \left( \begin{array}{cc} 0 & 1 \\ 0 & 0 \end{array} \right) 
  \ = \ \half\,v_{k,\ast} (\bar{\mathbf{e}} + \mathbf{e} + \bi\bz) \,,
\]
and $\varrho_k := \exp \,\bi\tfrac{\pi}{4}(N_k^+ + N_k)$.  Then
\[
  F \ = \ \varrho_n \cdots \varrho_1 \cdot \varphi \,.
\]

\subsection{Proofs}

\subsubsection{Proof of Theorem \ref{T:2}} \label{S:T2}

Let $\sigma$, $N_{\sigma} \in \sigma$, and $F_{\sigma}\in\tilde{B}_{\bR}(\sigma)^{\circ}$ be given; we may assume that $(F_{\sigma},N_{\sigma})$ arises from $\fl_{\sigma}\in\cL_{\varphi,\ft}$ via \eqref{E:FN}.  As in \S\ref{S:po}, write $\tilde{\fl}_{\sigma,\bC} = \oplus\fg^{p,p}_{(F_{\sigma},N_{\sigma})}$.  Note that $\check{D}_{\tilde{\fl}} := \tilde{L}(\bC)\cdot F_{\varphi} = \tilde{L}(\bC)\cdot F_{\sigma}$ is the compact dual of the ``real M-T domain'' $D_{\tilde{\fl}} = \tilde{L}(\bR)^+\cdot \varphi$, all $\varphi '$ in which factor through $\tilde{L}(\bR)^+$.

Next, choose $\tau\in\Gamma_{\sigma}$, $N_{\tau}\in\tau$, and define $F_{\tau}:=e^{-\mathbf{i}\delta}e^{\mathbf{i}\sum_{j\notin J_{\tau}} N_j} F_{\sigma} \in \tilde{B}_{\bR}(\tau)^{\circ}$ (notations as in \S\ref{S:ncs}).  Then $\psi^{\circ}(\tau)=[F_{\tau},N_{\tau}]$ and $\psi^{\circ}(\sigma)=[F_{\sigma},N_{\sigma}]$, and we want to show that $\psi^{\circ}(\tau)\leq \psi^{\circ}(\sigma)$.  Equivalently, if $\fl_{\tau}\in \cL_{\varphi,\ft}$ is a Levi representing $\psi^{\circ}(\tau)$, it suffices to show that $\fl_{\tau}\subset w\tilde{\fl}_{\sigma}$ for some $w\in\sW^0$.

We construct such an $\fl_{\tau}$ as follows.

First, observe that $\sigma\subset \fg^{(-1,-1)}_{(F_{\sigma},N_{\sigma})}\subset \tilde{\fl}_{\sigma}$, so that $N_{\tau}\in\tau\subset\overline{\sigma}\subset\tilde{\fl}_{\sigma,\bR}$.  Next, since $\bC\tau\subset\tilde{\fl}_{\sigma,\bC}$, we have $\tilde{F}_{\tau}:=e^{\mathbf{i}\sum_{j\notin J_{\tau}}N_j}F_{\sigma}\in\check{D}_{\tilde{\fl}}$, so that the mixed-Hodge representation $\tilde{\varphi}$ associated to $(\tilde{F}_{\tau},W(N_\tau))$ (cf. \cite[\S I.C]{MR2918237}) factors through $\tilde{L}(\bC)$.  The Deligne splitting element $\delta$ which produces $F_{\tau}= e^{-\mathbf{i}\delta}\tilde{F}_{\tau}$ commutes with all $(r,r)$ morphisms of $\bR$-MHS, not just of $\fg$ but of all tensor spaces $T^{a,b}\fg$.  Equivalently, $\delta$ kills all $(p,p)$ tensors, hence belongs to $\tilde{\fl}_{\sigma}$, and so $F_{\tau}$ remains in $\check{D}_{\tilde{\fl}}$.  Applying \S\ref{S:subH}(\textsc{b}) to $\tilde{\fl}_{\sigma}$ and $(F_{\tau},W(N_{\tau}))$, and setting $F_{\tau} ':=\tilde{\fl}^{\tss}_{\bC}\cap F_{\tau}$, $(F_{\tau} ',N_{\tau})$ defines a PMHS on $\tilde{\fl}^{\tss}_{\bR}$.

Let $\fl_{\tau} '\in \cL_{\varphi|_{\tilde{\fl}_{\sigma}},\ft '}$ represent the class of $[F_{\tau} ',N_{\tau}]\in\Psi_{D_{\tilde{\fl}}^{\tss}} \simeq \cL_{\varphi|_{\tilde{\fl}_{\sigma}},\ft '}/\sW^0_{\tilde{\fl}_{\sigma}}$.  In particular, $\fl_{\tau}'$ is a Levi subgroup of $\tilde{\fl}^{\tss}_{\sigma}$, and $\fl_{\tau}:=\fz_{\sigma}\oplus \fl_{\tau} '$ is an element of $\cL_{\varphi,\ft}$.  We claim that this element represents $[F_{\tau},N_{\tau}]$.

Consider the commutative diagram
\begin{equation*}
\begin{tikzcd}
   \Psi_{D_{\tilde{\fl}^{\tss}_{\sigma}}} \arrow[r] 
   & 
   \Psi_D \\
   \cL_{\varphi|_{\tilde{\fl}_{\sigma}},\ft '} \arrow[r,"\oplus\fz_{\sigma}"] \arrow[u,two heads,"/\sW^0_{\tilde{\fl}_{\sigma}}"] 
  &
  \cL_{\varphi,\ft}\, . \arrow[u,two heads,"/\sW^0"]
\end{tikzcd}
\end{equation*}
with top row induced by the inclusions (of $\check{D}_{\tilde{\fl}_{\sigma}^{\tss}}\subset \check{D}$ and $\tilde{\fl_{\sigma}^{\tss}}\subset \fg$).  This obviously sends $[F_{\tau}',N_{\tau}]\mapsto [F_{\tau},N_{\tau}]$, and the claim follows.  Since $\fl_{\tau}\subset \tilde{\fl}_{\sigma}$ by construction, we are done.

\subsubsection{Proof of Theorem \ref{T:4}} \label{S:T4}

We will show that for every pair of faces $\tau \le \tau' \in \Gamma_\s$ there exists a homomorphism \eqref{E:up} and associated cone $\hat \s$ (depending on $\tau$ and $\tau'$) with faces $\hat \tau < \hat \tau' \in\Gamma_{\hat\s}$ so that 
\begin{equation}\label{E:T4a}
\begin{array}{rcl}
  \phi_\infty^\circ(\tau) & = & \hat\phi_\infty^\circ(\hat \tau) \,,\\
  \phi_\infty^\circ(\tau') & = & \hat\phi_\infty^\circ(\hat \tau') \,.
\end{array}
\end{equation}
This is precisely the assertion that every polarized relation on $\bfD_D$ may be realized by commuting horizontal $\tSL(2)$'s.  As $\phi_\infty : \Psi_D \to \bfD_D$ is a bijection (Theorem \ref{T:0}), it then follows from the commutativity of \eqref{E:cd} and the definition of ``$\preceq$'' that every polarized relation on $\Psi_D$ may be realized by commuting horizontal $\tSL(2)$'s.  Finally, invoking the commutativity of \eqref{E:cd} again, we conclude that every polarized relation on $\bfN_D$ may be realized by commuting, horizontal $\tSL(2)$'s.

Recall the notation preceding the statement of Theorem \ref{T:4}, and set 
\begin{eqnarray*}
  \tau_j & = & \tspan_{\bR_{>0}}\{ N_1 , \ldots , N_j \} \\
  \hat \tau_j & = & \tspan_{\bR_{>0}}\{\hat N_1 , \ldots , \hat N_j\} \,.
\end{eqnarray*}
Take $\tilde B(\s)^\circ$ and $\tilde B(\hat\s)^\circ$ to be the connected components containing $F$ and $\hat F$, respectively, and consider the corresponding maps $\phi_\infty^\circ$ defined on $\Gamma_\s$, and $\hat\phi^\circ_\infty$ defined on $\Gamma_{\hat\s}$.  To establish \eqref{E:T4a}, it clearly suffices to show that
\begin{equation} \label{E:T4b}
  \phi_\infty^\circ(\tau_j) \ = \ \hat\phi_\infty^\circ(\hat \tau_j) \, ,
\end{equation}
as our ordering of the $\{N_j\}$ is arbitrary.

The assertion \eqref{E:T4b} is precisely the statement that 
\[
  \lim_{y\to\infty} \exp(\bi\,y\,\sN_j) F 
  \tand 
  \lim_{y\to\infty} \exp(\bi\,y\,\hat \sN_j) \hat F
\]
lie in the same $G(\bR)^+$--orbit $\cO_j \subset \overline D$ for some (and therefore every, by Theorem \ref{T:KP}) $\sN_j \in \tau_j$ and $\hat\sN_j \in \hat\tau_j$.  This is a direct consequence of Cattani, Kaplan and Schmid's \cite[Theorem 4.20.vii-viii]{\CKSdeg}, as extended to Mumford-Tate domains in \S\ref{S:CKSMTD}.  \hfill\qed

\begin{remark}
Parts (v)-(vi) [resp. (vi)] of \cite[Theorem 4.20]{\CKSdeg} may be used to give a direct (but more complicated) argument for the realization of polarized relations on $\Psi_D$ [resp. $\bfN_D$] by commuting horizontal $\tSL(2)$'s.
\end{remark}

\section{$\mathrm{SL}(2)$-orbit theorem for MT domains} \label{S:CKSMTD}

In this section, we prove the extension of the multivariable $\tSL(2)$-orbit theorem to Mumford-Tate domains, which is needed in the proof of Theorem \ref{T:4} above.  More precisely, we show that if $\theta$
is a nilpotent orbit which takes values in a Mumford--Tate domain $D_M$ with
Mumford--Tate group $M$ then all of the constructs of Theorem $(4.20)$ 
in~\cite{\CKS} can be done using analytic functions and representations
with values in $M(\bR)^+$ and filtrations in the compact dual $\check D_M$.
  
\subsection{Splittings} Let $V_{\bR}$ be a finite
dimensional $\mathbb R$-vector space and $(F,W)$ be a mixed Hodge
structure on $V_{\mathbb C}=V_{\mathbb R}\otimes\mathbb C$.  Let
$V_{\mathbb C} = \bigoplus_{p,q}\, I^{p,q}_{(F,W)}$ denote the associated 
Deligne bigrading of $(F,W)$ (cf. \cite[(2.14)]{\CKS}).  
By \cite[Prop. 2.20]{\CKS}, there exists a unique
real endomorphism $\delta$ of $V_{\mathbb C}$ such that 
$\delta(I^{p,q})\subseteq \bigoplus_{a<p,b<q}\, I^{a,b}_{(F,W)}$
and 
\beq      
          (\tilde F,W) = (e^{-\mathbf{i}\delta}F,W)            \label{eq:ds-1}
\eeq
is a mixed Hodge structure which is split over $\mathbb R$.  Moreover,
every $(r,r)$-morphism of $(F,W)$ commutes with $\delta$. 

\par Lemma 6.60 of~\cite{\CKS} gives a further construction of a real
endomorphism $\zeta$ of $V_{\mathbb C}$ in terms of universal Lie
polynomials involving the Hodge components of $\delta$ with respect to
$(F,W)$.  In particular $\zeta$ commutes with all $(r,r)$-morphisms of 
$(F,W)$.  The mixed Hodge structure 
\beq
           (\hat F,W) = (e^{\zeta}e^{-\mathbf{i}\delta}F,W)      \label{eq:ds-3}
\eeq
is called the $\mathfrak{sl}_2$ or {\it canonical} splitting of $(F,W)$.

\par Let $D$ be a (connected) classifying space of pure Hodge structures of weight 
$k$ on $V_{\mathbb C}$ which are polarized by $Q$, and $\check D$ be 
the corresponding compact dual. Let $G:=\mathrm{Aut}(V,Q)^{\circ}$ (algebraic identity
component), so that $G(K)\cong \mathrm{Sp}(V_K,Q)$ or $\mathrm{SO}(V_K,Q)$ for $K=\bR$ or $\bC$
(with Lie algebra $\fg_K$), and $D$ (resp. $\check{D}$) is a $G(\bR)^+$- (resp. $G(\bC)$-) orbit.  Recall from above:

\begin{definition}[$(1.14)$,\,\cite{\CKS}] A nilpotent orbit is a  
map $\theta:\mathbb C^n\to\check D$ of the form 
$
     \theta(z) = \exp(\sum_j\, z_j N_j)F
$
where $N_1,\dots,N_n$ are commuting nilpotent elements of 
$\mathfrak g_{\mathbb R}$ which are horizontal at $F\in\check D$
for which there exists $\alpha\in\mathbb R$ such that 
$\theta(z)\in D$ whenever ${\rm Im}(z_1),\dots,{\rm Im}(z_n)>\alpha$.
\end{definition} 

\par If $N$ is a nilpotent endomorphism of $V_{\mathbb C}$ we let
$W(N)$ denote the monodromy weight filtration of $N$ centered at
zero.  Let $\theta(z) = e^{zN}F$ be a 1-variable nilpotent orbit
and $W=W(N)[-k]$. Then, by Schmid's 1-variable ${\rm SL}_2$-orbit 
theorem~\cite{\Schmid}, $(F,W)$ is a mixed Hodge structure relative
to which $N$ is a $(-1,-1)$-morphism.  Let  $(\tilde F,W)$ denote 
the Deligne splitting \eqref{eq:ds-1} of $(F,W)$ and $\tilde Y$ 
denote the semisimple endomorphism of $V_{\mathbb R}$ which acts 
by multiplication by $p+q-k$ on $I^{p,q}_{(\tilde F,W)}$.

\begin{lemma}\label{lem:split-1} If $(N,F)$ determines a 1-variable nilpotent orbit
then $(N,\tilde Y)$ is an $\mathfrak{sl}_2$-pair with associated
triple $(N,\tilde Y,\tilde N^+)$.  Moreover, (i) $N$ commutes with 
$\delta$ and $\zeta$, (ii) $N$ is horizontal at $\tilde F$, 
(iii) $\delta$, $\zeta$, $\tilde Y$, $\tilde N^+\in\mathfrak g_{\mathbb R}$ and (iv)
\beq
        e^{\mathbf{i}yN}\tilde F = \exp(-(1/2)\log(y)\tilde Y)e^{\mathbf{i}N}\tilde F\in D
        \label{eq:sl2-1}
\eeq
for all $y>0$.
\end{lemma}
\begin{proof} See \cite[(3.10)]{\CKS} for the statement that $(N,\tilde Y)$
is an $\mathfrak{sl}_2$-pair.  Since $N$ is a $(-1,-1)$-morphism of $(F,W)$ it
commutes with $\delta$ and $\zeta$.  See \cite[(3.11)]{\CKS} for the fact that 
$\delta$, $\tilde Y$ and $\tilde N^+$ belong to $\mathfrak g_{\mathbb R}$.  For the
statement that $\zeta\in\mathfrak g_{\mathbb R}$, see \cite[(6.60)]{\CKS}.
Equation \eqref{eq:sl2-1} is \cite[Lemma 3.12]{\CKS}.
\end{proof}

\begin{corollary}\label{cor:split-1} Let $(N,F)$ determine a nilpotent orbit, 
$(\hat F,W)$ denote the $\mathfrak{sl}_2$-splitting of $(F,W)$ and $\hat Y = {\rm Ad}(e^{\zeta})\tilde Y$,
$N^+ = {\rm Ad}(e^{\zeta})\tilde N^+$.
Since $\zeta\in\mathfrak g_{\mathbb R}$ and commutes with $N$, it follows 
that $(N,\hat Y,N^+)$ is an $\mathfrak{sl}_2$-triple, $\hat Y$, $N^+\in\mathfrak g_{\mathbb R}$, 
and $N$ is horizontal at $\hat F$.  Moreover, \eqref{eq:sl2-1} remains valid with 
$\tilde F$ replaced by $\hat F$ and $\tilde Y$ replaced by  $\hat Y$.  In particular, 
$\hat\theta(z) = e^{zN}\hat F$ is a nilpotent orbit which takes values in $D$ for 
${\rm Im}(z)>0$.
\end{corollary}

\par Let $\{N_1,\dots,N_n\}$ be commuting nilpotent elements of $\mathfrak g_{\mathbb R}$.
For $j=1,\dots,n$ let $\mathcal C_j = \{\sum_{\ell=1}^j\, a_{\ell} N_{\ell} \mid a_1,\dots,a_j>0\,\}$.
By~\cite{\CKpmhs}, if $\{N_1,\dots,N_n\}$ underlie a nilpotent orbit then every element 
$N\in\mathcal C_j$ determines the same monodromy weight filtration $W^j = W(N)[-k]$. 

\begin{lemma}\label{lem:split-2} If $(N_1,\dots,N_n;F)$ determine a nilpotent orbit 
$\theta_n:\mathbb C^n\to\check D$ then $(F,W^n)$ is a mixed Hodge structure.
Define $(F_n,W^n)$ to be the $\mathfrak{sl}_2$-splitting of $(F,W^n)$.  Then,
$\hat\theta_n(\uz) = e^{\sum_{\ell=1}^n\, z_{\ell}N_{\ell}}F_n$ is a nilpotent orbit 
with values in $D$ for ${\rm Im}(z_1),...,{\rm Im}(z_n)>0$.
\end{lemma}
\begin{proof} Let $N\in\mathcal C_n$.  Then, $(N,F)$ defines a 1-variable nilpotent orbit
and hence $(F,W^n)$ is a mixed Hodge structure.  By Corollary \ref{cor:split-1} it follows
that $\hat\theta(z) = e^{zN}F_n$ is a nilpotent orbit which takes values in $D$ for ${\rm Im}(z)>0$.
As the choice of $N\in\mathcal C_n$ was arbitrary, it follows that 
$\hat\theta_n$ is a nilpotent orbit with values in $D$ for ${\rm Im}(z_1)$,...,${\rm Im}(z_n)>0$.
\end{proof}

\par Given $\theta_n$ and $\hat\theta_n$ as in Lemma~\ref{lem:split-2}, set
$
      \theta_{n-1}(z_1,\dots,z_{n-1}) = \hat\theta_n(z_1,\dots,z_{n-1},\mathbf{i})
$.
Then, $\theta_{n-1}$ is a nilpotent orbit which takes values in $D$ for 
${\rm Im}(z_1)$,...,${\rm Im}(z_{n-1})>0$. Application of Lemma~\ref{lem:split-2} 
to $\theta_{n-1}$ produces a nilpotent orbit $\hat\theta_{n-1}$ with associated 
limit mixed Hodge structure $(F_{n-1},W^{n-1})$.  Iterating this process produces
nilpotent orbits 
\beq
       \theta_j(z_1,\dots,z_j) = \theta_{j+1}(z_1,\dots,z_j,\mathbf{i})   \label{eq:iterate}
\eeq
and
\beq     
     \hat\theta_j(z_1,\dots,z_j) = e^{\sum_{\ell=1}^j\, z_{\ell}N_{\ell}}F_{j}
     \label{eq:split-orbits}
\eeq
with values in $D$ for ${\rm Im}(z_1),...,{\rm Im}(z_j)>0$, terminating  
at the constant orbit $\theta_0 = \hat\theta_0 = \hat\theta_1(\mathbf{i}) = F_0\in D$.  
Let $W^0$ be the trivial filtration of weight $k$ on $H_{\mathbb C}$ and 
$Y_{(j)}$ denote the semisimple endomorphism which acts as multiplication 
by $p+q-k$ on $I^{p,q}_{(F_j,W^j)}$.   By  Corollary~\ref{cor:split-1}, 
$Y_{(j)}\in\mathfrak g_{\mathbb R}$. In~\cite{\CKS}, $F_j$ is denoted
$\tilde F_{\bf j}$.

\begin{theorem}\label{thm:ds-1} The elements $Y_{(0)},\dots,Y_{(n)}$ commute.  
Define $H_j = Y_{(j)} - Y_{(j-1)}$ for $j>0$ and let $N_j^0$ denote the 
projection of $N_j$ onto $\ker({\rm ad}\, Y_{(j-1)})$ with respect to the 
decomposition of $\mathfrak g_{\mathbb R}$ into the eigenspaces 
of ${\rm ad}\, Y_{(j-1)}$.
Then,
\beq
       (N_1^0,H_1),\dots,(N_n^0,H_n)                 \label{eq:ds-2}
\eeq
are commuting $\mathfrak{sl}_2$-pairs.  In particular, $N_1^0 = N_1$ since 
$Y_{(0)} = 0$.
\end{theorem}  
\begin{proof} The assertion that the elements $Y_{(0)},\dots,Y_{(n)}$ commute
is part of \cite[Thm. 4.20]{\CKS}. An alternative algebraic proof was 
sketched by Deligne in~\cite{\DL}.  More precisely, let $Y^j$ be the grading 
of $W^j$ which acts on $I^{p,q}_{(F_j,W^j)}$ as multiplication by $p+q$. By 
Deligne's results, $[Y^{n-1},Y^n]=0$.  Moreover $N_n^0$ and 
$H_n=Y^n-Y^{n-1} = Y_{(n)} -Y_{(n-1)}$ form an $\mathfrak{sl}_2$-pair which commutes 
with $N_1,\dots,N_{n-1}$.  Proceeding by downward induction gives the system 
of commuting $\mathfrak{sl}_2$-pairs \eqref{eq:ds-2}.  To recover the fact that 
$Y_{(0)},\dots,Y_{(n)}$ commute, observe that 
$Y_{(j)} = \sum_{\ell=1}^j H_{\ell}$ for $j>0$ since $Y^0 = k\mathbf{1}$.
\end{proof}

\begin{corollary}\label{cor:sl2-rep} $N_j^0$ commutes with $Y_{(\ell)}$ for $\ell<j$.
Moreover, we have the following dictionary with~\cite{\CKS}{\rm :}
$Y_{(j)}\leftrightarrow\hat Y_{\bf j}$, $H_j\leftrightarrow\hat Y_j$ and 
$N_j^0\leftrightarrow \hat N_j^-$. In particular, the $\mathfrak{sl}_2$-pairs 
\eqref{eq:ds-2} generate the representation $\rho:({\rm SL}_2)^n\to G(\bR)^+$ 
of \cite[Thm. 4.20]{\CKS}.
\end{corollary}
\begin{proof} By part $(ii)$ of \cite[Thm. 4.20]{\CKS}, $\hat Y_{\bf j} = Y_{(j)}$
since $\tilde F_{\bf j} = F_j$.  If $\ell<j$ then $[N_j^0,Y_{(\ell)}] =0$ since 
$Y_{(\ell)} = \sum_{r=1}^{\ell} H_r$ and $[N_j^0,H_r] = 0 $ for $r=1,\dots,\ell$.  
By part $(iii)$ of \cite[Thm. 4.20]{\CKS}, it follows that $\hat N_j^- = N_j^0$.  
By equation $(4.18)$ of~\cite{\CKS}, $\hat Y_{\bf r} = \sum_{j\leq r}\, \hat Y_j$ 
and hence $\hat Y_r = \hat Y_{\bf r}-\hat Y_{{\bf r}-1} = Y_{(r)} - Y_{(r-1)} = H_r$.
\end{proof}

\begin{remark}\label{rmk:ds-3} The proof that Deligne's construction gives the 
$\mathfrak{sl}_2$-splitting along nilpotent orbits appears in~\cite{\BP}.  A survey of 
Deligne's results appears in \S 6 of~\cite{BPR}. 
\end{remark}

\par Suppose now that $V_{\mathbb R}$ admits a rational form $V_{\mathbb Q}$
relative to which $Q$ is rational.  Let $D_M$ be a (connected) Mumford--Tate subdomain of 
$D$ with Mumford--Tate group $M$ and compact dual $\check D_M$.  By Proposition (VI.B.11) of~\cite{\GGK}, 
$D_M$ is a closed subset of $D$ in the analytic topology.  
For $K=\mathbb R$ or $\mathbb C$ let $\mathfrak m_K$ be the Lie algebra of $M_K$.
 
\begin{lemma}\label{lem:split-3} Suppose that $N\in\mathfrak m_{\mathbb R}$ and 
$F\in\check D_M$ determine a nilpotent orbit 
$\theta:\mathbb C\to\check D$ such that $\theta(z)\in D_M$ for ${\rm Im}(z)>\alpha$.  
Then {\rm (}cf.~Lemma~\ref{lem:split-1}, Corollary~\ref{cor:split-1}{\rm )}, $\delta$, $\zeta$, 
$\tilde Y$, $\hat Y$, $\tilde N^+$, $N^+\in\mathfrak m_{\mathbb R}$ and 
$\hat\theta(z) = e^{zN}\hat F$ is a nilpotent orbit which takes values in $D_M$ 
for ${\rm Im}(z)>0$.
\end{lemma}

\begin{remark}
In general $\check{D}_M \cap D$ can have multiple (finitely many) connected components;
$D_M$ is, by definition, one of these. If we only assumed in Lemma \ref{lem:split-3} that $\theta(z)\in D_M$
for $\mathrm{Im}(z)>\alpha$, then the conclusion would be that $\hat{\theta}(z)$ takes values in one of these components (not
necessarily $D_M$) for $\mathrm{Im}(z)>0$.
\end{remark}

\begin{proof} An analytic proof that $\tilde Y$ and $\tilde N^+$ belong to $\mathfrak m_{\mathbb R}$ 
following the methods of~\cite{\Schmid} appears in Proposition (IV.A.13) of~\cite{\GGK}.  An 
algebraic proof that $\tilde Y$ and $\tilde N^+$ belong to $\mathfrak m_{\mathbb R}$ is given 
in~\cite{\KP}.  Let $W=W(N)[-k]$.  The fact that $\delta$ and its Hodge components relative 
to $(F,W)$ belong to $\mathfrak m_{\mathbb R}$ is stated on \cite[p. 682]{\KP}.  Since $\zeta$ 
is given by universal Lie polynomials in the Hodge components of $\delta$, 
it follows that  $\zeta\in\mathfrak m_{\mathbb R}$, and hence so do 
$\hat Y = {\rm Ad}(e^{\zeta})\tilde Y$ and $N^+ = {\rm Ad}(e^{\zeta})\tilde N^+$.

\par By Corollary~\ref{cor:split-1}, $\hat\theta$ is a nilpotent
orbit such that $\hat\theta(z)\in D$ for ${\rm Im}(z)>0$.  Let
$e^{\xi} = e^{\zeta}e^{-\mathbf{i}\delta}$ and define $\xi(y)$ by the formula 
$e^{\xi(y)} = {\rm Ad}(\exp((1/2)\log(y)\hat Y))e^{\xi}$. 
Then $\lim_{y\to\infty}\, \xi(y) = 0$ because $\hat Y$ is a grading
of $W$ and $\xi(W_{\ell})\subset W_{\ell-2}$ for each index $\ell$. 
Likewise, because $\hat Y$ preserves $\hat F$ and $[\hat Y,N] = -2N$
we have
$$
     \exp((1/2)\log(y)\hat Y)e^{\mathbf{i}yN}F = e^{-\xi(y)}e^{\mathbf{i}N}\hat F
$$
Accordingly, as the left hand side of this equation takes values
in $D_M$ whereas the right hand side limits to $e^{\mathbf{i}N}\hat F\in D$,
it follows that $e^{\mathbf{i}N}\hat F\in D_M$ since $D_M$ is a closed
subset of $D$.  Consequently, 
$
        e^{\mathbf{i}yN}\hat F = \exp(-(1/2)\log(y)\hat Y)e^{\mathbf{i}N}\hat F\in D_M
$
for $y>0$.
\end{proof}

\begin{theorem}\label{thm:mt-rep} Let $(N_1,\dots,N_n;F)$ define a 
nilpotent orbit $\theta_n:\mathbb C^n\to\check D$.  Suppose that 
$N_1,\dots,N_n\in\mathfrak m_{\mathbb R}$, $F\in\check D_M$ and there exists
$\alpha>0$ such that $\theta(z)$ takes values in $D_M$ for 
${\rm Im}(z_1),\dots,{\rm Im}(z_n)>\alpha$.  Then, the $\mathfrak{sl}_2$-pairs
\eqref{eq:ds-2} take values in $\mathfrak m_{\mathbb R}$ and hence the
representation $\rho$ of Corollary~\ref{cor:sl2-rep} takes values in 
$M_{\mathbb R}$.  Moreover the filtrations $F_1,\dots,F_n\in\check D_M$
and $F_0\in D_M$.
\end{theorem}
\begin{proof} For any $N\in\mathcal C_n$, the pair $(N,F)$ determines
a nilpotent orbit $e^{zN}F$ with values in $D_M$ for ${\rm Im}(z)$
sufficiently large.  Therefore, $e^{zN}\hat F=e^{zN}F_n$ takes values 
in $D_M$ for ${\rm Im}(z)>0$ by Lemma~\ref{lem:split-3}, and hence
$\hat Y = Y_{(n)}\in\mathfrak m_{\mathbb R}$.  Iterating as in 
\eqref{eq:split-orbits} shows that each $Y_{(j)}\in\mathfrak m_{\mathbb R}$.
Likewise, since $N_j^0$ is the projection of $N_j\in\mathfrak m_{\mathbb R}$
to $\ker({\rm ad}\, Y_{(j-1)})$ with respect to the eigenspaces of 
${\rm ad}\,Y_{(j-1)}$, it follows that $N_j^0\in\mathfrak m_{\mathbb R}$.
Finally, since the orbits \eqref{eq:split-orbits} take values in $D_M$
and $N_1,\dots,N_n\in\mathfrak m_{\mathbb R}$ it follows that 
$F_j\in\check D_M$ and $F_0\in D_M$.
\end{proof}

\par Equation $(4.9)$ of~\cite{\CKS} defines a choice of reference Hodge 
structure on $\mathfrak{sl}_{2,\bC}$. Part (i) of \cite[Thm. 4.20]{\CKS} asserts 
that if $\rho$ is the representation attached to a nilpotent orbit 
$\theta_n$ as in Theorem \ref{thm:ds-1} and Corollary \ref{cor:sl2-rep}, then
\beq
          \rho_*:(\mathfrak{sl}_{2,\bC})^{\oplus n}\to \mathfrak g_{\mathbb C}
          \label{eq:hodge-rep}
\eeq
is morphism of Hodge structure of type $(0,0)$ when 
$\mathfrak g_{\mathbb C}$ is equipped with the Hodge structure
induced by $F_0\in D$.  In the setting of Theorem \ref{thm:mt-rep} above,
$\mathfrak m_{\mathbb C}$ is a Hodge substructure of 
$\mathfrak g_{\mathbb C}$ relative to $F_0\in D_M$.  Therefore, by
the strictness of morphisms of Hodge structures, it follows 
that $\rho_*$ defines a morphism of Hodge structure to 
$\mathfrak m_{\mathbb C}\subset\mathfrak g_{\mathbb C}$.

\par Part (i) of \cite[Thm. 4.20]{\CKS} also asserts that
$\tilde F_{\bf r} = (\Pi_{j=1}^r\, e^{-\mathbf{i}\hat N_j^-})e^{\mathbf{i}\hat N_1^-}\tilde F_{\bf 1}$.
In the setting of Theorem \ref{thm:mt-rep}, $\tilde F_{\bf r}$ and $\tilde F_{\bf 1}$ 
belong to $\check D_M$ and $e^{-\mathbf{i}\hat N_1^-},\dots,e^{-\mathbf{i}\hat N_r^-}\in M_{\mathbb C}$. 
Thus, in the setting of Theorem \ref{thm:mt-rep}, the constructs of parts (i) to (iii) 
of \cite[Thm. 4.20]{\CKS} only involve representations with values in $M$ 
and filtrations in $\check D_M$.  

\subsection{Univariate Orbits} Parts (iv) to (ix) of \cite[Thm. 4.20]{\CKS}
involve analytic functions with values in $G(\bR)^+$.  In this section, we
show that for 1-variable nilpotent orbits with values in $D_M$, these functions
take values in $M_{\mathbb R}$.

\par Let $(N,F)$ determine a nilpotent orbit $\theta:\mathbb C\to\check D$, and 
define $\hat F$, $\delta$, $\zeta$, $N^+$ and $\hat Y$ as in Corollary \ref{cor:split-1}; then we have
$F_b := e^{\mathbf{i}N}\hat F\in D$.  Let $G^0(\bR)$ denote the stabilizer of 
$F_b$ in $G(\bR)$, and $\mathfrak{g}^0_{\mathbb R}$ be the Lie algebra of 
$G^0(\bR)$.
Let 
$
       \mathfrak g_{\mathbb C} = \oplus_p\, \mathfrak g^{p,-p}
$
be the Hodge decomposition induced by $F_b$ on $\mathfrak g_{\mathbb C}$.  
As above, write $\fg^{0,0}=:\fg^0$.
Then $\mathfrak{g}^0_{\mathbb R} = \mathfrak{g}^{0}\cap\mathfrak g_{\mathbb R}$ 
and hence 
\beq
     \mathfrak{g}'= (\bigoplus_{p\neq 0}\, \mathfrak g^{p,-p})\cap\mathfrak g_{\mathbb R}  
     \label{eq:c-1}
\eeq
is an ${\rm Ad}(G^0(\bR))$-invariant vector space complement to $\mathfrak{g}^0_{\mathbb R}$
in $\mathfrak g_{\mathbb R}$.  Let $\nabla$ denote the associated connection on the 
principal bundle
\beq
      G^0(\bR)\to G(\bR)^+\to G(\bR)^+/G^0(\bR) \cong D   \label{eq:c-2}
\eeq

\par Suppose that $\theta(z)\in D$ for ${\rm Im}(z)>a$ and let $h:(a,\infty)\to G(\bR)^+$
be a lifting of $y\mapsto\theta(\mathbf{i}y)$ which is tangent to $\nabla$, i.e.
\begin{itemize}
\item[(i)] $\theta(\mathbf{i}y) = h(y) F_b$ for $y>a$;
\item[(ii)] $h^{-1}(y)h'(y)\in\mathfrak{g}'$. 
\end{itemize}
Set $W=W(N)[-k]$. Then, \S 3 and \S 6 of~\cite{\CKS} can be summarized as
the following version of the 1-variable ${\rm SL}_2$-Orbit theorem:

\begin{theorem}\label{thm:single-sl2} Let $(N,F)$ define a nilpotent
orbit $\theta:\mathbb C\to\check D$ such that $\theta(z)\in D$
for ${\rm Im}(z)>a$.  Then, there exists a real-analytic function 
$g:(a,\infty)\to G(\bR)^+$ such that 
\begin{itemize}
\item[(a)]  $\theta(\mathbf{i}y) = g(y)e^{\mathbf{i}yN}\hat F = g(y)\exp(-(1/2)\log(y)\hat Y)F_b$
            for $y>a$;
\item[(b)] $h(y) = g(y)\exp(-(1/2)\log(y)\hat Y)$ satisfies conditions $(i)$ and $(ii)$
            above;
\item[(c)] $g(y)$ and $g^{-1}(y)$ have convergent power series about $\infty$ of the form 
           \begin{equation*}
           \begin{gathered}
                g(y) = 1 + g_1 y^{-1} + g_2 y^{-2} + \cdots \\
                g^{-1}(y) = 1 + f_1 y^{-1} + f_2 y^{-2} + \cdots 
           \end{gathered}
           \end{equation*}
           with $g_k$, $f_k\in\ker({\rm ad}\, N)^{k+1}$;
\item[(d)] $$
               e^{\mathbf{i}\delta}e^{-\zeta} = 1 + \sum_{k>0}\, \frac{1}{k!}(-\mathbf{i})^k ({\rm ad}\, N)^k g_k
           $$ 
           Moreover, $g_k$ and $f_k$ can be expressed as universal Lie polynomials over
           $\mathbb Q(\sqrt{-1})$ in the Hodge components of $\delta$ with respect to 
           $(\hat F,W)$ and ${\rm ad}\, N^+$.
\end{itemize}
\end{theorem}

\par Suppose now that $N\in\mathfrak m_{\mathbb R}$ and $F\in\check D_M$ determine a
nilpotent orbit $\theta:\mathbb C\to\check D$ such that $\theta(z)\in D_M$ for 
${\rm Im}(z)>a$.  Then, $F_b\in D_M$ by Lemma \ref{lem:split-3}, and $\delta$, 
$\zeta$, $\hat Y$, $N^+\in\mathfrak m_{\mathbb R}$.  Let 
$M^0(\bR)=G^0(\bR)\cap M(\bR)$ be the stabilizer of $F_b$ in
$M(\bR)$ with Lie algebra ${\mathfrak m}^0_{\mathbb R}$.  Then
$\mathfrak{m}^0_{\mathbb R} = \mathfrak g^{0}\cap\mathfrak m_{\mathbb R}$ and 
$\mathfrak{m}_{\bR} ' = \mathfrak{g}'\cap\mathfrak m_{\mathbb R}$ is an 
${\rm Ad}(M^0(\bR))$-invariant complement to $\mathfrak{m}^0_{\bR}$
in $\mathfrak{m}_{\mathbb R}$.  We therefore obtain a corresponding 
connection $\nabla^{\mathfrak m}$ on the principal bundle
\beq
      M^0(\bR)\to M(\mathbb{R})^+\to M(\mathbb{R})^+ /M^0(\mathbb{R}) \cong D_M   \label{eq:c-3}
\eeq

\par Let $g:(a,\infty)\to G(\bR)^+$ be the function constructed from $\theta$ by 
Theorem~\ref{thm:single-sl2}. By parts (c) and (d), it follows that $g(y)\in M(\mathbb{R})^+$.  
By part (b), it then follows that $h$ is an $M(\mathbb{R})^+$-valued function which satisfies 
conditions (i) and (ii).  Moreover, since $h$ takes values in $M(\mathbb{R})^+$, condition
(ii) implies that
$$
        h^{-1}(y)h'(y) \in \mathfrak{g}'\cap\mathfrak{m}_{\mathbb R} = \mathfrak{m}'
$$
Thus, $h$ is a lift of $\theta$ which is tangent to $\nabla^{\mathfrak m}$.  In summary, 
in the case of a 1-variable nilpotent orbit with values in $D_M$, the analytic functions 
$g$ and $h$ of the ${\rm SL}_2$-orbit theorem take values in  $M(\bR)^+$, the filtrations
belong to $\check D_M$, and all of the Lie algebra theoretic data takes values in 
$\mathfrak m_{\mathbb R}$.

\subsection{Several Variable Orbits}  Let $(N_1,\dots,N_n;F)$ determine a nilpotent orbit
$\theta_n:\mathbb C^n\to\check D$ such that $\theta(\uz)\in D$ if 
$\text{\rm Im}(z_1),\dots,\text{\rm Im}(z_n)>\beta$.  Fix $\alpha>\beta$ and set $c=\beta/\alpha$.
Then, given $\uy\in\mathbb R^n$ with coordinates $y_1,\dots,y_n>\alpha$, the map
\beq
       \theta_{n,\uy}(w) = \theta_n(wy_1,\dots,wy_n)           \label{eq:sev-1}
\eeq
is a nilpotent orbit such that $\theta_{n,\uy}(w)\in D$ if ${\rm Im}(w)>c$.  Let 
$g_{n,\uy}:(c,\infty)\to G(\bR)^+$ be the function attached to $\theta_{n,\uy}(w)$ 
by Theorem~\ref{thm:single-sl2} and observe that $c<1$.  Define
\beq
      g_{\bf n}(y_1,\dots,y_n) = g_{n,\uy}(1)\in G(\bR)^+,\qquad
      y_1,\dots,y_n>\alpha         \label{eq:sev-2}
\eeq
as on \cite[p. 496]{\CKS}.

\par To continue, we recall that for $1\leq j\leq n-1$ the orbits $\theta_j$
defined by \eqref{eq:iterate} take values in $D$ for 
${\rm Im}(z_1),\dots,{\rm Im}(z_j)>0$.  Given $\uy\in \bR^j$ with coordinates
$y_1,\dots,y_j$ let 
\beq
       \theta_{j,\uy}(w) = \theta_j(wy_1,\dots,wy_j)         \label{eq:sev-3}
\eeq
and $g_{j,\uy}:(0,\infty)\to G(\bR)^+$ be the function attached 
to $\theta_{j,\uy}(w)$ by Theorem~\ref{thm:single-sl2}.  Define 
\beq
     g_{\bf j}(y_1,\dots,y_j) = g_{j,\uy}(1)                \label{eq:sev-4}
\eeq
as in [loc. cit.].  Finally, define
$h_{\bf r}(y_1/y_r,\dots,y_{r-1}/y_r;y_r)$ (also as in [loc. cit.]) via the formula
\beq
      g_{\bf r}(y_1,\dots,y_r) 
       = h_{\bf r}(y_1/y_r,\dots,y_{r-1}/y_r;y_r)\exp((1/2)\log(y_r)Y_{(r)})
         \label{eq:sev-5}
\eeq
(recall $Y_{(r)}\leftrightarrow\hat Y_{\bf r}$ in our dictionary with~\cite{\CKS}). 
          
\par Suppose now that $N_1,\dots,N_n\in\mathfrak m_{\mathbb R}$ and $F\in\check D_M$
define a nilpotent orbit $\theta_n:\mathbb C^n\to\check D$ such that 
$\theta_n(\uz)\in D_M$ if ${\rm Im}(z_1),\dots,{\rm Im}(z_n)>\beta$.  Then, 
$\theta_{n,\uy}(w)$ takes values in $D_M$ for ${\rm Im}(w)>c$ and hence the
function $g_{\bf n}$ defined by \eqref{eq:sev-2} takes values in $M(\bR)^+$.
Likewise, $\theta_{j,\uy}(w)$ take values in $D_M$ for ${\rm Im}(w)>0$, and hence
the function $g_{\bf j}$ defined by \eqref{eq:sev-4} takes values in
$M(\bR)^+$. Finally, since $Y_{(1)},\dots, Y_{(n)}$ take values in 
$\mathfrak m_{\mathbb R}$ by Theorem~\ref{thm:mt-rep}, equation \eqref{eq:sev-5}
defines a function with values in $M(\bR)^+$.

\subsection{Supplements} Let $(N_1,\dots,N_n;F)$ 
determine a nilpotent orbit $\theta_n:\mathbb C^n\to\check D$ and $(H_j,N_j^0)$
denote the associated $\mathfrak{sl}_2$-pairs \eqref{eq:ds-2}. Recall that $H_j = Y^j - Y^{j-1}$
where $Y^0,\dots,Y^n$ commute and $Y^0 = k\mathbf{1}$ (proof of Theorem~\ref{thm:ds-1}).  
Given positive real numbers $y_1,\dots,y_n$ let $y_{n+1} = 1$ and $t_j = y_{j+1}/y_j$. 
Define $y_j^A = \exp(\log(y_j)A)$.  Then, a reindexing argument shows that  
\beq
            t(\uy) = \Pi_{j=1}^n\, t_j^{(1/2)Y^j} 
                 = y_1^{-(k/2)}\Pi_{j=1}^n\, y_j^{-H_j/2} . \label{eq:bp-ident}
\eeq
The function $t(\uy)$ appears in equations \cite[(6.1)]{\BP} 
and \cite[(1.12)]{\HP}.  The scalar factor $y_1^{-(k/2)}$ can be omitted when 
considering the adjoint action of $t(\uy)$ on $\mathfrak g_{\mathbb C}$ or the action of $t(\uy)$ 
on $\check D$.  In the notation of~\cite{\CK2}, $y_1^{k/2}t(\uy) = e^{-1}(t)$. 

\par Let $\Delta^n$ denote the unit polydisc with coordinates $(s_1,\dots,s_n)$ and $\Delta^{*n}$
be the complement of the divisor $s_1\cdots s_n=0$.  Let $(z_1,\dots,z_n)$ be Cartesian coordinates
on $\mathbb C^n$ and $z_j = x_j + \mathbf{i} y_j$.  Let $\mathfrak{H}^n$ be the product of upper half-planes defined
by $y_1,\dots,y_n>0$ and $I'\subset \mathfrak{H}^n$ be the set defined by $y_1\geq y_2\geq\cdots\geq y_n\geq 1$.
Let $\mathfrak{H}^n\to\Delta^{*n}$ be the covering map defined by $s_j = e^{2\pi \mathbf{i}z_j}$.  Recall that a period map 
$\Phi:\Delta^{*r}\to\Gamma\backslash D_M$ lifts to a holomorphic, horizontal map $\tilde{\Phi}:\mathfrak{H}^n\to D_M$.
Let $\mathfrak m_{\mathbb C} = \oplus_{p,q}\,\mathfrak m^{p,q}$ denote the Deligne bigrading of
$\mathfrak m_{\mathbb C}$ induced by the limit mixed Hodge structure of $\Phi$.

\begin{theorem}\label{thm:rel-compact} Let $\tilde{\Phi}:\mathfrak{H}^n\to D_M$ be a lift of a period map 
$\Phi:\Delta^{*n}\to\Gamma\backslash D_M$ with unipotent monodromy.  Let $t(\uy)$ be the 
function~\eqref{eq:bp-ident} attached to the nilpotent orbit $\theta(\uz) = e^{\sum_j\, z_j N_j}F_{\infty}$ 
of $\Phi$.  Then, the image of $I'$ under the map $\uz\in \mathfrak{H}^n \mapsto t(\uy)^{-1}e^{-\sum_j\, x_j N_j}\tilde{\Phi}(z)$ 
is a relatively compact subset of $D_M$.
\end{theorem}
\begin{proof} For period maps into $D$, this is \cite[Thm. 4.7]{\CK2}.  The analog for period
maps of admissible variations of mixed Hodge structure is \cite[Thm. 7.1]{\BP}.  Let 
$\psi(\underline{s}) = e^{-\sum_j\, z_j N_j}\tilde{\Phi}(\uz)$ and $\mathfrak q = \oplus_{p<0,q}\,\mathfrak m^{p,q}$.
The hypothesis that $\Phi:\Delta^{*n}\to\Gamma\backslash D_M$ forces $y_1^{k/2}t(\uy)$ to take values 
in $M(\bR)^+$.  Moreover, we can write $\psi(\underline{s}) = e^{\Gamma(\underline{s})}F_{\infty}$ for a unique
$\mathfrak q$-valued holomorphic function $\Gamma(\underline{s})$ on a neighborhood of $\underline{s}=\underline{0}$ (cf. \cite[(2.5)]{\CK2}).  With these two observations in hand, the proof now follows verbatim from \cite[Thm. 4.7]{\CK2} or \cite[Lemma 7.1]{\BP}.
\end{proof}

\begin{theorem} In the setting of Theorem~\ref{thm:rel-compact} there exist constants $\alpha$, 
$\beta_1,\dots,\beta_n$ and $C$ such that if ${\rm Im}(z_1),\dots,{\rm Im}(z_n)>\alpha$ then 
$\theta(\uz)\in D_M$ and 
$$
          d(\tilde{\Phi}(\uz),\theta(\uz))< C\sum_{j=1}^n\, {\rm Im}(z_j)^{\beta_j}e^{-2\pi{\rm Im}(z_j)}
$$
where $d$ denotes the $M(\bR)^+$-invariant metric on $D_M$ induced by the Hodge metric.
\end{theorem}
\begin{proof} This follows verbatim from the proofs in \S 5 and \S 6 of~\cite{\HP} since
$y_1^{k/2}t(\uy)$ takes values in $M(\bR)^+$ and $\psi(\underline{s})=e^{\Gamma(\underline{s})}F_{\infty}$
with $\Gamma(\underline{s})$ taking values in $\mathfrak q\subset\mathfrak m_{\mathbb C}$.  Alternatively,
one can revisit the proof given in~\cite{\CKS} for period maps into $D$.  The main point is to
use the fact that $M(\bR)^+$ acts transitively by isometries and a careful analysis of
the 1-variable case.
\end{proof}

\begin{theorem}\label{thm:field-def} Let $K$ be a subfield of $\mathbb R$, and assume
that $V_{\mathbb R}$ arises by extension of scalars from a finite dimensional $K$-vector space 
$V_K$ such that $Q:V_K\otimes V_K\to K$.  Let $\mathfrak g_K$ denote the Lie algebra of 
infinitesimal isometries of $Q$ over $K$, and suppose that in the setting of Theorem \ref{thm:ds-1}, 
$Y_{(n)}$ and $N_1,\dots,N_n\in\mathfrak g_K$.  Then, each of the $\mathfrak{sl}_2$-pairs 
\eqref{eq:ds-2} consists of elements of $\mathfrak g_K$.
\end{theorem}
\begin{proof} Since $N_1+\dots+N_j\in \mathfrak{gl}(V_K)$ it follows that $W^j$
arises by extension of scalars from an increasing filtration of $V_K$.
The data $(W^{n-1},N_n,Y^n)$ is therefore a Deligne system as defined
in \S 6 of~\cite{BPR} where $Y^n = Y_{(n)} + k\mathbf{1}$.  
Consequently, $Y^{n-1}$ is defined over $K$, and hence so are 
$H_n = Y^n - Y^{n-1}$ and $Y_{(n-1)} = Y^{n-1} - k\mathbf{1}$.  
Accordingly $N_n^0 = $ degree zero eigencomponent of $N_n$ with
respect to ${\rm ad}\, Y_{(n-1)}$ is also defined over $K$.
Iterating this construction shows that each pair \eqref{eq:ds-2}
consists of elements of $\mathfrak g_K$.
\end{proof}

\begin{corollary} In the setting of Theorem \ref{thm:mt-rep}, if 
$Y_{(n)}$ and $N_1,\dots,N_n\in\mathfrak m_{\mathbb Q}$ then each
of the $\mathfrak{sl}_2$-pairs \eqref{eq:ds-2} consists of elements of 
$\mathfrak m_{\mathbb Q}$.
\end{corollary}
\begin{proof} The pairs \eqref{eq:ds-2} belong to 
$\mathfrak m_{\mathbb R}\cap\mathfrak g_{\mathbb Q}$.
\end{proof}

\section{Period domains} \label{S:pd}

\subsection{Period domains versus Mumford--Tate domains} \label{S:PDvMT}

The next three sections focus on the computation of polarized relations on $\Psi_D$ in three special cases.  In this section, we consider the period domain parameterizing all Hodge structures on $V$ polarized by $Q$ with Hodge numbers $\mathbf{h}=\{ h^{p,n-p}\}_{0\leq p\leq n}$. (To avoid some pathologies, we shall assume in the even weight case $n=2m$ that both $h^{m,m}$, and some other $h^{p,2m-p}$ with $p$ odd, are nonzero.) In this case it turns out that the conjugacy classes and relations introduced above may be enumerated ``hieroglyphically'' by Hodge diamonds:

\begin{i_list}
\item The conjugacy classes of nilpotent elements are classified by partially signed Young diagrams, and a theorem of \Dokovic's characterizes the partial order in terms of these diagrams; see \cite[\S2]{BPR} and the references therein.
\item It is implicit in the work of Deligne, Cattani, Kaplan and Schmid (and the representation theoretic classification of (i)) that the conjugacy classes of horizontal $\mathrm{SL}(2)$s are classified by the possible Hodge diamonds (Proposition \ref{P:hd} below).
\item Together (ii) and Theorem \ref{T:4} yield a relatively simple test to determine when a relation $\le$ on these $\mathrm{SL}(2)$s is polarized (Theorem \ref{T:PDpr} below).
\end{i_list}

\noindent However, there is a caveat:  in even weight, we must coarsen the equivalence relation.  Recall that an even-weight period domain is a union of \emph{two} connected components $$\tilde{D} = D \amalg D'$$ on which the \emph{full} orthogonal group $\tilde{G}(\bR)=\mathrm{O}(V_{\bR},Q)$ acts transitively. Actually $G(\bR)=\mathrm{SO}(V_{\bR},Q)$ acts transitively, and $G$ is the generic Mumford-Tate group,\footnote{$\tilde{G}$ is not connected as an algebraic group!} but (i)-(iii) only pertain to conjugacy classes for $\tilde{G}(\bR)$, as we shall see.  This is all in contrast to the general M-T domain setting treated in the rest of this paper, where our use of results from the literature (incl. \cite{KP2012, SL2}) require us to work with a connected domain, and with $\mathrm{SL}(2)$s, $\partial D$, and $\mathrm{Nilp}(\mathfrak{g}_{\bR})$ modulo the action of the identity connected component $G(\bR)^+$.

To define more precisely the the objects of study in this section, for even or odd weight, write $\tilde{G}:=\mathrm{Aut}(V,Q)$ with (algebraic) identity component $G:=\tilde{G}^{\circ}$, $\varphi: S^1 \to G(\bR)^+$ a Hodge structure on $V$ (with Hodge numbers $\mathbf{h}$) polarized by $Q$, and $$\check{D} =\tilde{G}(\bC).F_{\varphi} \supseteq \tilde{D} = \tilde{G}(\bR).\varphi \supseteq  D= G(\bR)^+ .\varphi .$$ (For odd weight, $\tilde{D}=D$, $\tilde{G}=G$, and $G(\bR)=G(\bR)^+=\tSp(r,\bR)$.) Denote by $\Psi_{\tilde{D}}$ the $G(\bR)^+$-conjugacy classes of $\bR$-split PMHS on $\tilde{D}$, and set $\pi (\Psi_{\tilde{D}})=:\bfN_{\tilde{D}}$, $\phi_{\infty}(\Psi_{\tilde{D}})=:\bfD_{\tilde{D}}$.  Then we may further quotient these objects by $\tilde{G}(\bR)/G(\bR)^+ \cong \mathbb{Z}/2\mathbb{Z} \times \mathbb{Z}/2\mathbb{Z}$ to define 

\begin{equation*}
\begin{tikzcd}
  & \overline{\Psi}_D \arrow[rd, two heads, "\phi_\infty"] \arrow[ld, two heads, "\pi"'] & \\
  \overline{\bfN}_D & & \overline{\bfD}_D .
\end{tikzcd}
\end{equation*}

When $n$ is odd, these are the same as the un-barred objects. For $n$ even, we have $\Psi_{\tilde{D}} = \Psi_D \amalg \Psi_{D'}$; and $G(\bR)/G(\bR)^+ \cong \mathbb{Z}/2\mathbb{Z}$ swaps $D$ and $D'$, giving an identification between $\Psi_D$ and $\Psi_{D'}$. So $\overline{\Psi}_D$ is the further quotient of $\Psi_D$  by the action of $\tilde{G}(\bR)/G(\bR)\cong \mathbb{Z}/2\mathbb{Z}$, and this quotient can indeed be nontrivial. We will encounter this phenomenon when $D$ admits a polarized mixed Hodge structure $(F,W(N))$ with the property that the $N$--strings are all of even length, and each length occurs with even multiplicity;\footnote{These PMHS correspond to the ``very even'' partitions in the classification \cite{BPR} of $\bfN$.} see Example \ref{eg:242}.

Before we turn to (ii) and (iii), here is a first glimpse of why $\tilde{G}(\bR)$-conjugacy classes are the natural object when $n$ is even. While $\phi_{\infty}$ gives bijections $\Psi_D \to \bfD_D$ and $\Psi_{D'} \to \bfD_{D'}$, the fact that $\bfD_{\tilde{D}} = \bfD_D \cup \bfD_{D'}$ need \emph{not} be a disjoint union means that $\Psi_{\tilde{D}} \twoheadrightarrow \bfD_{\tilde{D}}$ need \emph{not} be a bijection, and the problem may not be resolved after quotienting by $G(\bR)/G(\bR)^+$ (for example, if one has a pair of boundary orbits exchanged by $g\in G(\bR)\backslash G(\bR)^+$). However, in view of Corollary \ref{cor:barequiv} below, $\overline{\Psi}_D \to \overline{\bfD}_D$ is always a bijection.

\subsection{Hodge--Deligne numbers} \label{S:hd}

Let $(F,W(N))$ be a polarized mixed Hodge structure on $D$, and let $V_\bC = \op I^{p,q}$ denote the Deligne bigrading.  The \emph{Deligne--Hodge numbers} of the PMHS are 
\[
  i^{p,q} \ := \ \tdim_\bC\,I^{p,q}\,.
\] 
It is sometimes convenient to view these dimensions as giving a function
\[
  \Diamond(F,N) : \bZ \times \bZ \ \to \ \bZ_{\ge0} 
  \quad\hbox{sending}\quad
  (p,q) \ \mapsto \ i^{p,q} \,.
\]
We call the function $\Diamond(F,N)$ the \emph{Hodge diamond} of $(F,N)$.  Recall that the $p$--th column of the Hodge diamond must sum to $h^{p,n-p}$; that is,
\begin{subequations}\label{SE:hd}
\begin{equation}
  \sum_q i^{p,q} \ = \ h^{p,n-p} \,.
\end{equation} 
The Hodge diamond is also symmetric about the lines $p=q$ and $p+q=n$: that is, 
\begin{equation} \label{E:hd2}
  i^{p,q} \ = \ i^{q,p} \ = \ i^{n-q,n-p}\,.
\end{equation}
Moreover, the Hodge--Deligne numbers must be nondecreasing as they approach the diagonal $p+q=n$ along a line $p-q = k$, with $-n\le k \le n$ fixed.  That is, 
\begin{equation}
  i^{p-1,q-1} \ \le \ i^{p,q} \quad\hbox{if}\quad
  p+q \,\le\, n \,.
\end{equation}
\end{subequations}
(By the symmetry \eqref{E:hd2} this implies $i^{p,q} \ge i^{p+1,q+1}$ if $p+q \ge n$.)  

The following proposition asserts that the possible Hodge diamonds enumerate the elements of $\overline{\Psi}_D$.

\begin{proposition} \label{P:hd}
Any function $f: \bZ \times \bZ \to \bZ_{\ge0}$ satisfying \eqref{SE:hd} may be realized as the Hodge diamond $\Diamond(F,N)$ of an $\bR$--split polarized mixed Hodge structure $(F,W(N))$ on the period domain $D$.  Moreover, $\Diamond(F_1,N_1) = \Diamond(F_2,N_2)$ if and only if $[F_1,N_1] = [F_2,N_2] \in \overline{\Psi}_D$.
\end{proposition}

\noindent As demonstrated in Example \ref{eg:242} (see also the examples of \cite{SL2}), Proposition \ref{P:hd} is in general \emph{false} for $\Psi_D$.  The proof of the proposition makes use of the notion of primitive subspaces; these are introduced in \S\ref{S:prim}, and the proof is given on page \pageref{prf:hd}.

\begin{corollary}\label{cor:barequiv}
$\phi_{\infty}$ induces a bijection from $\overline{\Psi}_D$ to $\overline{\bfD}_D$.
\end{corollary}
\begin{proof}
Let $\mathbf{H}_D$ denote the set of Hodge diamonds of LMHS/naive limits; by Prop. \ref{P:hd}, these are just the functions $\Diamond = \{ i^{p,q}\}$ satisfying \eqref{SE:hd}.  Construct a map $\mathtt{h}: \overline{\bfD}_D \to \mathbf{H}_D$ by sending a representative flag $F$ on $V$ to the function 
\[
  \mathtt{h}_F (p,q)\ := \ \dim \left( 
  \frac{F^p \cap \overline{F^q}}{F^p \cap \overline{F^{q+1}} + F^{p+1}\cap \overline{F^q}} 
  \right)\,.
\] 
This is evidently well-defined.  If $(F,W)$ is $\bR$-split, then $\mathtt{h}_F$ is precisely its Hodge diamond.

Now given an $\bR$-split PMHS $(F,N)$ with corresponding $\mathfrak{sl}_2$-triple $(N,Y,N^+)$, homomorphism $\rho: \mathrm{SL}_2(\bR)\to G(\bR)^+$, and naive limit $\hat{F} := \lim_{y\to \infty} e^{iyN}F$, $(\hat{F},-N^+)$ is an $\bR$-split PMHS representing the same element of $\Psi_D$ (hence $\overline{\Psi}_D$).  Indeed, this is just the image of $(F,N)$ by $\rho\left( \left( \tiny \begin{matrix} 0 & 1 \\ -1 & 0 \end{matrix} \right) \right) .$ We therefore have $\Diamond(F,N) = \Diamond(\hat{F},-N^+)=\mathtt{h}_{\hat{F}} = (\mathtt{h}\circ \phi_{\infty})(F,N)$, which is to say that the composite $$\overline{\Psi}_D \overset{\phi_{\infty}}{\twoheadrightarrow} \overline{\bfD}_D \overset{\mathtt{h}}{\to} \mathbf{H}_D$$ yields the bijection $\Diamond$ of Prop. \ref{P:hd}.  This forces $\phi_{\infty}$ to be a bijection.
\end{proof}

We finish this subsection by illustrating Prop. \ref{P:hd}.  In all the examples that follow, $\pi: \overline{\Psi}_D\to\overline{\bfN}_D$ is seen to be an isomorphism (as the map from possible Hodge diamonds to partially signed Young diagrams is one-to-one).  For a situation where this is not the case, see Example \ref{eg:borel7}.  (Because this example is of odd weight, of course $\overline{\Psi}_D = \Psi_D$ and $\overline{\bfN}_D = \bfN_D$.)

\begin{example}[Curves of genus $g$] \label{eg:curve1}
Suppose that $\bh = (g,g)$.  The possible Hodge diamonds, which we denote $\mathrm{I}_a = \Diamond(F_a,N_a)$, are 
\begin{center}
\setlength{\unitlength}{20pt}
\begin{picture}(6.5,2)
\put(0,0){\vector(1,0){2}} \put(0,0){\vector(0,1){2}}
\put(0,0){\circle*{0.2}} \put(-0.35,0.2){\footnotesize{$a$}}
\put(1,0){\circle*{0.2}} \put(1.2,0.2){\footnotesize{$g-a$}}
\put(0,1){\circle*{0.2}} \put(-1.2,1.2){\footnotesize{$g-a$}}
\put(1,1){\circle*{0.2}} \put(1.2,1.2){\footnotesize{$a$}}
\put(4,0.5){$0\le a \le g$.}
\end{picture}
\end{center}
The partially signed Young diagram classifying the conjugacy class $\cN_a \in \bfN$ of $N_a$ is 
\begin{center}
\begin{small}
\setlength{\unitlength}{10pt}
\begin{picture}(10,9)
  \put(0,7){\scriptsize{\young(+-)}}
  \put(0.8,5.5){$\vdots$}
  \put(0,3){\scriptsize{\young(+-,\blank)}}
  \put(0.3,1.5){$\vdots$}
  \put(0,0){\scriptsize{\young(\blank)}}
  \put(2.5,5.7){$\Bigg\}{a \ \mathrm{rows}}$}
  \put(2.5,1.6){$\Bigg\}{g-a \ \mathrm{rows}}$}
\end{picture}
\end{small}
\end{center}
See \cite[\S2.5]{BPR} for a discussion of how one obtains the Young diagram from the Hodge diamond.  Let $\cN_a = \pi(\mathrm{I}_a) \in \bfN_D$.  \Dokovic's \cite[Theorem 2.21]{BPR} yields 
\[
  \cN_a \ < \ \cN_b
  \quad\hbox{if and only if} \quad
  a \ < \ b \,.
\]
Likewise, setting $\cO_a = \phi_\infty(\mathrm{I}_a) \in \bfD_D$, we have 
\[
  \cO_a \ < \ \cO_b
  \quad\hbox{if and only if} \quad
  a \ < \ b \,; 
\]
see \S\ref{S:hs}.
\end{example}

\begin{example}[K3 surfaces] \label{eg:K31}
Suppose that $\bh = (1,m,1)$, with $m\geq 1$.  The three possible Hodge diamonds are depicted below.  (The first is the trivial $[D,0] \in \overline{\Psi}_D$.)  In this example we omit the labels $i^{p,q}$ as they may easily be deduced from the Hodge numbers $\bh$ and \eqref{SE:hd}.
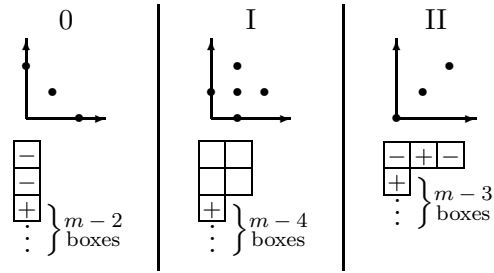
\begin{center}
\begin{small}
\setlength{\unitlength}{10pt}
\begin{tabular}{c|c|c}
  0 & I & II \\
    \begin{picture}(5,3)(-1,0)
    \put(0,0){\vector(1,0){3}} \put(0,0){\vector(0,1){3}}  
    \put(0,2){\circle*{0.25}} 
    \put(1,1){\circle*{0.25}} 
    \put(2,0){\circle*{0.25}} 
    \end{picture}
  & 
    \begin{picture}(5,3)(-1,0)
    \put(0,0){\vector(1,0){3}} \put(0,0){\vector(0,1){3}}  
    \put(0,1){\circle*{0.25}}  
    \put(1,0){\circle*{0.25}} \put(1,1){\circle*{0.25}} 
    \put(1,2){\circle*{0.25}} 
    \put(2,1){\circle*{0.25}} \end{picture}
  & 
    \begin{picture}(5,3)(-1,0)
    \put(0,0){\vector(1,0){3}} \put(0,0){\vector(0,1){3}}
    \put(0,0){\circle*{0.25}} 
    \put(1,1){\circle*{0.25}} 
    \put(2,2){\circle*{0.25}} \end{picture}
  \\ 
    \begin{picture}(6,5)(-1,0)
    \put(0,1.5){\scriptsize{\young(-,-,+)}}
    \put(0.4,0.35){$\vdots$}  
    \put(1.2,0.9){$\Big\}\mystack{m-2}{\mathrm{boxes}}$}
    \end{picture}
  &
    \begin{picture}(6,5)(-1,0)
    \put(0,1.5){\scriptsize{\young(\blank\blank,\blank\blank,+)}}
    \put(0.4,0.35){$\vdots$}  
    \put(1.2,0.9){$\Big\}\mystack{m-4}{\mathrm{boxes}}$}
    \end{picture}
  &
    \begin{picture}(6,5)(-1,-1)
    \put(0,1.5){\scriptsize{\young(-+-,+)}}
    \put(0.4,0.35){$\vdots$}  
    \put(1.2,0.9){$\Big\}\mystack{m-3}{\mathrm{boxes}}$}
    \end{picture}
\end{tabular}
\end{small}
\end{center}
The second row of the table depicts the partially signed Young diagram classifying the conjugacy class $\cN \in \overline{\bfN}_D$ of $N$.  \Dokovic's \cite[Theorem 2.21]{BPR} yields $\cN_0 < \cN_\mathrm{I} < \cN_\mathrm{II}$.   If $m=1$ then there is no type `I'.  (Remark that ``0,I,II'' in our nomenclature correspond to types ``I,II,III'' in Kulikov's classification.)
\end{example}

\begin{example}[Surfaces with contact IPR] \label{eg:H1}
Suppose that $\bh = (2,m,2)$ with $m\ge 4$.  (In this case the horizontal subbundle $T^hD \subset TD$ is a contact distribution.  In particular, the horizontal subbundle is of corank one, and so very close to the classical Hermitian symmetric case in which $T^hD = TD$.)  Then the six possible Hodge diamonds are listed in the first row of the table below.  (The first is the trivial $[D,0] \in \overline{\Psi}_D$.)  Again, we omit the labels $i^{p,q}$ as they may easily be deduced from the Hodge numbers $\bh$ and \eqref{SE:hd}.
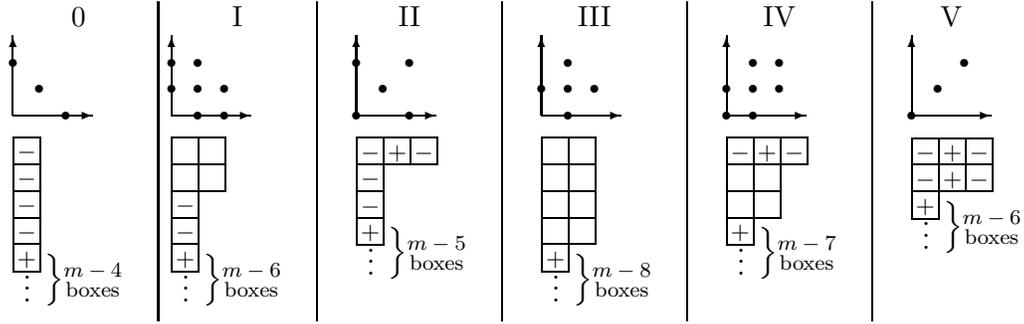
\begin{center}
\begin{small}
\setlength{\unitlength}{10pt}
\begin{tabular}{c|c|c|c|c|c}
  0 & I & II & III & IV & V \\
    \begin{picture}(5,3)(0,0)
    \put(0,0){\vector(1,0){3}} \put(0,0){\vector(0,1){3}}  
    \put(0,2){\circle*{0.25}} \put(1,1){\circle*{0.25}} 
    \put(2,0){\circle*{0.25}} \end{picture}
  & 
    \begin{picture}(5,3)(0,0)
    \put(0,0){\vector(1,0){3}} \put(0,0){\vector(0,1){3}}  
    \put(0,1){\circle*{0.25}} \put(0,2){\circle*{0.25}} 
    \put(1,0){\circle*{0.25}} \put(1,1){\circle*{0.25}} 
    \put(1,2){\circle*{0.25}} \put(2,0){\circle*{0.25}} 
    \put(2,1){\circle*{0.25}} \end{picture}
  & 
    \begin{picture}(6,3)(-1,0)
    \put(0,0){\vector(1,0){3}} \put(0,0){\vector(0,1){3}}
    \put(0,0){\circle*{0.25}} \put(0,2){\circle*{0.25}}
    \put(1,1){\circle*{0.25}} \put(2,0){\circle*{0.25}}
    \put(2,2){\circle*{0.25}} \end{picture}
  & 
    \begin{picture}(6,3)(-1,0)
    \put(0,0){\vector(1,0){3}} \put(0,0){\vector(0,1){3}}
    \put(0,1){\circle*{0.25}} \put(1,0){\circle*{0.25}}
    \put(1,1){\circle*{0.25}} \put(1,2){\circle*{0.25}}
    \put(2,1){\circle*{0.25}} \end{picture}
  & 
    \begin{picture}(6,3)(-1,0)
    \put(0,0){\vector(1,0){3}} \put(0,0){\vector(0,1){3}}
    \put(0,0){\circle*{0.25}} \put(0,1){\circle*{0.25}}
    \put(1,0){\circle*{0.25}} \put(1,1){\circle*{0.25}}
    \put(1,2){\circle*{0.25}} \put(2,1){\circle*{0.25}}
    \put(2,2){\circle*{0.25}}  \end{picture}
  & 
    \begin{picture}(5,3)(-1,0)
    \put(0,0){\vector(1,0){3}} \put(0,0){\vector(0,1){3}}
    \put(0,0){\circle*{0.25}} \put(1,1){\circle*{0.25}} 
    \put(2,2){\circle*{0.25}} \end{picture}
  \\ 
    \begin{picture}(5,7)
    \put(0,1.5){\scriptsize{\young(-,-,-,-,+)}}
    \put(0.4,0.35){$\vdots$}  
    \put(1.2,0.9){$\Big\}\mystack{m-4}{\mathrm{boxes}}$}
    \end{picture}
  &
    \begin{picture}(5,7)
    \put(0,1.5){\scriptsize{\young(\blank\blank,\blank\blank,-,-,+)}}
    \put(0.4,0.35){$\vdots$}  
    \put(1.2,0.9){$\Big\}\mystack{m-6}{\mathrm{boxes}}$}
    \end{picture}
  &
    \begin{picture}(6,7)(-1,-1)
    \put(0,1.5){\scriptsize{\young(-+-,-,-,+)}}
    \put(0.4,0.35){$\vdots$}  
    \put(1.2,0.9){$\Big\}\mystack{m-5}{\mathrm{boxes}}$}
    \end{picture}
  & 
    \begin{picture}(6,7)(-1,0)
    \put(0,1.5){\scriptsize{\young(\blank\blank,\blank\blank,\blank\blank,\blank\blank,+)}}
    \put(0.4,0.35){$\vdots$} 
    \put(1.2,0.9){$\Big\}\mystack{m-8}{\mathrm{boxes}}$}
    \end{picture} 
  &     
    \begin{picture}(6,7)(-1,-1)
    \put(0,1.5){\scriptsize{\young(-+-,\blank\blank,\blank\blank,+)}}
    \put(0.4,0.35){$\vdots$}  
    \put(1.2,0.9){$\Big\}\mystack{m-7}{\mathrm{boxes}}$}
    \end{picture}
  &     
    \begin{picture}(5,7)(-1,-2)
    \put(0,1.5){\scriptsize{\young(-+-,-+-,+)}}
    \put(0.4,0.35){$\vdots$}  
    \put(1.2,0.9){$\Big\}\mystack{m-6}{\mathrm{boxes}}$}
    \end{picture}
\end{tabular}
\end{small}
\end{center}
The second row of the table depicts the partially signed Young diagram classifying the conjugacy class $\cN \in \overline{\bfN}_D$ of $N$.   \Dokovic's \cite[Theorem 2.21]{BPR} yields 
\[
  0 \ < \ 
  \cN_\mathrm{I} \ < \ 
  \left\{ \begin{array}{c} 
       \cN_\mathrm{II} \\ \cN_\mathrm{III}
  \end{array}\right\} \ < \ 
  \cN_\mathrm{IV} \ < \ 
  \cN_\mathrm{V}\,.
\]

\end{example}

\begin{example}[$\bh = (2,4,2)$] \label{eg:242}
Here we specialize Example \ref{eg:H1} to $m=4$.  In order to illustrate our assertion that Prop. \ref{P:hd} can fail for $\Psi_D$, we consider the pairs $(F,N)$ modulo the connected identity component $\tSO(4,4)^+ \subsetneq \tO(4,4)$.  In this case the conjugacy classes are classified by \eqref{E:FNvl}, \emph{not} Proposition \ref{P:hd}.  We find that there are six nontrivial conjugacy classes, rather than the five of Example \ref{eg:H1}: the $\tO(4,4)$--conjugacy class $\mathrm{III}$ splits into two $\tSO(4,4)^+$--conjugacy classes.

To describe these six conjugacy classes, we follow the notation of \S\ref{S:rt}.   Then $\ttE_\varphi = \ttS^2$.  This implies that the the Weyl subgroup $\sW^0 \subset \sW$ is generated by the simple reflections $\{(1),(3),(4)\}$.  We will let $\sS'$ denote the simple roots of a representative $\fl_\bR$ of the element $[\fl_\bR]\in\Lambda_{\varphi,\ft}$ indexing the conjugacy class.  Finally, we employ the short-hand $(a_1,\ldots,a_4)$ to indicate the distinguished $\sZ = \pi^\tss_\fl(\ttE_\varphi) = a_1\ttS^1+\cdots+a_4\ttS^4$.  Then the six nontrivial $\tSO(4,4)^+$--conjugacy classes  $[F,N]$ are 
\begin{center}
\begin{tabular}{c||c|c|c}
  HD & $\fl_\bR^\tss$ & $\sZ$ & $\sS'$ \\ \hline
  I & $\fsu(1,1)$ & $(-1,2,-1,-1)$ & $\{\a_2\}$ \\
  II & $\fsu(1,1)\op\fsu(1,1)$ & $(-2,2,0,0)$ 
    & $\{ \a_2 \,,\, \a_2+\a_3+\a_4\}$ \\
  III & $\fsu(1,1)\op\fsu(1,1)$ & $(0,2,0,-2)$ 
    & $\{ \a_2 \,,\, \a_1+\a_2+\a_3\}$ \\
  III & $\fsu(1,1)\op\fsu(1,1)$ & $(0,2,-2,0)$ 
    & $\{ \a_2 \,,\, \a_1+\a_2+\a_4\}$ \\
  IV & $\fsu(1,1)^{\op3}$ & $(-1,2,1,-1)$ 
    & $\{ \a_2 \,,\, \a_1+\a_2+\a_3 \,,\, \a_2+\a_3+\a_4\}$ \\
  V & $\fsu(2,1)$ & $2\ttE_\varphi$ 
    & $\{\a_2\,,\,\a_1+\cdots+\a_4\}$
\end{tabular}
\end{center}
\end{example}

\begin{example}[Calabi--Yau 3-folds] \label{eg:CY1}
Suppose that $\bh = (1,m,m,1)$.  There are $4m$ possible Hodge diamonds, including the trivial one; they are listed in the first row of the table below.  
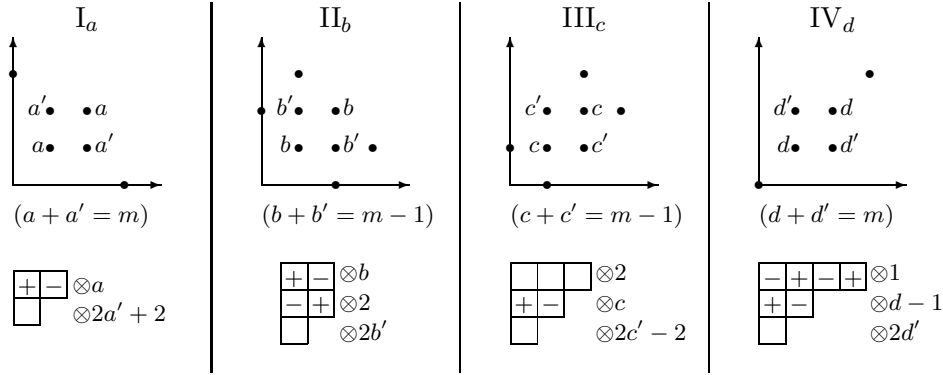
\begin{center}
\begin{small}
\setlength{\unitlength}{14pt}
\begin{tabular}{c|c|c|c}
  $\mathrm{I}_a$ & $\mathrm{II}_b$ & $\mathrm{III}_c$ & $\mathrm{IV}_d$ \\
    \begin{picture}(6,4)(-1,0)
    \put(0,0){\vector(1,0){4}} \put(0,0){\vector(0,1){4}}  
    \put(0,3){\circle*{0.2}} 
    \put(1,2){\circle*{0.2}} \put(0.4,1.9){\scriptsize{$a'$}}
    \put(1,1){\circle*{0.2}} \put(0.5,0.9){\scriptsize{$a$}}
    \put(2,1){\circle*{0.2}} \put(2.2,0.9){\scriptsize{$a'$}}
    \put(2,2){\circle*{0.2}} \put(2.2,1.9){\scriptsize{$a$}}
    \put(3,0){\circle*{0.2}} 
    \put(0,-1){\scriptsize{$(a+a'=m)$}}
    \end{picture}
  & 
    \begin{picture}(6,4)(-1,0)
    \put(0,0){\vector(1,0){4}} \put(0,0){\vector(0,1){4}}
    \put(0,2){\circle*{0.2}} 
    \put(1,1){\circle*{0.2}} \put(0.5,0.9){\scriptsize{$b$}}
    \put(1,2){\circle*{0.2}} \put(0.4,1.9){\scriptsize{$b'$}} 
    \put(1,3){\circle*{0.2}}
    \put(2,0){\circle*{0.2}} 
    \put(2,1){\circle*{0.2}} \put(2.2,0.9){\scriptsize{$b'$}}
    \put(2,2){\circle*{0.2}} \put(2.2,1.9){\scriptsize{$b$}}
    \put(3,1){\circle*{0.2}} 
    \put(0,-1){\scriptsize{$(b+b'=m-1)$}}
    \end{picture}
  & 
    \begin{picture}(6,4)(-1,0)
    \put(0,0){\vector(1,0){4}} \put(0,0){\vector(0,1){4}}
    \put(0,1){\circle*{0.2}} 
    \put(1,1){\circle*{0.2}} \put(0.5,0.9){\scriptsize{$c$}}
    \put(2,1){\circle*{0.2}} \put(0.4,1.9){\scriptsize{$c'$}} 
    \put(1,0){\circle*{0.2}}
    \put(2,3){\circle*{0.2}} 
    \put(1,2){\circle*{0.2}} \put(2.2,0.9){\scriptsize{$c'$}}
    \put(2,2){\circle*{0.2}} \put(2.2,1.9){\scriptsize{$c$}}
    \put(3,2){\circle*{0.2}} 
    \put(0,-1){\scriptsize{$(c+c'=m-1)$}}
    \end{picture}
  & 
    \begin{picture}(6,4)(-1,0)
    \put(0,0){\vector(1,0){4}} \put(0,0){\vector(0,1){4}}
    \put(0,0){\circle*{0.2}} 
    \put(1,1){\circle*{0.2}} \put(0.5,0.9){\scriptsize{$d$}}
    \put(1,2){\circle*{0.2}} \put(0.4,1.9){\scriptsize{$d'$}}
    \put(2,1){\circle*{0.2}} \put(2.2,0.9){\scriptsize{$d'$}}
    \put(2,2){\circle*{0.2}} \put(2.2,1.9){\scriptsize{$d$}}
    \put(3,3){\circle*{0.2}}
    \put(0,-1){\scriptsize{$(d+d'=m)$}}
    \end{picture}
  \\ 
    \begin{picture}(3,4.5)(0.5,-1)
    \put(0,0){\scriptsize{\young(+-,\blank)}}
    \put(1.6,0.9){\scriptsize{$\ot a$}}
    \put(1.6,0.1){\scriptsize{$\ot 2a'+2$}}
    \end{picture}
  &
    \begin{picture}(3,3)(0,-0.5)
    \put(0,0){\scriptsize{\young(+-,-+,\blank)}}
    \put(1.6,1.7){\scriptsize{$\ot b$}}
    \put(1.6,0.9){\scriptsize{$\ot 2$}}
    \put(1.6,0.1){\scriptsize{$\ot 2b'$}}
    \end{picture}
  & 
    \begin{picture}(5,3)(-0.5,-0.5)
    \put(0,0){\scriptsize{\young(\blank\blank\blank,+-,\blank)}}
    \put(2.3,1.7){\scriptsize{$\ot 2$}}
    \put(2.3,0.9){\scriptsize{$\ot c$}}
    \put(2.3,0.1){\scriptsize{$\ot 2c'-2$}}
    \end{picture} 
  &     
    \begin{picture}(5,3)(-0.5,-0.5)
    \put(0,0){\scriptsize{\young(-+-+,+-,\blank)}}
    \put(3,1.7){\scriptsize{$\ot 1$}}
    \put(3,0.9){\scriptsize{$\ot d-1$}}
    \put(3,0.1){\scriptsize{$\ot 2d'$}}
    \end{picture} 
\end{tabular}
\end{small}
\end{center} 
Above we have $0 \le a \le m$, $0 \le b \le m-1$, $0 \le c \le m-2$ and $1 \le d \le m$.  The second row of the table depicts the partially signed Young diagram classifying the conjugacy class $\cN \in \bfN$ of $N$.   \Dokovic's \cite[Theorem 2.21]{BPR} yields 
\begin{eqnarray*}
  \cN_{\mathrm{I}_a} < \cN_{\mathrm{I}_b} \,,\
  \cN_{\mathrm{II}_a} < \cN_{\mathrm{II}_b} \,,\
  \cN_{\mathrm{III}_a} < \cN_{\mathrm{III}_b}\,,\
  \cN_{\mathrm{IV}_a} < \cN_{\mathrm{IV}_b}
  & \hbox{if and only if} & a < b \,;\\
  \cN_{\mathrm{I}_a} < \cN_{\mathrm{II}_b}
  & \hbox{if and only if} & a \le b \,;\\
  \cN_{\mathrm{I}_a} , \cN_{\mathrm{II}_a} < \cN_{\mathrm{III}_b}
  & \hbox{if and only if} & a \le b+2 \,;\\
  \cN_{\mathrm{I}_a} , \cN_{\mathrm{II}_a} < \cN_{\mathrm{IV}_b}
  & \hbox{if and only if} & a \le b \,;\\
  \cN_{\mathrm{III}_a}  < \cN_{\mathrm{IV}_b}
  & \hbox{if and only if} & a+2 \le b \,.
\end{eqnarray*}
\end{example}

\begin{remark}
The enumeration of $\Psi_D$ (or a quotient thereof) by ``numerically admissible'' Hodge diamonds is particular to period domains.  Besides the greater constraints on degenerations in a M-T domain, there is the fact that a general such domain may have no ``standard representation'' (as for exceptional groups).  Even when $G$ is classical, and a M-T group, the standard representation may not be realizable as a Hodge representation (e.g. $\tSp(a,b)$ \cite{MR2918237}).  
\end{remark}

\subsection{Primitive subspaces} \label{S:prim}

The \emph{$N$--primitive subspace} $P(N)$ of $V$ is defined by 
\[
  P(N) \ := \  \bigoplus_{k=0}^n P(N)_{n+k} \,,
\]
where
\begin{equation}\label{E:Ppq}
\renewcommand{\arraystretch}{1.3}
\begin{array}{rcl}
  P(N)_{n+k,\bC} & := & 
  \displaystyle\bigoplus_{p+q=n+k} P(N)^{p,q} \,,\quad\hbox{and}\\
  P(N)^{p,q} & := & \tker\{ N^{k+1} : I^{p,q} \to I^{-p-1,-q-1} \} \,.
\end{array}
\end{equation}
Recall that the \emph{weight $n+k$ $N$--primitive subspace} $P(N)_{n+k}$ is defined over $\bR$, and 
\begin{equation}\label{E:Ndecomp}
  V \ = \ \bigoplus_{\mystack{0\le k}{0 \le a \le k}} N^a P(N)_{n+k} \,.
\end{equation}
In particular, the decomposition \eqref{E:Ppq} determines the Deligne bigrading $V_\bC = \op\,I^{p,q}$ of $(F,W(N))$.  Moreover,  \eqref{E:Ppq} is a weight $n+k$ Hodge decomposition of $P(N)_{k,\bQ}$ polarized by 
\[
  Q^N_k(\cdot,\cdot) \ := \ Q(\cdot , N^k\cdot)\,.
\]
The \emph{$N$--primitive Hodge--Deligne numbers} are the 
\[
  j^{p,q} \ := \ \tdim_\bC\,P(N)^{p,q}\,.
\]
The \emph{weight $n+k$ primitive part} of $\Diamond(F,N)$ is the function
\[
  \Diamond^\tprim_{n+k}(F,N) : \bZ \times \bZ \ \to \ \bZ_{\ge0}
  \quad\hbox{sending}\quad
  (p,q) \ \mapsto \ j^{p,q} \, ,
\]
which is supported on $p+q=n+k$.  The \emph{primitive part} of $\Diamond(F,N)$ is the sum
\[
  \Diamond^\tprim(F,N) \ = \ \sum_{k=0}^n
  \Diamond^\tprim_{n+k}(F,N)
\] 
of the weight $n+k$ primitive parts.  We will call any such $\Diamond^\tprim(F,N)$ a \emph{primitive sub-diamond for the period domain $D$}.\footnote{Note that it is not, in fact, a diamond:  \eqref{SE:hd} will fail whenever $N\neq 0$.}  From \eqref{E:Ndecomp} we see that 
\begin{equation}\label{E:pHD}
  \hbox{\emph{$\Diamond^\tprim(F,N)$ determines 
  $\Diamond(F,N)$ (and visa versa).}}
\end{equation}
To be more precise, given $f : \bZ \times \bZ \to \bZ_{\ge0}$ define $f[k,k]: \bZ \times \bZ \to \bZ_{\ge0}$ by $(p,q) \mapsto f(p+k,q+k)$.  Then \eqref{E:Ndecomp} implies
\begin{equation}\nonumber
  \Diamond(F,N) \ = \ \sum_{\mystack{0 \le k}{0 \le a\le k}} 
  \Diamond^\tprim_{n+k}(F,N)[a,a] \,.
\end{equation}
From Proposition \ref{P:hd} we then obtain

\begin{corollary}\label{C:hd}
The conjugacy class $[F,N] \in \Psi_D$ is determined by the primitive sub-diamond $\Diamond^\tprim(F,N)$.
\end{corollary}

\begin{example}[Surfaces with contact IPR] \label{eg:H2}
The primitive sub-diamonds for the five nontrivial $[F,N] \in \Psi_D$ of Example \ref{eg:H1} are depicted below.
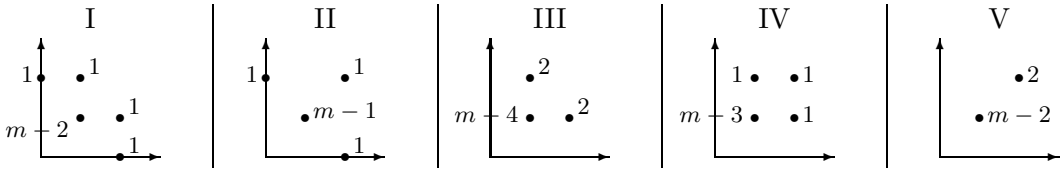
\begin{center}
\begin{small}
\setlength{\unitlength}{15pt}
\begin{tabular}{c|c|c|c|c}
  I & II & III & IV & V \\
    \begin{picture}(5.5,3)(-1.5,0)
    \put(0,0){\vector(1,0){3}} \put(0,0){\vector(0,1){3}}  
    \put(0,2){\circle*{0.2}} \put(-0.5,1.9){\scriptsize{$1$}}
    \put(1,1){\circle*{0.2}} \put(-0.9,0.5){\scriptsize{$m-2$}}
    \put(1,2){\circle*{0.2}} \put(1.2,2.1){\scriptsize{$1$}}
    \put(2,0){\circle*{0.2}} \put(2.2,0.1){\scriptsize{$1$}}
    \put(2,1){\circle*{0.2}} \put(2.2,1.1){\scriptsize{$1$}}
    \end{picture}
  & 
    \begin{picture}(5,3)(-1,0)
    \put(0,0){\vector(1,0){3}} \put(0,0){\vector(0,1){3}}
    \put(0,2){\circle*{0.2}} \put(-0.5,1.9){\scriptsize{$1$}}
    \put(1,1){\circle*{0.2}} \put(1.2,1.0){\scriptsize{$m-1$}}
    \put(2,0){\circle*{0.2}} \put(2.2,2.1){\scriptsize{$1$}}
    \put(2,2){\circle*{0.2}} \put(2.2,0.1){\scriptsize{$1$}}
    \end{picture}
  & 
    \begin{picture}(5,3)(-1,0)
    \put(0,0){\vector(1,0){3}} \put(0,0){\vector(0,1){3}}
    \put(1,1){\circle*{0.2}} \put(-0.9,0.9){\scriptsize{$m-4$}}
    \put(1,2){\circle*{0.2}} \put(2.2,1.1){\scriptsize{$2$}}
    \put(2,1){\circle*{0.2}} \put(1.2,2.1){\scriptsize{$2$}}
    \end{picture}
  & 
    \begin{picture}(5,3)(-1,0)
    \put(0,0){\vector(1,0){3}} \put(0,0){\vector(0,1){3}}
    \put(1,2){\circle*{0.2}} \put(0.4,1.9){\scriptsize{$1$}}
    \put(1,1){\circle*{0.2}} \put(-0.9,0.9){\scriptsize{$m-3$}}
    \put(2,1){\circle*{0.2}} \put(2.2,0.9){\scriptsize{$1$}}
    \put(2,2){\circle*{0.2}} \put(2.2,1.9){\scriptsize{$1$}}
    \end{picture}
  & 
    \begin{picture}(5,3)(-1,0)
    \put(0,0){\vector(1,0){3}} \put(0,0){\vector(0,1){3}}
    \put(1,1){\circle*{0.2}} \put(1.2,0.9){\scriptsize{$m-2$}}
    \put(2,2){\circle*{0.2}} \put(2.2,1.9){\scriptsize{$2$}}
    \end{picture}
\end{tabular}
\end{small}
\end{center}
\end{example}

\begin{proof}[Proof of Proposition \ref{P:hd}] \label{prf:hd}
Given a function $f: \bZ \times \bZ \to \bZ_{\ge0}$ satisfying \eqref{SE:hd}, the pair $(F,N)$ may be constructed as follows.  For convenience, and without loss of generality, we assume that $n\ge0$ and that $f(p,q) \not=0$ only when $0 \le p,q \le n$.  Fix a direct sum decomposition $V_\bC = \op\,I^{p,q}_0$ so that $\tdim_\bC\,I^{p,q}_0 = f(p,q)$ and $\overline{I^{p,q}_0} = I^{q,p}_0$ for all $p,q$.  Then \eqref{SE:hd} implies one may define a nilpotent $N_0 \in \tEnd(V_\bR)$ so that $N(I^{p,q}_0) \subset I^{p-1,q-1}_0$ for all $p,q$, and so that
\[
  N^{p+q-n}_0 : I^{p,q}_0 \ \to \ I^{n-q,n-p}_0
\]
is an isomorphism for all $p+q\ge n$.

Suppose that $Q$ is symmetric ($n$ is even) and of signature $(a,b)$ over $\bR$.  Then it follows from \eqref{SE:hd} and the classification of nilpotent elements in $\fso(a,b)$ (see \cite[\S2]{BPR}), that there exists a symmetric, bilinear form $Q_0$ on $V_\bR$ of signature $(a,b)$ such that $N_0 \in \tEnd(V_\bR,Q_0)$.  More precisely, given a basis $\{v^{p,q}_i\}$ ($\overline{v^{p,q}_i}=v^{q,p}_i$) for each $P(N_0)^{p,q}:=\ker(N_0^{p+q-n+1})\subset I^{p,q}_0$, we can set 
\[
  Q_0 ( N^j v^{p,q}_i ,N^{p+q-n-j}\overline{v^{p',q'}_i} ) \ := \ 
  \mathbf{i}^{q-p+2j} \delta_{kk'}\delta_{pp'}\delta_{qq'} \,.
\]
It is then the case that $Q_0$ is conjugate to $Q$ under the action of some $g \in \tAut(V_\bR)$.  Set $I^{p,q} = g(I^{p,q}_0)$ and $N = \tAd_g(N_0)$.  Then $V_\bC = \op\,I^{p,q}$ is the Deligne bigrading of an $\bR$--split mixed Hodge structure $(F,W)$ that is polarized by $N$.  The symplectic case works in the same way.  This establishes the first assertion of the proposition.  

For the second assertion, let $V_\bC = \op I^{p,q}_i$ denote the respective Deligne splittings of $(F_i,W(N_i))$ for $i=1,2$.  It is clear that equality of the equivalence classes $[F_i,N_i]$ implies equality of the Hodge diamonds $\Diamond(F_i,N_i)$, since any $g \in G(\bR)$ with the property $g \cdot (F_1,N_1) = (F_2,N_2)$ will satisfy $g(I^{p,q}_1) = I^{p,q}_2$.  We will establish the converse by constructing an explicit $g \in \tAut(V_\bR,Q)$ with the properties that $\tAd_gN_1 = N_2$ and $g(I^{p,q}_1) = I^{p,q}_2$.  The latter implies $F_2 = g \cdot F_1$, so  that $(F_2 , N_2) = g\cdot(F_1 , N_1)$, completing the proof. 

By \eqref{E:pHD}, the hypothesis $\Diamond(F_1,N_1) = \Diamond(F_2,N_2)$ is equivalent to equality $j^{p,q}_1 = j^{p,q}_2$ of the primitive Hodge numbers.  Given $p\ge q$ with $p+q = n+k$, fix bases $\sB^{p,q}_i = \{ v^{p,q}_{i,1} , \ldots , v^{p,q}_{i,j_{p,q}} \}$ of the primitive spaces $P(N_i)^{p,q}$ so that $Q^{N_i}_k( C v^{p,q}_{i,a} , \overline{v^{p,q}_{i,b}}) = \d_{ab}$.  (Here $C$ denotes the Weil operator on $P(N_i)_{n+k}$.)   If $p=q$, then we may assume that $\sB^{p,p}_i \subset V_\bR$ is real.  Then \eqref{E:Ndecomp} implies
\[
  \bigcup_{k\ge0} \bigcup_{\mystack{0 \le \ell \le k}{p+q = n+k}}
  N^\ell_i \left\{ \sB^{p,q}_i \,\cup\, \overline{\sB^{p,q}_i} \right\}
\]
are bases of $V_\bC$, $i=1,2$.  So we may define $g \in \tAut(V_\bC)$ by 
\[
  g\left(N^\ell_1 \, v^{p,q}_{1,a}\right) \ := \ N^\ell_2 \, v^{p,q}_{2,a} \tand
  g\left(N^\ell_1 \, \overline{v^{p,q}_{1,a}}\right) \ := \ 
  N^\ell_2 \, \overline{v^{p,q}_{2,a}} \,.
\]
By construction $g$ is defined over $\bR$ and preserves $Q$ -- that is, $g \in \tAut(V_\bR,Q)$ -- and has the properties $\tAd_g N_1 = N_2$ and $g(I^{p,q}_1) = I^{p,q}_2$.  
\end{proof}

\begin{remark} \label{R:Aut(P,Q)}
Implicit in the proof of Proposition \ref{P:hd} is the following, very useful fact:
\begin{quote}
\emph{
  There is a natural injection 
  $\prod_{k\ge0}\,\tAut( P(N)_k , Q^N_k ) \inj \tAut(V_\bR,Q).$
}
\end{quote}
Given $h \in \prod \tAut(P(N)_k , Q^N_k)$, \eqref{E:Ndecomp} allows us to define $g \in \tAut(V_\bR,Q)$ by $g(N^\ell v) := N^\ell h(v)$, where $v \in P(N)_k$ and $0 \le \ell \le k$.  This is in fact what makes Theorem \ref{T:PDpr} below work for period domains (as opposed to general Mumford-Tate domains).
\end{remark}

\subsection{Primitive subspaces and polarized relations} \label{S:prim+sl2}

Now suppose that $[F_1 , N_1] \preceq [F_2 , N_2] \in \overline{\Psi}_D$, where (in the even weight case) ``$\preceq$'' just means the quotient relation.  Then Theorem \ref{T:4} asserts that there exist commuting $\tSL(2)$'s
\[
  \upsilon : \tSL(2,\bC) \times \tSL(2,\bC) \ \to \ G(\bC)
\]
with the following properties: 
\begin{i_list}
\item There exist commuting DKS--triples
\[
  \{\overline\sE_1 , \sZ_1 , \sE_1\} \tand 
  \{\overline\sE' , \sZ' , \sE' \}
\]
spanning the first and second summands, respectively, of the Lie algebra $\fsl(2,\bC) \op \fsl(2,\bC)$.  Let's distinguish the first and second factors of $\upsilon$ by denoting them $\tSL(2,\bC)_1$ and $\tSL(2,\bC)'$, respectively.
\item 
The sum
\[
  \{\overline\sE_2 = \overline\sE_1 + \overline\sE' \,,\, 
    \sZ_2 = \sZ_1+\sZ' \,,\, \sE_2 = \sE_1+\sE'\} 
\]
is also a DKS--triple.
\item 
For $i=1,2$, the Cayley transform 
\[
  \{ N^+_i , Y_i , N_i \} \ := \
  \tAd_{\varrho_i}^{-1}\{ \overline\sE_i , \sZ , \sE_i\} \ \subset \ \fg_\bR
\]
is a standard triple containing the nilpotent $N_i$ as the nilnegative element, and 
\[
  F_i \ = \ \varrho_i^{-1}\cdot \varphi \,.
\]
Here $\varrho_i = \exp \bi\frac{\pi}{4}(\overline\sE_i+\sE_i)$ is defined as in \eqref{E:rho}.
\end{i_list}
What we want to do in this section is explain how $\Diamond(F_2,N_2)$ is obtained from the action of $\tSL(2)'$ on $P(N_1)$.

Let $V_\bC = \op\,I^{p,q}_1$ denote the Deligne bigrading of the mixed Hodge structure $(F_1,W(N_1))$, and let $j_1^{p,q} = \tdim_\bC\,P(N_1)^{p,q}$ denote the primitive Deligne--Hodge numbers.  Recall that 
\[
  \bigoplus_{p+q = k} I^{p,q}_1 \ = \ \{ v \in V_\bC \ | \ Y_1(v) = k\,v \}
\]
is a $Y_1$--eigenspace.  
\begin{a_list}
\item The fact that the two $\tSL(2)$'s commute implies that $P(N_1)_{n+k}$ is preserved under the action of $\tSL(2)'$.  In fact, setting $N = N_1$ in \eqref{E:Ndecomp} yields
\begin{equation}\label{E:1decomp}
  V_\bR \ = \ \bigoplus_{\mystack{0\le k}{0 \le a \le k}} N_1^a P(N_1)_{n+k}\,,
\end{equation} 
so that the action of $\tSL(2)'$ on $P(N_1)$ determines the action of $\tSL(2)'$ on all of $V$.
\item The restriction $\upsilon$ to the second factor yields a representation 
\[
  \upsilon' : \tSL(2,\bC) \ \to \ 
  \bigoplus_{k=0}^n \tAut(P(N_1)_{n+k} \,,\, Q^{N_1}_k)\,.
\]  
Composing this map with the obvious projection yields a horizontal $\tSL(2)$
\[
  \upsilon_k : \tSL(2,\bC)' \ \to \ \tAut(P(N_1)_{n+k} \,,\, Q^{N_1}_k)
\]
on the period domain $D_k$ for the weight $n+k$, $Q^{N_1}_k$--polarized Hodge structures on $P(N_1)_{n+k}$ with Hodge numbers $\{ j^{p,q}_1 \ | \ p+q = n+k \}$.
\item Set $\varrho' = \exp \bi\frac{\pi}{4}(\overline\sE' + \sE')$ and $(F',N') = (\varrho')^{-1}(\varphi,\sE')$.  The final observations of (b) imply that $(F',N')$ determines a polarized mixed Hodge structure $(F'_k,W(N'_k))$ on $P(N_1)_{n+k}$ with $F'_k \in \check D_k$ and $N'_k \in \tEnd(P(N_1)_{n+k}\,,\,Q^{N_1}_k)$, for each $k=0,\ldots,n$.  Let $\Diamond(F'_k,N'_k)$ denote the Hodge diamond.
\item
The commutativity of the two horizontal $\tSL(2)$'s and \eqref{E:1decomp} imply
\begin{equation}\label{E:HD2}
  \Diamond(F_2,N_2) \ = \ \sum_{\mystack{0\le k}{0 \le a \le k}}  
  \Diamond(F'_k,N'_k)(a) \, ,
\end{equation}
\end{a_list}
where $\Diamond(a)(p,q):=\Diamond(p+a,q+a)$.

\begin{theorem} \label{T:PDpr}
Let $D$ be a period domain parameterizing weight $n$, $Q$--polarized Hodge structures on $V_\bR$.  Let $[F_1,N_1], [F_2,N_2] \in \overline{\Psi}_D$.  Then $[F_1,N_1] \preceq [F_2,N_2]$ if and only if $\Diamond(F_2,N_2)$ can be expressed as a sum \eqref{E:HD2} for Hodge diamonds \emph{(that is, functions subject to \eqref{SE:hd} (applied to $D_k$))} $\Diamond(F'_k,N'_k)$ on the period domains $D_k$ parameterizing weight $n+k$, $Q^{N_1}_k$--polarized Hodge structures on $P(N_1)_{n+k}$ with Hodge numbers $\{j_1^{p,q} \ | \ p+q = n+k \}$.
\end{theorem}

\begin{remark*}
This observation was made independently by Mark Green and Phillip Griffiths \cite{PDpr} in the case that $n=2$.
\end{remark*}

\begin{proof}
$(\implies)$: Necessity was established in the discussion preceding the statement of the proposition, using Theorem \ref{T:4}. Alternatively, this is just admissibility of the degeneration of MHS along a face of the nilpotent cone -- i.e. existence of $M(N',W(N_1))\, (=W(N_2))$.

$(\impliedby)$:  To establish sufficiency note that the converse to items (a) and (b) above holds.  More precisely, if we are given
\begin{blist}
\item a horizontal $\upsilon_1 : \tSL(2,\bC)_1 \to G(\bC)$, and 
\item for each $k=0,\ldots,n$, a horizontal $\upsilon_k  ': \tSL(2,\bC) \to \tAut(P(N_1)_k , Q^{N_1}_k)$ on $D_k$, 
\end{blist}
then we may use the injection of Remark \ref{R:Aut(P,Q)} to assemble these $v_k '$ into a second horizontal $\upsilon' : \tSL(2,\bC)' \to G(\bC)$ commuting with $\upsilon_1$ (Remark \ref{R:Aut(P,Q)}).  This observation, taken with Proposition \ref{P:hd} and \eqref{E:pHD}, yields sufficiency.
\end{proof}

The following four examples illustrate the application of Theorem \ref{T:PDpr}.

\begin{example}[Curves of genus $g$] \label{eg:curve2}
The polarized relations amongst the $[F,N] \in \Psi_D$ of Example \ref{eg:curve1} are 
\[
  \mathrm{I}_a \ \prec \ \mathrm{I_b} \quad\iff\quad a \ < \ b \,.
\]
\end{example}

\begin{example}[K3 surfaces] \label{eg:K32}
The polarized relations amongst the three $[F,N] \in \overline{\Psi}_D$ of Example \ref{eg:K31} are $\mathrm{0} \prec \mathrm{I} \prec \mathrm{II}$.
\end{example}

\begin{example}[Surfaces with contact IPR] \label{eg:H3}
The polarized relations amongst the five nontrivial $[F,N] \in \overline{\Psi}_D$ of Example \ref{eg:H1} are 
\begin{eqnarray*}
  \mathrm{I} & \prec & \mathrm{II}\,,\ \mathrm{III} \,,\ 
  \mathrm{IV}\,,\ \mathrm{V} \,;\\
  \mathrm{II}\,,\ \mathrm{III} & \prec & \mathrm{IV}\,,\ \mathrm{V} \,;\\
  \mathrm{IV} & \prec & \mathrm{V} \,.
\end{eqnarray*}
In this case the polarized relations are transitive, and so define a partial order that we may more compactly express as
\[
  0 \ \prec \ 
  \mathrm{I} \ \prec \ 
  \left\{ \begin{array}{c} 
       \mathrm{II} \\ \mathrm{III}
  \end{array}\right\} \ \prec \ 
  \mathrm{IV} \ \prec \ 
  \mathrm{V}\,.
\]
In particular, all the relations on $\bfN_D$ are polarized.
\end{example}

\begin{example}[Calabi--Yau 3-folds]\label{eg:CY2}
The polarized relations amongst the $4m$ elements of $\Psi_D$ in Example \ref{eg:CY1} are 
\[
 \left.\begin{array}{c} \rmI_a \prec \rmI_b \\
   \rmII_a \prec \rmII_b \\
   \rmIII_a \prec \rmIII_b \\
   \rmIV_a \prec \rmIV_b 
  \end{array} \right\} \ \iff \ a < b \,,
\]
and
\[ \begin{array}{rclcl}
  \rmI_a & \prec & \rmII_b \,,\ \rmIII_b
  & \iff & a \le b \,,\ a < m \,,\\
  \rmI_a & \prec & \rmIV_d 
  & \iff & a < d \,,\ a < m \,,\\
  \rmII_b & \prec & \rmIII_c & \iff & 2 \le b \le c+2 \,,\\
  \rmII_b & \prec & \rmIV_d & \iff & 1 \le b \le d-1 \,,\\
  \rmIII_c & \prec & \rmIV_d & \iff & c+2 \le d \,.
\end{array}\]
Comparing with the result of Example \ref{eg:CY1}, we find that not all the relations on $\bfN_D$ are polarized.  Furthermore, the polarized relation fails to be transitive: $\rmII_0 \prec \rmII_1 \prec \rmIV_2$, but $\rmII_0 \not\prec \rmIV_2$, and so is not a partial order.
\end{example}

\section{The classical case} \label{S:HS}

In this section we study the ``classical case'' that $D$ is Hermitian symmetric, and $T^h D = D$.  This includes the cases that $D$ is the period domain parameterizing polarized Hodge structures with Hodge numbers $(g,g)$ or $(1,m,1)$ (which corresponds to curves and principally polarized abelian varieties, and K3--surfaces, respectively).

It will be helpful to first review commuting horizontal $\tSL(2)$'s associated with strongly orthogonal roots.  (The discussion in \S\ref{S:rtsl2} is terse; see \cite[\S6]{KR1} for more detail.)

\subsection{Roots and horizontal $\tSL(2)$'s}\label{S:rtsl2}
Let $\sR \subset \fh^*$ denote the roots of $\fg_\bC$.  Given a root $\a\in\sR$, let $\fg^\a \subset \fg_\bC$ denote the root space.  Then $\fg^\a$ is 1--dimensional, and $[\fg^\a,\fg^\b]$ is nonzero if and only if $\a+\b$ is a root.

Let 
\[
  \fsl^\a(2,\bC) \ := \ \fg^\a \,\op\,[\fg^\a,\fg^{-\a}] \,\op\, \fg^{-\a}
  \ \simeq \ \fsl(2,\bC)
\]
denote the associated 3--dimensional subalgebra.  If $\a(\ttE_\varphi) = 1$, then we can choose a DKS--triple $\{\overline\sE_\a , \sZ_\a , \sE_\a\}$ spanning the subalgebra $\fsl^\a(2,\bC)$ so that $\sE_\a \in \fg^{-\a} \subset \fg^{-1}_\varphi$ and $\sZ_\a\in\bi\ft \subset \fg^0_\varphi$ (which imply $\overline\sE_\a \in \fg^\a \subset \fg^1_\varphi$).  Note that $\fsl^\a(2,\bC)$ determines a horizontal $\tSL(2)$ 
\[
  \upsilon : \tSL(2,\bC) \ \to \ \tSL^\a(2,\bC) \ \subset \ G(\bC)
\]
at $\varphi$ (\S\ref{S:hSL2}).  The DKS--triple yields a Cayley transform
\begin{equation}\label{E:rb}
  \varrho_\a \ := \ \exp\,\bi\tfrac{\pi}{4} ( \overline\sE_\a + \sE_\a) \,.
\end{equation}

We say that two roots $\a,\b\in\sR$ are \emph{strongly orthogonal} if the two subalgebras $\fsl^\a_2$ and $\fsl^\b_2$ commute.  (Equivalently, neither $\a\pm\b$ are roots.)  In particular, a set of strongly orthogonal roots $\{ \b_1,\ldots,\b_s\} \in \sR$ satisfying $\b_i(\ttE_\varphi) = 1$ yields commuting horizontal $\tSL(2)$'s
\[
  \upsilon : \prod \tSL(2,\bC) \ \to \ \prod_i \tSL^{\b_i}(2,\bC) \ \subset \ G(\bC)
\] 
at $\varphi \in D$.  The corresponding nilpotent cone
\[
  \sigma \ = \ \tspan_{\bR_{>0}}\{N_1 , \ldots , N_s\}
\]
is given by 
\[
  N_i \ := \ \tAd_{\varrho_i}^{-1} \sE_i \ \in \ \fg_\bR \,.
\]
Moreover, given a subset $I \subset \{1,\ldots,s\}$, we see from \eqref{E:rho}, \eqref{E:FN} and \eqref{E:phiinf-ct} that the orbit $\cO_I \in \bfD_D$ polarized by the face
\[
  \s_I \ := \ \tspan_{\bR_{>0}}\{N_i \ | \ i \in I\}
\]
passes through the point
\[
  F_{I,\infty} \ = \ \phi_\infty(F_I,N_I) \,,
\]
where 
\[
  N_I \ := \ \sum_{i \in I} N_i \ \in \ \s_I
\]
and 
\[
  F_I \ := \ \varrho_I^{-1}\cdot\varphi
  \quad\hbox{with}\quad
  \varrho_I \ := \ \prod_{i\in I} \varrho_i \ \in \ G(\bC) \,.
\]

\subsection{The case that $D$ is Hermitian symmetric} \label{S:hs}

In the case that $D$ is Hermitian symmetric the infinitesimal period relation is trivial.  As a result every $G(\bR)^+$--orbit $\cO \subset \overline D$ is polarized; that is, 
\[
  \overline D \ = \ \bigsqcup_{\cO \subset \bfD_D} \cO \,.
\]

\noindent Our proof of Theorem \ref{T:hs} makes use of 

\begin{theorem}[Kor\'anyi--Wolf] \label{T:KW}
Assume that $D$ is Hermitian symmetric.
\begin{a_list}
\item The $G(\bR)^+$--orbits $\cO \subset \overline D$ are linearly ordered.  That is, they may be enumerated so that 
\[
  D \ < \ \cO_1 \ < \cdots < \ \cO_s \,.
\]
\item Moreover, there exists a set $\{\b_1,\ldots,\b_s\} \subset \sR$ of strongly orthogonal roots such that $\b_i(\ttE_\varphi) = 1$ and the orbit $\cO_i$ passes through the point $\varrho_i \circ \cdots \varrho_2 \circ \varrho_1(\varphi) \in \check D$, where $\varrho_i$ is the Cayley transform \eqref{E:rb} associated with $\b_i$.
\end{a_list}
\end{theorem}

\begin{proof}[Proof of Theorem \ref{T:KW}]
See \cite[Theorem 3.2.1]{MR2188135}, and the references therein (especially \cite{MR0174787, MR0192002}).
\end{proof}

\begin{proof}[Proof of Theorem \ref{T:hs}]
Assertion (i) of the theorem follows directly from Theorems \ref{T:1}, \ref{T:0} and \ref{T:KW}(a).  Likewise, part (ii) of the theorem is an immediate consequence of Theorem \ref{T:KW}(b) and the discussion of \S\ref{S:rtsl2}.  (In fact, the complete statement of \cite[Theorem 3.2.1]{MR2188135} implies that $\phi_\infty(F_I,N_I) \in \cO_i$, where $(F_I,N_I)$ is as defined in \S\ref{S:rtsl2} and $|I| = i$.)
\end{proof}

\begin{remark}\label{R:1cone}
It is a consequence of the properties of Cayley transforms that we may choose a (noncompact) real Cartan $\fh_\bR \subset \fg_\bR$ so that the root spaces $\fg^{\b_i}$ are defined over $\bR$.  Moreover, we may choose root vectors $N_i \in \fg^{\b_i}_\bR$ so that every polarized relation in $\Psi_D$ is realized by some face of the cone
\[
  \s \ = \ \tspan_{\bR_{>0}}\{N_1 , \ldots , N_s\} \,.
\]

For general domains, not necessarily Hermitian symmetric, no such single cone will exist.  Special cases in which such a cone does exist include the ``nearly classical'' case that $D$ is a period domain parameterizing polarized Hodge structures with $\bh = (2,m,2)$ (Examples \ref{eg:H1}, \ref{eg:242}, \ref{eg:H2} and \ref{eg:H3}), see \cite[\S5.3]{BPR} and \cite[\S5.7]{KR1}.

See \cite[\S6]{KR1} for a thorough discussion of Cayley transforms, the construction of nilpotent cones (underlying nilpotent orbits) from sets of strongly orthogonal noncompact roots, the corresponding polarized $\cO \in \bfD_D$ and the polarized relationships between these orbits.
\end{remark}

\section{The other extreme: the case that $G^0(\bR)$ is a torus} \label{S:fd}

Here we study the case where the stabilizer $G^0(\bR)$ of the Hodge structure $\varphi \in D$ is a torus, and $T^hD$ is bracket-generating. Under the latter assumption, this is as far from classical as it is possible to be in the following sense: the IPR is trivial when $D$ is Hermitian symmetric; that is, $T^hD = D$ has the maximal possible rank.  In contrast, the length of the filtration
\[
  T^hD \ \subset \ [T^hD , T^hD ] \ \subset \ 
  \left[ T^h D , [T^hD,T^h] \right] \ \subset \cdots \subset TD
\]
is maximized (under the bracket-generating assumption) by the case $G^0(\bR) = T$ as $D$ ranges over all Mumford--Tate domains with automorphism group $G(\bR)^+$.  This length may be thought of as measuring the degree to which the subbundle $T^hD$ fails to be involutive.  In the Hermitian symmetric case $T^hD$ is involutive, while in the case that $G^0(\bR)$ is a torus $T^hD$ is maximally non-involutive in the sense above.

Recall \eqref{E:tg00} that the Lie algebra $\fg^0_{\varphi,\bR}$ of $G^0(\bR)$ contains the compact Cartan subalgebra $\ft \subset \fg_\bR$.  Here we consider the case that the equality 
\begin{equation}\label{E:t=g00}
  \fg^0_{\varphi,\bR} \ = \ \ft
\end{equation}
holds.  Equivalently, the Weyl subgroup $\sW^0$ of \S\ref{S:po} is trivial, so that 
\[
  \cL_{\varphi,\ft} \ = \Lambda_{\varphi,\ft}\,.
\]  

The complexification $\fh = \ft\ot\bC$ is a Cartan subalgebra of $\fg_\bC$.  The equality \eqref{E:t=g00} holds if and only if the stabilizer of $F_{\varphi} \in \check D$ in $G(\bC)$ is a Borel subgroup $B$; that is, $\check D = G(\bC)/B$. Equivalently, there exists a choice of simple roots $\sS \subset \fh^*$ of $\fg_\bC$ such that 
\begin{equation}\label{E:a=1}
  \a\in\sR \quad\hbox{\emph{satisfies}}\quad 
  \a(\ttE_\varphi) \ = \ 1 \quad \hbox{\emph{if and only if}} \quad 
  \a\in \sS \,.
\end{equation}
(This equivalence requires the hypothesis that the IPR is bracket--generating.  See \cite[\S3.1.2]{SL2} and the references therein for further discussion.)

Let $\cP(\sS)$ denote the power set of the simple roots $\sS = \{\a_1,\ldots,\a_r\} \subset \fh^*$ of $\fg_\bC$.  Define a partial order $\le$ on $\cP(\sS)$ by declaring $\sS_1 \le \sS_2$ if and only if $\sS_1 \subseteq \sS_2$.  We say the relation $\sS_1 \le \sS_2$ is \emph{polarized} if the elements of $\sS_1$ are strongly orthogonal (\S\ref{S:rtsl2}) to the elements of $\sS_2 \backslash \sS_1$; in this case we write $\sS_1 \preceq \sS_2$.  Note that $\preceq$ is also a partial order, since given $\sS_1\subseteq \sS_2 \subseteq \sS_3$ with $\sS_1$ [resp. $\sS_2$] strongly orthogonal to $\sS_2\setminus \sS_1$ [resp. $\sS_3\setminus \sS_2$], we have $\sS_1$ strongly orthogonal to $\sS_2 \setminus \sS_1 \cup \sS_3 \setminus \sS_2 = \sS_3 \setminus \sS_1$.

\begin{remark}\label{R:sosr}
If $\a_i,\a_j \in \sS$ are two simple roots of $\fg_\bC$, then $\a_i - \a_j$ is never a root.  So $\a_i$ and $\a_j$ are strongly orthogonal if and only if $\a_i+\a_j$ is not a root.
\end{remark}

\begin{proposition}\label{P:fd}
There is a natural bijection between $\cL_{\varphi,\ft} = \Lambda_{\varphi,\ft} \simeq \Psi_D$ and the power set $\cP(\sS)$ that preserves both the relations $\le$ and the polarized relations $\preceq$.
\end{proposition}

\begin{corollary}\label{C:fd}
The relations $\leq$ and $\preceq$ on $\Psi_D$ are partial orders.
\end{corollary}

\begin{proof}
\emph{Step 1.}  First we will establish the bijection 
\begin{equation}\label{E:bij}
  \cL_{\varphi,\ft} \ \leftrightarrow \ \cP(\sS)\,.
\end{equation}  
Fix a Levi subalgebra $\ft \subset \fl_\bR \subset \fg_\bR$.  Let $\sR' \subset\sR$ denote the roots of $\fl_\bC$.  Let $\fb \subset \fg_\bC$ be the Lie algebra of $B$.  Note that $\fb' = \fb\cap\fl_\bC^\tss$ is a Borel subalgebra of $\fl^{\tss}_\bC$.  Equivalently, $\fl^{\tss}_\bC$ admits a choice of simple roots $\sS'\subset\sR'$ with the properties
\begin{subequations} \label{SE:b}
\begin{eqnarray}
  \label{E:b>0}
  \b\in\sS' & \hbox{implies} & \b(\ttE_\varphi) > 0 \,,\\
  \label{E:b=1}
  \b\in\sR' \tand
  \b(\ttE_\varphi) \ = \ 1 & \hbox{implies} &
  \b\in \sS' \,.
\end{eqnarray}
\end{subequations}  
Let $\fl_\bC^\tss = \op\,\fl^p$ be the $\ttE_\varphi$--eigenspace decomposition.  Then \eqref{E:b>0} holds if and only if $\fl^0$ is the Cartan subalgebra $\fh' = \fh \cap \fl^\tss_\bC$; in particular, 
\begin{equation}\label{E:bhl}
  \tdim\,\fl^0 \ = \ \tdim\,\fh' \ = \ |\sS'| \,.
\end{equation}
Note that   
\begin{equation} \nonumber
  \fl^1 \ = \ \bigoplus_{\mystack{\b\in\sR'}{\b(\ttE_\varphi)=1}} \fg^\b 
  \ \subset \ \fl^\tss_\bC \,.
\end{equation}
Recollect that $\sZ$ is a distinguished grading element of $\fl^\tss_\bC$ if and only if $\tdim\,\fl^0 = \tdim\,\fl^1$ (\S\ref{S:dge}).  Equations \eqref{E:a=1}, \eqref{E:b=1} and \eqref{E:bhl} imply that this is equivalent to $\sS' \subset \sS$.  This establishes the bijection \eqref{E:bij}.  

For later use we note that the argument above establishes
\begin{equation}\label{E:bor2}
  \fl^1 \ = \ \bigoplus_{\b\in\sS'} \fg^\b \,.
\end{equation}

\smallskip

\emph{Step 2.}  Given two Levis $\fl_1 , \fl_2 \in \cL_{\varphi,\ft}$, let $\sS_1 , \sS_2 \subset \sS$ denote the corresponding simple roots.  Note that $\fl_1 \subset \fl_2$ if and only if $\sS_1 \subset \sS_2$.  So to prove that the bijection \eqref{E:bij} preserves the relations it suffices to show that 
\begin{equation}\label{E:=t}
  \tilde \fl \ = \ \fl \quad\hbox{for all} \quad 
  \fl \ \in \ \cL_{\varphi,\ft}\,.
\end{equation}

Both $\fb \cap \fl$ and $\fb \cap \tilde\fl$ are Borel subalgebras of $\fl$ and $\tilde\fl$, respectively.  These Borel subalgebras uniquely determine simple roots $\sS' \subset \tilde\sS \subset \sR$ of $\fl$ and $\tilde\fl$, respectively, with the property that 
\begin{equation}\label{E:bor1}
  2\a(\ttE_\varphi) \ = \ \a(\sZ) > 0 \,,
\end{equation}
for all $\a\in\tilde\sS$ (cf. \S\ref{S:rt} and \eqref{E:tL}).  By definition $\sZ \in \fl^\tss_\bC$.  If $\tilde \fl \not= \fl$, then there exists a fundamental weight $\w_i$ of $\tilde\fl^\tss$ so that $\w_i(\sZ) = 0$ (\S\ref{S:wt}(a)).  However, \eqref{E:bor1} and \S\ref{S:wt}(b) imply no such $\w_i$ exists.  This establishes \eqref{E:=t}.  

\smallskip

\emph{Step 3.}  It remains to show that \eqref{E:bij} preserves the polarized relations.  First suppose $[F_1,N_1] \prec [F_2,N_2] \in \Psi_D$.  Let $\fl_1 \subset \fl_2 \in \cL_{\varphi,\ft}$ denote the corresponding Levis and $\sS_1 \subset \sS_2 \subset \sS$ their simple roots.  We need to show the roots of $\sS_1$ are strongly orthogonal to the roots of $\sS' = \sS_2\backslash\sS_1$. 

By Theorem \ref{T:4} the polarized relation $[F_1,N_1] \prec [F_2,N_2]$ may be realized by commuting horizontal $\tSL(2)$'s.  Equivalently, there exist commuting DKS--triples
\begin{equation}\label{E:2dks}
  \{ \overline\sE_1 , \sZ_1 , \sE_1 \} \tand 
  \{ \overline\sE' , \sZ' , \sE' \} 
\end{equation}
(these span the two commuting horizontal $\fsl(2,\bC)$'s) such that
\[
  \overline\sE_2 \ = \ \overline\sE_1 + \overline\sE' \,,\quad
  \sZ_2 \ = \ \sZ_1 + \sZ' \,,\quad
  \sE_2 \ = \ \sE_1 + \sE'
\]
is a third DKS--triple and 
\[
  F_i \ = \ \varrho_i^{-1} \varphi \tand 
  N_i \ = \ \tAd_{\varrho_i}^{-1} \sE_i
\]
with $\varrho_i$ given by \eqref{E:rho}, for $i=1,2$, \cf\eqref{E:FN}

Let $\{ \overline\sE , \sZ , \sE\}$ be any one of the three DKS--triples above.  By virtue of the isomorphism $\Upsilon_D \simeq \Lambda_{\varphi,\ft}$ of \S\ref{S:1varSL2}, each triple determines a Levi $\fl \in \cL_{\varphi,\ft}$.  Let $\sS_\fl \subset \sS$ denote the simple roots.   Then $\sE \in \fl^{-1}$ and \eqref{E:bor2} imply that 
\[
  \sE \ = \ \sum_{\a\in\sS_\fl}\sE^\a \,,\quad \sE^\a \in \fg^{-\a}
\]
is a sum of \emph{simple} root vectors.  The Bala--Carter Theorem \cite{MR0417306, MR0417307} asserts that $\fl_\bC$ is the minimal Levi subalgebra containing $\sE$.  This forces $\sE^\a \not=0$ for all $\a \in \sS_\fl$.  

Therefore,
\begin{equation}\label{E:2E}
  \sE_1 \ = \ \sum_{\a\in\sS_1} \sE^\a \tand 
  \sE' = \sum_{\b\in\sS'} \sE^\b \,,
\end{equation}
and the $\sE^\a$ and $\sE^\b$ are nonzero.  A necessary condition for the two DKS--triples to commute is $[\sE^\a,\sE^\b]=0$ for all $\a\in\sS_1$ and $\b\in\sS_2\backslash\sS_1$.  Equivalently (\S\ref{S:rt}), $\a+\b$ is not a root.  By Remark \ref{R:sosr}, this implies the roots of $\sS_1$ and $\sS' = \sS_2\backslash \sS_1$ are strongly orthogonal.

Conversely, if $\sS_1$ and $\sS' = \sS_2\backslash\sS_1$ are strongly orthogonal, the Levis $\fl_1 , \fl' \in \cL_{\varphi,\ft}$ with simple roots $\sS_1$ and $\sS'$ we have two horizontal DKS--triples \eqref{E:2dks} at $\varphi$ satisfying \eqref{E:2E}.  We claim that the two DKS--triples commute.  It is immediate from the definition of strong orthogonality that $\{ \overline \sE_1 , \sE_1\}$ and $\{\overline\sE' , \sE'\}$ commute.  The Jacobi identity then implies that $\sZ_1 = [\overline\sE_1 , \sE_1]$ commutes with $\{\overline\sE' , \sE'\}$ and that $\sZ' = [\overline\sE',\sE']$ commutes with $\{ \overline \sE_1 , \sE_1\}$.
\end{proof}

\begin{example}[Weight 7 and $\bh = (1,\ldots,1)$] \label{eg:borel7}

Let $D$ be the period domain parameterizing Hodge structures of weight $n=7$ with Hodge numbers $\bh = (1,\ldots,1)$.  In this case the full automorphism group $G(\bR) = \tSp_8\bR$ is connected, and $G^0(\bR)$ is a torus.  By Proposition \ref{P:fd} the elements of $[F,N] \in \Psi_D$ are indexed by subsets $\sS' \subset \sS$.  For each subset we give the corresponding Hodge diamond, the distinguished semisimple element $\sZ$, the codimension of the $G(\bR)$--orbit $\cO = \phi_\infty([F,N])$ in $\check D$ (see \cite[\S4]{MR3331177}), and the partially signed Young diagram classifying the nilpotent conjugacy class $\cN = \pi([F,N]) \in \bfN$ (see \cite[\S2.3]{BPR}).  Making use of Proposition \ref{P:fd}, we can read the relations $\leq$ and $\preceq$ on $\Psi_D$ off from the simple roots $\sS'$.  Finally, in order order to specify $\sZ$, we let $\{\ttS^1,\ldots,\ttS^4\}$ denote the basis of $\fh$ dual to the simple roots $\sS = \{\a_1,\ldots,\a_4\}\subset \fh^*$.  Then $\ttE_\varphi = \ttS^1 + \cdots+ \ttS^4$.  We will employ the shorthand $\sum_i a_i \ttS^i = (a_1,\ldots,a_4)$ to denote the grading element $\sZ = 2\,\pi^\tss_\fl(\ttE_\varphi)$.  
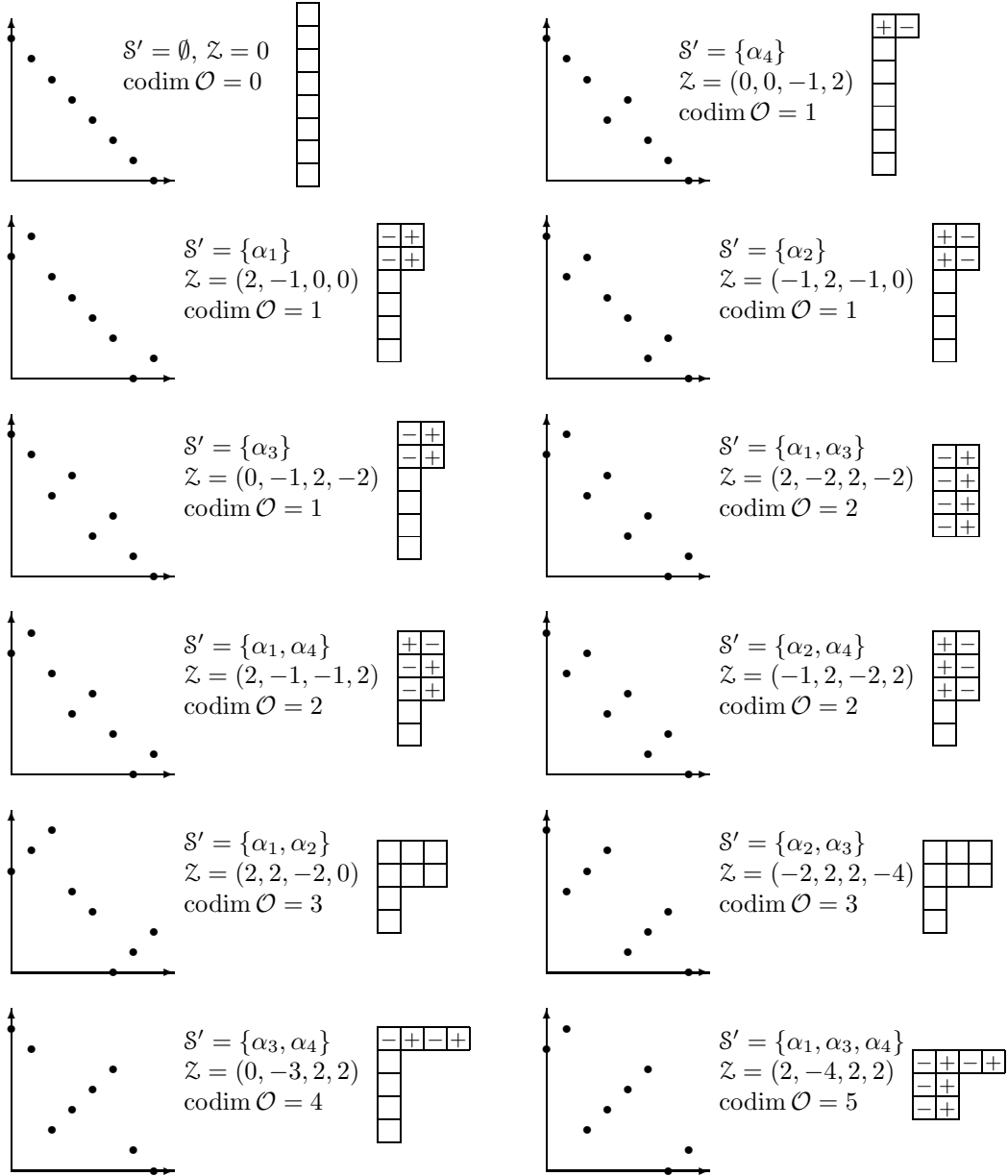
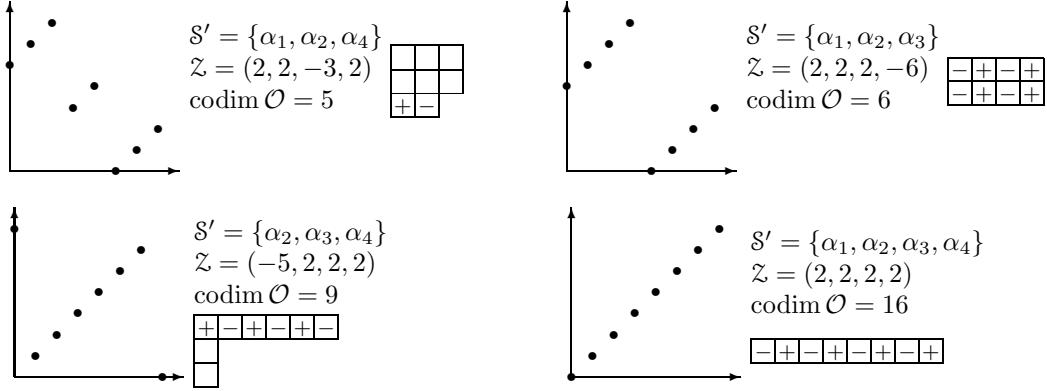
\begin{center}
\begin{footnotesize}
\setlength{\unitlength}{8pt}
\begin{picture}(23,9)(0,0)
\put(0,0){\vector(1,0){8}} \put(0,0){\vector(0,1){8}}  
\put(0,7){\circle*{0.4}} \put(1,6){\circle*{0.4}} \put(2,5){\circle*{0.4}} \put(3,4){\circle*{0.4}} \put(4,3){\circle*{0.4}} \put(5,2){\circle*{0.4}} \put(6,1){\circle*{0.4}} \put(7,0){\circle*{0.4}}
\put(5.5,6.0){$\sS' = \emptyset$, $\sZ = 0$}
\put(5.5,4.5){$\tcodim\,\cO = 0$}
\put(14,4.0){\tiny{$\yng(1,1,1,1,1,1,1,1)$}}
\end{picture}
\hsp{20pt}
\begin{picture}(23,9)(0,0)
\put(0,0){\vector(1,0){8}} \put(0,0){\vector(0,1){8}}
\put(0,7){\circle*{0.4}} \put(1,6){\circle*{0.4}} \put(2,5){\circle*{0.4}} \put(3,3){\circle*{0.4}} \put(4,4){\circle*{0.4}} \put(5,2){\circle*{0.4}} \put(6,1){\circle*{0.4}} \put(7,0){\circle*{0.4}}
\put(6.5,6.0){$\sS' = \{\a_4\}$}
\put(6.5,4.5){$\sZ = (0,0,-1,2)$}
\put(6.5,3.0){$\tcodim\,\cO = 1$}
\put(16,4.0){\tiny{$\young(+-,\blank,\blank,\blank,\blank,\blank,\blank)$}}
\end{picture}
\vspace{5pt}\\
\begin{picture}(23,9)(0,0)
\put(0,0){\vector(1,0){8}} \put(0,0){\vector(0,1){8}}  \put(0,6){\circle*{0.4}} \put(1,7){\circle*{0.4}} \put(2,5){\circle*{0.4}} \put(3,4){\circle*{0.4}} \put(4,3){\circle*{0.4}} \put(5,2){\circle*{0.4}} \put(6,0){\circle*{0.4}} \put(7,1){\circle*{0.4}}
\put(8.5,6.0){$\sS' = \{\a_1\}$}
\put(8.5,4.5){$\sZ = (2,-1,0,0)$}
\put(8.5,3.0){$\tcodim\,\cO = 1$}
\put(18,4.0){\tiny{$\young(-+,-+,\blank,\blank,\blank,\blank)$}}
\end{picture}
\hsp{20pt}
\begin{picture}(23,9)(0,0)
\put(0,0){\vector(1,0){8}} \put(0,0){\vector(0,1){8}}  \put(0,7){\circle*{0.4}} \put(1,5){\circle*{0.4}} \put(2,6){\circle*{0.4}} \put(3,4){\circle*{0.4}} \put(4,3){\circle*{0.4}} \put(5,1){\circle*{0.4}} \put(6,2){\circle*{0.4}} \put(7,0){\circle*{0.4}}
\put(8.5,6.0){$\sS' = \{\a_2\}$}
\put(8.5,4.5){$\sZ = (-1,2,-1,0)$}
\put(8.5,3.0){$\tcodim\,\cO = 1$}
\put(19,4.0){\tiny{$\young(+-,+-,\blank,\blank,\blank,\blank)$}}
\end{picture}
\vspace{5pt}\\
\begin{picture}(23,9)(0,0)
\put(0,0){\vector(1,0){8}} \put(0,0){\vector(0,1){8}}
\put(0,7){\circle*{0.4}} \put(1,6){\circle*{0.4}} \put(2,4){\circle*{0.4}} \put(3,5){\circle*{0.4}} \put(4,2){\circle*{0.4}} \put(5,3){\circle*{0.4}} \put(6,1){\circle*{0.4}} \put(7,0){\circle*{0.4}}
\put(8.5,6.0){$\sS' = \{\a_3\}$}
\put(8.5,4.5){$\sZ = (0,-1,2,-2)$}
\put(8.5,3.0){$\tcodim\,\cO = 1$}
\put(19,4.0){\tiny{$\young(-+,-+,\blank,\blank,\blank,\blank)$}}
\end{picture}
\hsp{20pt}
\begin{picture}(23,9)(0,0)
\put(0,0){\vector(1,0){8}} \put(0,0){\vector(0,1){8}}
\put(0,6){\circle*{0.4}} \put(1,7){\circle*{0.4}} \put(2,4){\circle*{0.4}} \put(3,5){\circle*{0.4}} \put(4,2){\circle*{0.4}} \put(5,3){\circle*{0.4}} \put(6,0){\circle*{0.4}} \put(7,1){\circle*{0.4}}
\put(8.5,6.0){$\sS' = \{\a_1,\a_3\}$}
\put(8.5,4.5){$\sZ = (2,-2,2,-2)$}
\put(8.5,3.0){$\tcodim\,\cO = 2$}
\put(19,4.0){\tiny{$\young(-+,-+,-+,-+)$}}
\end{picture}
\vspace{5pt}\\
\begin{picture}(23,9)(0,0)
\put(0,0){\vector(1,0){8}} \put(0,0){\vector(0,1){8}}
\put(0,6){\circle*{0.4}} \put(1,7){\circle*{0.4}} \put(2,5){\circle*{0.4}} \put(3,3){\circle*{0.4}} \put(4,4){\circle*{0.4}} \put(5,2){\circle*{0.4}} \put(6,0){\circle*{0.4}} \put(7,1){\circle*{0.4}}
\put(8.5,6.0){$\sS' = \{\a_1,\a_4\}$}
\put(8.5,4.5){$\sZ = (2,-1,-1,2)$}
\put(8.5,3.0){$\tcodim\,\cO = 2$}
\put(19,4.0){\tiny{$\young(+-,-+,-+,\blank,\blank)$}}
\end{picture}
\hsp{20pt}
\begin{picture}(23,9)(0,0)
\put(0,0){\vector(1,0){8}} \put(0,0){\vector(0,1){8}}
\put(0,7){\circle*{0.4}} \put(1,5){\circle*{0.4}} \put(2,6){\circle*{0.4}} \put(3,3){\circle*{0.4}} \put(4,4){\circle*{0.4}} \put(5,1){\circle*{0.4}} \put(6,2){\circle*{0.4}} \put(7,0){\circle*{0.4}}
\put(8.5,6.0){$\sS' = \{\a_2,\a_4\}$}
\put(8.5,4.5){$\sZ = (-1,2,-2,2)$}
\put(8.5,3.0){$\tcodim\,\cO = 2$}
\put(19,4.0){\tiny{$\young(+-,+-,+-,\blank,\blank)$}}
\end{picture}
\vspace{5pt}\\
\begin{picture}(23,9)(0,0)
\put(0,0){\vector(1,0){8}} \put(0,0){\vector(0,1){8}}
\put(0,5){\circle*{0.4}} \put(1,6){\circle*{0.4}} \put(2,7){\circle*{0.4}} \put(3,4){\circle*{0.4}} \put(4,3){\circle*{0.4}} \put(5,0){\circle*{0.4}} \put(6,1){\circle*{0.4}} \put(7,2){\circle*{0.4}}
\put(8.5,6.0){$\sS' = \{\a_1,\a_2\}$}
\put(8.5,4.5){$\sZ = (2,2,-2,0)$}
\put(8.5,3.0){$\tcodim\,\cO = 3$}
\put(18,4.0){\tiny{$\yng(3,3,1,1)$}}
\end{picture}
\hsp{20pt}
\begin{picture}(23,9)(0,0)
\put(0,0){\vector(1,0){8}} \put(0,0){\vector(0,1){8}}
\put(0,7){\circle*{0.4}} \put(1,4){\circle*{0.4}} \put(2,5){\circle*{0.4}} \put(3,6){\circle*{0.4}} \put(4,1){\circle*{0.4}} \put(5,2){\circle*{0.4}} \put(6,3){\circle*{0.4}} \put(7,0){\circle*{0.4}}
\put(8.5,6.0){$\sS' = \{\a_2 , \a_3\}$}
\put(8.5,4.5){$\sZ = (-2,2,2,-4)$}
\put(8.5,3.0){$\tcodim\,\cO = 3$}
\put(18.5,4.0){\tiny{$\yng(3,3,1,1)$}}
\end{picture}
\vspace{5pt}\\
\begin{picture}(23,9)(0,0)
\put(0,0){\vector(1,0){8}} \put(0,0){\vector(0,1){8}}
\put(0,7){\circle*{0.4}} \put(1,6){\circle*{0.4}} \put(2,2){\circle*{0.4}} \put(3,3){\circle*{0.4}} \put(4,4){\circle*{0.4}} \put(5,5){\circle*{0.4}} \put(6,1){\circle*{0.4}} \put(7,0){\circle*{0.4}}
\put(8.5,6.0){$\sS' = \{\a_3,\a_4\}$}
\put(8.5,4.5){$\sZ = (0,-3,2,2)$}
\put(8.5,3.0){$\tcodim\,\cO = 4$}
\put(18.0,4.0){\tiny{$\young(-+-+,\blank,\blank,\blank,\blank)$}}
\end{picture}
\hsp{20pt}
\begin{picture}(23,9)(0,0)
\put(0,0){\vector(1,0){8}} \put(0,0){\vector(0,1){8}}
\put(0,6){\circle*{0.4}} \put(1,7){\circle*{0.4}} \put(2,2){\circle*{0.4}} \put(3,3){\circle*{0.4}} \put(4,4){\circle*{0.4}} \put(5,5){\circle*{0.4}} \put(6,0){\circle*{0.4}} \put(7,1){\circle*{0.4}}
\put(8.5,6.0){$\sS' = \{\a_1,\a_3,\a_4\}$}
\put(8.5,4.5){$\sZ = (2,-4,2,2)$}
\put(8.5,3.0){$\tcodim\,\cO = 5$}
\put(18.0,4.0){\tiny{$\young(-+-+,-+,-+)$}}
\end{picture}
\vspace{5pt}\\
\begin{picture}(23,9)(0,0)
\put(0,0){\vector(1,0){8}} \put(0,0){\vector(0,1){8}}
\put(0,5){\circle*{0.4}} \put(1,6){\circle*{0.4}} \put(2,7){\circle*{0.4}} \put(3,3){\circle*{0.4}} \put(4,4){\circle*{0.4}} \put(5,0){\circle*{0.4}} \put(6,1){\circle*{0.4}} \put(7,2){\circle*{0.4}}
\put(8.5,6.0){$\sS' = \{\a_1,\a_2,\a_4\}$}
\put(8.5,4.5){$\sZ = (2,2,-3,2)$}
\put(8.5,3.0){$\tcodim\,\cO = 5$}
\put(18.0,4.0){\tiny{$\young(\blank\blank\blank,\blank\blank\blank,+-)$}}
\end{picture}
\hsp{20pt}
\begin{picture}(23,9)(0,0)
\put(0,0){\vector(1,0){8}} \put(0,0){\vector(0,1){8}}
\put(0,4){\circle*{0.4}} \put(1,5){\circle*{0.4}} \put(2,6){\circle*{0.4}} \put(3,7){\circle*{0.4}} \put(4,0){\circle*{0.4}} \put(5,1){\circle*{0.4}} \put(6,2){\circle*{0.4}} \put(7,3){\circle*{0.4}}
\put(8.5,6.0){$\sS' = \{\a_1,\a_2,\a_3\}$}
\put(8.5,4.5){$\sZ = (2,2,2,-6)$}
\put(8.5,3.0){$\tcodim\,\cO = 6$}
\put(18.0,4.0){\tiny{$\young(-+-+,-+-+)$}}
\end{picture}
\vspace{5pt}\\
\begin{picture}(23,9)(0,0)
\put(0,0){\vector(1,0){8}} \put(0,0){\vector(0,1){8}}
\put(0,7){\circle*{0.4}} \put(1,1){\circle*{0.4}} \put(2,2){\circle*{0.4}} \put(3,3){\circle*{0.4}} \put(4,4){\circle*{0.4}} \put(5,5){\circle*{0.4}} \put(6,6){\circle*{0.4}} \put(7,0){\circle*{0.4}}
\put(8.5,6.5){$\sS' = \{\a_2,\a_3,\a_4\}$}
\put(8.5,5.0){$\sZ = (-5,2,2,2)$}
\put(8.5,3.5){$\tcodim\,\cO = 9$}
\put(8.5,1.0){\tiny{$\young(+-+-+-,\blank,\blank)$}}
\end{picture}
\hsp{20pt}
\begin{picture}(23,9)(0,0)
\put(0,0){\vector(1,0){8}} \put(0,0){\vector(0,1){8}}
\put(0,0){\circle*{0.4}} \put(1,1){\circle*{0.4}} \put(2,2){\circle*{0.4}} \put(3,3){\circle*{0.4}} \put(4,4){\circle*{0.4}} \put(5,5){\circle*{0.4}} \put(6,6){\circle*{0.4}} \put(7,7){\circle*{0.4}}
\put(8.5,6.0){$\sS' = \{\a_1,\a_2,\a_3,\a_4\}$}
\put(8.5,4.5){$\sZ = (2,2,2,2)$}
\put(8.5,3.0){$\tcodim\,\cO = 16$}
\put(8.5,1.0){\tiny{$\young(-+-+-+-+)$}}
\end{picture}
\end{footnotesize}
\end{center}
Note that the two $[F,N]$ indexed by $\sS' = \{\a_1\}$ and $\sS' = \{\a_3\}$, respectively, have the same Young diagram.  That is, $\pi(\{\a_1\}) = \pi(\{\a_3\})$.  Likewise $\pi(\{\a_1,\a_2\}) = \pi(\{\a_2,\a_3\})$.  Thus, $\pi$ fails to be injective.
\end{example}

\begin{example}[Weight $5$ and $\bh=(1,\ldots,1)$]
In this example we answer a question asked of the authors by E. Cattani: given an inclusion $D_M = M(\bR)^+ .\varphi \subset G(\bR)^+ .\varphi =D$ of M-T domains, and a nilpotent $N\in \mathfrak{m}_{\mathbb{Q}}$ such that the (pre-)boundary component $\tilde{B}(N)$ is nonempty, must $\tilde{B}_M(N)$ be nonempty? Here $\tilde{B}(N)\subset \tilde{D}$ is the set of flags for which $e^{zN}F$ is a nilpotent orbit in $D$, and $\tilde{B}_M(N)\subset \check{D}_M$ is the subset yielding nilpotent orbits on $D_M$.

Assume $G(\bR)$ is connected; then according to \cite[(VI.B.11)]{\GGK}, we have $\check{D}_M \cap D = M(\bR).\varphi$. Clearly if this is connected, then it is $D_M$, and $\tilde{B}_M(N) = \tilde{B}(N) \cap \check{D}_M$. We show that even in this case the answer can be negative.

Let $D$ be the period domain for HS of weight $5$ with all Hodge numbers $1$, on $(V,Q)=\oplus_{i=0}^2 (V_i,Q_i),$ where the $(V_i,Q_i)$ are isomorphic symplectic planes. Let $\varphi = \oplus_{i=0}^2 \varphi_i$ be the sum of these CM HS of respective types $(i,5-i)+(5-i,i)$, with M-T groups $T_i$ (compact $1$-tori non-isomorphic over $\mathbb{Q}$). Set $M=\mathrm{SL}(V_0)\times T_1 \times T_2$, and put $N:=\left( \tiny \begin{matrix} 0 & 0 \\ 1 & 0 \end{matrix} \right)$ in the $\mathfrak{sl}(V_0)$-factor of $\mathfrak{m}$; then $D_M=M(\bR).\varphi$ is a $1$-dimensional subdomain with trivial horizontal distribution hence no boundary components. That is, $\tilde{B}_M(N)=\emptyset$.

Next let $\varphi ' = \oplus_{i=0}^2 \varphi_i  ' \in D$ be the CM HS given by $\varphi_0 ' = (\varphi_2)^{\frac{1}{3}}$, $\varphi_1 ' = \varphi_1$, $\varphi_2 ' = (\varphi_0)^3$; i.e. the Hodge numbers on $V_0$ and $V_2$ have been swapped.  Then $D_M ' := M(\bR).\varphi '$ is a \emph{horizontal} $1$-dimensional subdomain (with different compact dual!), which \emph{is} a nilpotent orbit in $D$ under $N$. So $\tilde{B}(N) \neq \emptyset$.

As in the last example, $\Psi_D \twoheadrightarrow \bfN_D$ is not bijective.  One wonders the extent to which this is responsible for the negative answer.
\end{example}

See Example \ref{eg:G2_12} for another domain with $G^0(\bR)$ a torus.

\section{The secondary poset and further examples} \label{S:eg}

In this final section we turn at last to the classification of nilpotent orbits and cones as promised in the Introduction.  The secondary poset and its refinements are defined in \S\ref{S:cubes}, with the remaining subsections devoted to examples.

\subsection{Admissible $n$-cubes and polarizable cones} \label{S:cubes}

Returning once more to the subject of $\S$\ref{S:nc+pr}, let $\sigma=\bR_{>0}\langle N_1,\ldots,N_n\rangle \subset \fg_{\bR}$ be a nilpotent cone\footnote{More precisely, we mean a nilpotent cone together with an ordering of its generators, but will quotient this ordering out below.} underlying a nilpotent orbit; assume in particular that $\mathrm{rk}(\sigma)=n$ so that $\sigma$ is simplicial. Let $\cC_k$ denote the poset consisting of functions $\be: \{1,\ldots,n\}\to\{0,1\}$, with the natural partial order: $\be\leq \be' \iff \be(j)\leq \be'(j) \;(\forall j)$. Write $\underline{0}$ and $\underline{1}$ for the constant functions, and $\be^i$ for $\delta_j^i$; and set $|\be|:=\sum_{i=1}^n \be(i)$. Recalling the map $\psi^{\circ}:\Gamma_{\sigma}\to \Psi_D$ defined in \eqref{E:glob}, we may consider the composite
\begin{equation*} \mu_{\sigma}: \cC_n \overset{\cong}{\to}\Gamma_{\sigma} \overset{\psi^{\circ}}{\to}\Psi_D , \end{equation*}
which maps $\underline{0}\mapsto  [\{0\}]$ (the trivial Levi), and relations $\leq$ to polarized relations $\preceq$.\begin{definition} \label{D:onc}
An {\it ordered $n$-cube} is a map $\mu:\cC_n\to\Psi_D$ with:
\begin{a_list}
\item $\mu^{-1} [\{0\}] = \{\underline{0}\}$; and
\item $\be \leq \be' \implies \mu(\be)\preceq \mu(\be')$.
\end{a_list}
We shall term $\mu$ {\it polarizable} if $\mu=\mu_{\sigma}$ for some ($n$-dimensional, simplicial) polarizable nilpotent cone $\sigma$.
\end{definition}
The obvious question here is how to find the polarizable ordered $n$-cubes (hard) among the ordered $n$-cubes (easy, once $(\Psi_D,\preceq)$ is known). In particular, how much can be deduced from ``combinatorial'' (finite) methods alone? To this end, we introduce an integer invariant on $\Psi_D$, defined as follows.
\par Given a Levi $\ell\in\cL_{\varphi,\ft}$, with $F,N,\tilde{\fl}=\oplus \tilde{\fl}^{p,p} = \oplus \fg_{(F,N)}^{p,p}$ as in \eqref{E:diag}, choose a decomposition $\cR_{\tilde{\fl}}=\cR^+_{\tilde{\fl}}\cup \cR^-_{\tilde{\fl}}$ of the roots so that $\cR^+_{\tilde{\fl}}$ contains the weights of the $\{\tilde{\fl}^{p,p}\}_{p>0}$; and let
\begin{itemize}
\item $\cW^+_{\tilde{\fl}} := \left\{ w\in \cW_{\tilde{\fl}} \left| w(\cR^+_{\tilde{\fl}})\supset \cR^+_{\tilde{\fl}^{0,0}}\right.\right\} $,
\item $\Delta(w):= w(\cR^+_{\tilde{\fl}})\cap \cR^-_{\tilde{\fl}}$,
\item $\cW^{\#}_{\tilde{\fl}} := \left\{\left. w\in \cW^+_{\tilde{\fl}} \right| \Delta(w)\subset \cR_{\tilde{\fl}^{-1,-1}} \right\}$, and
\item $\cW^{\#}_{\tilde{\fl}}(d):= \left\{\left. w\in \cW^{\#}_{\tilde{\fl}} \right| |\Delta(w)|=d \right\}$.
\end{itemize}
\begin{definition} \label{D:cap}
The {\it capacity} of $\fl$ is the nonnegative integer \[ \mathrm{cap}(\fl):=\max\left\{ d \left| \cW^{\#}_{\tilde{\fl}}(d)\neq \emptyset\right. \right\} .\]
\end{definition}
Since this is invariant under the Weyl group $\cW^0$, it yields a function \[ \mathrm{cap}:\Psi_D \to \bZ_{\geq 0} .\]
\begin{definition} \label{D:adm}
An ordered $n$-cube is {\it admissible} if $|\be|\leq \mathrm{cap}(\mu(\be))$ for each $\be\in \cC_n$.
\end{definition}
It is convenient to choose a system of representatives $\fl_{\be}\in \cL_{\varphi,\ft}$ with $[\fl_{\be} ]=\mu(\be)$ and $\be'\leq \be \implies \tilde{\fl}_{\be'}\subseteq \tilde{\fl}_{\be}$. (Note that this is possible by Theorem \ref{T:2}.) For each $\fl_{\be}$, write $N_{\be}$ for a corresponding nilpotent (and $N_i :=N_{\be^i}$); in particular, recall that this must belong to an {\it open} orbit of $L^{0,0}_{\be} (\bR)^+$ on $\fl_{\be}^{-1,-1}$. (For simplicity, here we can simply restrict the Hodge-Tate grading on $\tilde{\fl}_{\underline{1}}$ associated to $(F,W(N_{\underline{1}}))$ to all $\fl_{\be},\tilde{\fl}_{\be}$.) Set $I_{\be}:=\left\{ i\in \{1,\ldots,n\}\left| \be(i)=1\right. \right\}$.
\begin{definition} \label{D:str}
An ordered $n$-cube is {\it strongly admissible} if, for each $\be$, there exist $\{ g_i^{\be}\in \tilde{L}^{0,0}_{\be}(\bR)^+\}_{i\in I_{\be}}$ such that:
\begin{a_list}
\item the $\{ N_i^{\be}:=\mathrm{Ad}(g^{\be}_i )N_i\}_{i\in I_{\be}}$ commute and are linearly independent; and
\item $\sigma_{\be'}^{\be} =\sum \bR_{>0} \be'(i) N_i^{\be} \subset \tilde{L}^{0,0}_{\be}(\bR)^+.N_{\be'} \; (\forall \be'\leq \be )$.
\end{a_list}
\end{definition}
This condition is more subtle and difficult to check than admissibility, as we shall see in $\S$\ref{S:G2}. In a sense it is the full ``first Hodge-Riemann bilinear relation'' for $n$-cubes; that is, the only remaining obstacle to polarizability is a ``positivity condition''.
\begin{theorem} \label{T:cubes}
For ordered $n$-cubes, \[ \text{polarizability}\implies \text{strong admissibility} \implies \text{admissibility} .\]
\end{theorem}
\begin{proof}
To deal with the first implication, suppose $\mu=\mu_{\sigma}$, with $\sigma = \sum \bR_{>0}.N_i '$, and take $N_{\sigma}\in \sigma$, $F_{\sigma}\in \tilde{B}_{\bR}(\sigma)^{\circ}$. It will suffice to check (a) and (b) in Definition \ref{D:str} for $\be=\underline{1}$, since the faces of $\sigma$ are polarizable. We may assume (as in the proof of Theorem \ref{T:2}) that $(F_{\sigma},N_{\sigma})$ arises from $\fl_{\underline{1}}$ via \eqref{E:FN}. Writing $\fl_{\underline{1}}=\oplus \fl_{\underline{1}}^{p,p}$ for the corresponding decomposition, it is clear that $N_{\sigma}\in\fl^{-1,-1}_{\underline{1}} \implies \sigma \subset \tilde{\fl}_{\underline{1}}^{-1,-1} \implies \sigma \subset \tilde{L}_{\underline{1}}^{0,0}(\bR)^+.N_{\sigma}.$ More generally, $\sigma_{\be'} = \sum \bR_{>0}\be'(i)N_i ' \subset \tilde{L}_{\underline{1}}^{0,0}(\bR)^+.N_{\be'}$ follows from $\psi^{\circ}(\sigma_{\be'})=\mu_{\sigma}(\be')=[\fl_{\be'}]$, since $N_{\be'}$ is general in $\fl^{-1,-1}_{\be'}$ and $\sigma_{\be'}\subset \bar{\sigma}\subset \tilde{\fl}^{-1,-1}_{\underline{1}}$. In particular, there exist $\gamma_i \in \tilde{L}^{0,0}_{\underline{1}}(\bR)^+$ such that $\gamma_i N_i ' = N_i $, and (taking $N_i^{\underline{1}} := N_i '$, $g_i^{\underline{1}}:=\gamma_i^{-1}$) (a) and (b) are immediate.
\par For the second implication, observe that for each $\be$, (a) in Definition \ref{D:str} says that $\tilde{\fl}^{-1,-1}_{\be}$ contains an abelian subalgebra of dimension $|\be|$. But by the proof of Theorem 3.32 of \cite{MR3217458}, we have 
\begin{equation} \label{E:ker}
\ker \left\{ \delta: \bigwedge^k \tilde{\fl}^{-1,-1}_{\be} \to \tilde{\fl}^{-2,-2}_{\be} \otimes \bigwedge^{k-2} \tilde{\fl}^{-1,-1}_{\be} \right\} = \bigoplus_{w\in \cW_{\tilde{\fl}_{\be}}^{\#}(k)} \mathrm{span}\left\{ \tilde{L}^{0,0}_{\be}(\bR)^+.\hat{\fn}_w \right\}
\end{equation}
where $\hat{n}_w \subset \bigwedge^k \tilde{\fl}^{-1,-1}_{\be}$ is the top exterior power of $\fn_w = \oplus_{\delta \in \Delta(w)} (\tilde{\fl}^{-1,-1}_{\be})_{\delta}$.  The left-hand side of \eqref{E:ker} is nonzero iff an abelian subalgebra of dimension $k$ exists, while the right-hand side is nonzero iff $(\mathrm{cap}(\mu(\be))=)\,\mathrm{cap}(\fl_{\be})\geq k$.
\end{proof}

The symmetric group $\mathfrak{S}_n$ acts on $\cC_n$ (by $(\mathfrak{s}.\be)(j)=\be(\mathfrak{s}^{-1}j)$) and hence on the set of ordered $n$-cubes (by $(\mathfrak{s}.\mu)(\be)=\mu(\mathfrak{s}^{-1}.\be)$); obviously, polarizability, admissibility and strong admissibility are stable under $\mathfrak{S}_n$.
\begin{definition}\label{D:cube}
An {\it $n$-cube} $[\mu]$ is an $\mathfrak{S}_n$-equivalence class of ordered $n$-cubes.
\end{definition}
Now for each $I=\{i_1,\ldots, i_k\}\subseteq \{1,\ldots ,n\}$ there is an inclusion $\cC_k \overset{\imath_I}{\hookrightarrow}\cC_n$ defined by
\begin{equation*}
(\imath_I \be)(i):= \left\{ \begin{array}{cc} \be(j) & , \; i=i_j \\ 0 & , \; i\notin I . \end{array} \right.
\end{equation*}
There is a natural {\it inclusion relation} on the set of ordered cubes: given an ordered $k$-cube $\mu'$ and ordered $n$-cube $\mu$, we write $\mu' \leq \mu $ iff $k\leq n$ and $\mu' =\imath^*_I \mu$ for some $I$.
\begin{definition} \label{D:secpos}
The {\it secondary poset} $\tilde{\Psi}_D$ (or $\tilde{\Psi}_D^{\text{adm}}$) is the set of all admissible cubes $[\mu]$, with the inclusion relation.  Note that it is a poset by construction (even though neither $(\Psi_D,\preceq)$ nor $(\Psi_D,\leq)$ may be posets).  Write $\tilde{\Psi}_D^{\text{pol}}\subseteq \tilde{\Psi}_D^{\text{str}} \subseteq \tilde{\Psi}_D$ for the sub-posets of polarizable and strongly admissible cubes.
\end{definition}
\begin{example} \label{X:cubes}
The following $n$-cubes always belong to $\tilde{\Psi}_D^{\text{pol}}$:
\begin{a_list}
\item the trivial $0$-cube $\mu_0$. 
\item the $k$-cubes arising from $\tSL(2)^{\times k}$-orbits. (We mention this since the computation of $(\Psi_D,\preceq)$ produces a lot of $\tSL(2)^{\times 2}$-orbits in view of Theorem \ref{T:4}.)
\item for any polarized relation $[\fl']\preceq [\fl]$ in $\Psi_D$, the $2$-cube with $\mu(\be^1)=[\fl']$, $\mu(\be^2)=\mu(\underline{1})=[\fl]$. (Call this $\mu_{[\fl']\preceq [\fl]}$.)
\item for any $[\fl]\in \Psi_D$, with $k\leq \mathrm{cap}(\fl)$, the $k$-cube with $\mu(\be)=[\fl]$ ($\forall \be \neq \underline{0}$). (Call this $\mu_{[\fl]}^k$.)
\end{a_list}
Here (b) is immediate and (c) is essentially by the definition of polarized relations; while (d) is seen as follows: writing $[\fl]=[F,N]$, $\tilde{L}^{0,0}(\bR)^+.N$ is an open cone in $\tilde{\fl}^{-1,-1}_{\bR}$ hence contains a simplicial abelian (nilpotent) $k$-cone $\sigma \ni N$. (See the last step of the above proof.) Possibly shrinking this cone about $N$, we obtain the required positivity statement (that $F\in \tilde{B}(\sigma)$) as well as $\mu_{\sigma}=\mu^k_{[\fl]}$. Note that cones of types (c) and (d) typically don't arise from multi-$\tSL(2)$'s.
\end{example}

\begin{corollary} \label{C:finite}
$\tilde{\Psi}_D^{\cdots}$ is finite, and $(\tilde{\Psi}_D^{\cdots},\leq)$ surjects onto $(\Psi_D ,\preceq)$ {\rm[}resp. $(\mathbf{N}_D,\preceq)$, $(\mathbf{\Delta}_D,\preceq)${\rm ]} via the map $\Theta: \mu \mapsto \mu(\underline{1})$ {\rm [}resp. $\pi \circ \Theta$, $\phi_{\infty} \circ \Theta${\rm ]}. {\rm (}Here ``$\,\,\cdots$'' is ``adm'', ``str'', or ``pol''; and ``surjects'' means that every ``$\,\leq$'' maps to a ``$\,\preceq$'' and that every ``$\,\preceq$'' is obtained in this way.{\rm )}  Note that unlike $\pi$ and $\phi_{\infty}$, these are actual morphisms of posets.
\end{corollary}

\begin{proof}
Finiteness is immediate from the fact that $\mathrm{cap}[\fl]\leq \mathrm{cap}[\fg]<\infty$ for any $[\fl]\in \Psi_D$. The surjectivity statement follows from (b) in Definition \ref{D:onc} and Example \ref{X:cubes}(b).
\end{proof}

The computation of $\tilde{\Psi}_D^{\text{pol}}$ and its maps to $\mathbf{N}_D$ and $\mathbf{\Delta}_D$ might plausibly be regarded as a full solution,\footnote{However, there are invariants of nilpotent cones which are not well-defined on elements of $\tilde{\Psi}_D^{\text{pol}}$, such as the M-T group of the associated boundary component, or the Looijenga-Lunts group; this suggests a further refinement of the secondary poset, which we expect to address in a future work.} as far as the ``finite'' classification of nilpotent cones and their interaction with the $G(\bR)^+$-orbits on $\mathrm{Nilp}(\fg_{\bR})$ and $\check{D}$ are concerned. Moreover, $\tilde{\Psi}_D$ can be computed combinatorially, and comes remarkably close.

\begin{example} \label{X:f4}
($G=F_4$) To see in particular that $\tilde{\Psi}_D$ goes far beyond the multi--$\tSL(2)$'s involved in the construction of $(\Psi_D,\preceq)$, consider the case where $D$ is the $F_4$-adjoint (contact) domain \cite{KR1} parametrizing weight-two Hodge structures with Hodge numbers $(6,14,6)$. (Letting $\alpha_1,\alpha_2,\alpha_3,\alpha_4$ denote the simple roots, the grading element corresponding to $D$ is $\ttE=\ttS^1$.) Taking for $\fl$ the Levi (of type $A_2$) with simple roots $\alpha_1$ and $\alpha_1 + 3\alpha_2 + 4\alpha_3 + \alpha_4$, we have $\dim \fl^{-1,-1} = 2 = \dim (\fl^{\tss})^{0,0}$, but $\tilde{\fl}=\fg$ with $\tilde{\fl}^{-1,-1}=14$. Since $D$ is contact, the capacity is $\tfrac{1}{2} \cdot 14 = 7$; and so while the multi--$\tSL(2)$-orbits have dimension at most $\mathrm{rk}(G)=4$, we expect to find many admissible $7$-cubes with some of these polarizable (including $\mu^7_{[\fl]}$, as guaranteed by Example \ref{X:cubes}(d)).
\end{example}

\begin{remark} \label{R:f4}
The attentive reader will have noticed that we have said nothing about $G(\bR)^+$-conjugacy classes of polarizable nilpotent cones, concentrating instead on the coarser issue of what combinations of $G(\bR)^+$-conjugacy classes of 1-variable nilpotent orbits can appear on the faces. (This is of course valuable, as it determines, for an injective period map $(\Delta^*)^n \to \Gamma \backslash D$, the possible combinations of LMHS-types on the coordinate-$(\Delta^*)^k$'s.) The trouble is that the more refined classification certainly wouldn't be ``finite'', as the above example illustrates well: the space of abelian 7-dimensional subspaces in $\tilde{\fl}^{-1,-1}$ has dimension at least $7(14-7)-\binom{7}{2}=28$ (since $\dim \tilde{\fl}^{-2,-2} =1$), while the maximum dimension of an orbit of $\tilde{L}^{0,0}(\bR)^+$ on $Gr(7,\tilde{\fl}^{-1,-1})$ is $\dim \mathbb{P}\tilde{\fl}^{0,0} = \dim \tilde{\fl}^{0,0} - 1=21$.
\end{remark}

While we will not carry out large-scale computations of $\tilde{\Psi}_D^{\cdots}$ in this paper, in the remainder of this section we will present two examples which highlight aspects of the computation of $\tilde{\Psi}^{\text{str}}_D$ and $\tilde{\Psi}^{\text{pol}}_D$ once $\tilde{\Psi}_D$ is known.

\subsection{Some exceptional Mumford--Tate domains} \label{S:G2}

Let $G(\bR)$ be the (connected) noncompact real form of the exceptional, simple Lie group $G_2$ of rank two.  There are three Mumford--Tate domains $D$ (with bracket-generating IPR) for this group; they may be viewed as parameterizing Hodge structures on $V_\bR = \bR^7$ with Hodge numbers $\bh = (1,\ldots,1)$, $\bh=(1,2,1,2,1)$ and $\bh = (2,3,2)$, respectively.  In this section we will describe the the set $\Psi_D$, the relations $<$ and the polarized relations $\prec$ for each of these domains, and comment on the secondary posets $\tilde{\Psi}_D \supseteq \tilde{\Psi}_D^{\text{str}} \supseteq \tilde{\Psi}_D^{\text{pol}}$ introduced in $\S$\ref{S:cubes}.  The description of $\Psi_D$ will be given by the isomorphism \eqref{E:FNvl}.  To that end we recall the notation of \S\S\ref{S:rt} and \ref{S:weyl}, and that to describe a Levi $\fl \in \cL_{\varphi,\ft}$ representing $[F,N] \in \Psi_D$ it suffices to give simple roots $\sS' \subset \sR$ for $\fl_\bC$ (\S\ref{S:levi}). 

\begin{example}[Hodge numbers $\bh = (1,1,1,1,1,1,1)$] \label{eg:G2_12}
In this case we may choose our simple roots so that $\ttE_\varphi = \ttS^1+\ttS^2$, so that the subgroup $\sW^0 \subset \sW$ is trivial.  The poset $\Psi_D$ and the polarized relations are described in Proposition \ref{P:fd}; the three nontrivial elements are:
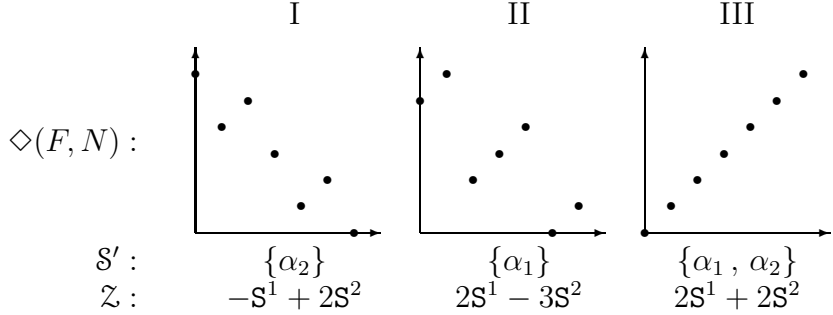
\begin{center}
\setlength{\unitlength}{10pt}
\begin{tabular}{rccc}
  & I & II & III \\ 
  \raisebox{6ex}[0pt]{$\Diamond(F,N) : \quad$}
  & \begin{picture}(7.5,7.5)
      \put(0,0){\vector(1,0){7}} \put(0,0){\vector(0,1){7}}  
      \put(0,6){\circle*{0.25}} \put(1,4){\circle*{0.25}} 
      \put(2,5){\circle*{0.25}} \put(3,3){\circle*{0.25}} 
      \put(4,1){\circle*{0.25}} \put(5,2){\circle*{0.25}} 
      \put(6,0){\circle*{0.25}}
    \end{picture} 
  & \begin{picture}(7.5,7.5)
       \put(0,0){\vector(1,0){7}} \put(0,0){\vector(0,1){7}}  
       \put(0,5){\circle*{0.25}} \put(1,6){\circle*{0.25}} 
       \put(2,2){\circle*{0.25}} \put(3,3){\circle*{0.25}} 
       \put(4,4){\circle*{0.25}} \put(5,0){\circle*{0.25}} 
       \put(6,1){\circle*{0.25}}
    \end{picture}
  & \begin{picture}(7,7.5)
      \put(0,0){\vector(1,0){7}} \put(0,0){\vector(0,1){7}}  
      \put(0,0){\circle*{0.25}} \put(1,1){\circle*{0.25}} 
      \put(2,2){\circle*{0.25}} \put(3,3){\circle*{0.25}} 
      \put(4,4){\circle*{0.25}} \put(5,5){\circle*{0.25}} 
      \put(6,6){\circle*{0.25}}
    \end{picture} 
  \\
  $\sS':\quad$ & $\{ \a_2 \}$ & $\{ \a_1 \}$ & $\{ \a_1\,,\,\a_2\}$ \\
  $\sZ:\quad$ & $-\ttS^1 + 2\ttS^2$ & $2\ttS^1-3\ttS^2$ & $2\ttS^1 + 2\ttS^2$
\end{tabular}
\end{center}
The nontrivial relations are 
\[
  \mathrm{I} , \mathrm{II} \ < \ \mathrm{III}\,,
\] 
but none of them are polarized. Note that for the corresponding orbits in $\overline{D}$, one has $D<\cO_{\mathrm{I}}<\cO_{\mathrm{II}}<\cO_{\mathrm{III}}$ \cite{MR3331177}.

Turning to the secondary poset, $\tilde{\Psi}_D = \tilde{\Psi}_D^{\text{str}} = \tilde{\Psi}_D^{\text{adm}}$ consists of the trivial $0$-cube $\mu_0$ and the $1$-cubes $\mu^1_{[\fl ]}$.  Denoting the latter by $\mu_1$, $\mu_2$, $\mu_3$, the poset is nothing but
$$ 
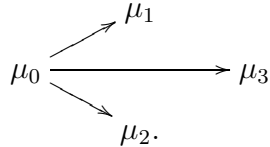
\xymatrixrowsep{0.3cm}
\xymatrixcolsep{0.8cm}
\xymatrix{
& \mu_1 \\ 
\mu_0 \ar [ur] \ar [rr] \ar [dr] & & \mu_3 \\ 
& \mu_2 .}
$$
\end{example}

\begin{example}[Hodge numbers $\bh = (1,2,1,2,1)$] \label{eg:G2_1}
In this case we may choose our simple roots so that $\ttE_\varphi = \ttS^1$.  The subgroup $\sW^0 \subset \sW$ is generated by the simple reflection $\{ (2) \}$.  There is only one nontrivial element in $\Psi_D$.  It is given by $\sS' = \{ \a_1 \}$ with $\sZ = 2\ttS^1 - 3\ttS^2$ and Hodge diamond
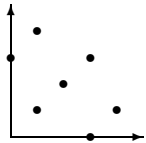
\begin{center}
\setlength{\unitlength}{10pt}
\begin{picture}(5,5)
  \put(0,0){\vector(1,0){5}} \put(0,0){\vector(0,1){5}}  
  \put(0,3){\circle*{0.25}} \put(1,1){\circle*{0.25}} 
  \put(1,4){\circle*{0.25}} \put(2,2){\circle*{0.25}} 
  \put(3,0){\circle*{0.25}} \put(3,3){\circle*{0.25}} 
  \put(4,1){\circle*{0.25}}
\end{picture}
\end{center}
For $\tilde{\Psi}_D = \tilde{\Psi}_D^{\text{str}} = \tilde{\Psi}_D^{\text{pol}}$, we simply have
$$ 
\xymatrixcolsep{0.8cm}
\xymatrix{
\mu_0 \ar[r] & \mu_1 .}
$$
\end{example}

\begin{example}[Hodge numbers $\bh = (2,3,2)$] \label{eg:G2_2}
(Note that this Mumford--Tate domain is a subdomain of the period domain in Example \ref{eg:H1} when $m=3$.)  In this case we may choose our simple roots so that $\ttE_\varphi = \ttS^2$, so that the subgroup $\sW^0 \subset \sW$ is generated by the simple reflection $(1)$.  The three nontrivial elements of $\Psi_D$, and their Hodge diamonds, are 
\begin{center}
\setlength{\unitlength}{10pt}
\begin{tabular}{c|ccc}
  & I & II & III \\ \hline
  \rb{$\Diamond(F,N)$}
  & \begin{picture}(3,3.5)
      \put(0,0){\vector(1,0){3}} \put(0,0){\vector(0,1){3}}  
      \put(0,1){\circle*{0.25}} \put(0,2){\circle*{0.25}} 
      \put(1,0){\circle*{0.25}} \put(1,1){\circle*{0.25}} 
      \put(1,2){\circle*{0.25}} \put(2,0){\circle*{0.25}} 
      \put(2,1){\circle*{0.25}}
    \end{picture}
  & \begin{picture}(3,3.5)
      \put(0,0){\vector(1,0){3}} \put(0,0){\vector(0,1){3}}  
      \put(0,0){\circle*{0.25}} \put(0,1){\circle*{0.25}} 
      \put(1,0){\circle*{0.25}} \put(1,1){\circle*{0.25}} 
      \put(1,2){\circle*{0.25}} \put(2,1){\circle*{0.25}} 
      \put(2,2){\circle*{0.25}}
    \end{picture} 
  & \begin{picture}(3,3.5)
      \put(0,0){\vector(1,0){3}} \put(0,0){\vector(0,1){3}}  
      \put(0,0){\circle*{0.25}} \put(1,1){\circle*{0.25}} 
      \put(2,2){\circle*{0.25}}
    \end{picture} 
  \\
  $\sS'$ & $\{\a_2\}$ & $\{ 2\a_1+\a_2\}$ & $\{\a_1,\a_2\}$ \\
  $\sZ$ & $-\ttS^1+2\ttS^2$ & $\ttS^1$ & $2\ttS^2$
\end{tabular}
\end{center}

To see that the nontrivial relations are
\[
  \mathrm{I}\,,\ \mathrm{II} \ < \ \mathrm{III}
\]
we recall Definition \ref{dfn:PsiPO} and note that
\begin{bcirclist}
\item 
$\tilde\fl_\mathrm{I} = \{ \ttS^1 = 0 \}$;
\item  
$\tilde\fl_\mathrm{II} = \{\ttS^1 - 2\ttS^2 = 0 \}$ and $(1)\tilde\fl_\mathrm{II} = \{ \ttS^1 - \ttS^2 \}$;
\item 
$\tilde\fl_\mathrm{III} = \fg$.
\end{bcirclist}

\noindent The relations trivially form a partial order.  To see that the relations are all polarized, observe that the roots $\sS_\mathrm{I}$ are strongly orthogonal to the roots $\sS_\mathrm{II}$; and therefore determine commuting $\fsl(2)$'s with $\sZ_\mathrm{I} + \sZ_\mathrm{II} = \sZ_\mathrm{III}$.

It is in this case that the computation of $\tilde{\Psi}_D \supset \tilde{\Psi}_D^{\text{str}} \supset \tilde{\Psi}_D^{\text{pol}}$ raises interesting issues.  Introducing the notation
$$
\langle \mu(\underline{\epsilon}^1) \mid \mu(\underline{1}) \mid \mu(\underline{\epsilon}^2) \rangle
$$
for admissible $2$-cubes, $\mu_{12} := \langle \rm{I} \mid \rm{III} \mid \rm{II} \rangle$ is automatically polarizable by Example \ref{X:cubes}(b). The remaining $2$-cubes
$$
\mu_{11} := \langle \rm{I} \mid \rm{III} \mid \rm{I} \rangle \;\;\;\;\text{and}\;\;\;\; \mu_{22} := \langle \rm{II} \mid \rm{III} \mid \rm{II} \rangle
$$
in $\tilde{\Psi}_D$ are \emph{not} obviously polarizable.

To specialize Definition \ref{D:str} to a $2$-cube $\mu$, suppose we have nilpotents $N$, $N_1$, $N_2$ as described there (with the Hodge-Tate grading imposed on $\tilde{\fl}$).  Then $\mu$ \emph{is strongly admissible if and only if}
\begin{equation} \label{E:8.2*}
\begin{split}
&\textit{there exist independent, commuting}\;\; \tilde{N}_i \in \tilde{L}^{0,0}(\mathbb{R})^+.N_i \;(i=1,2) 
\\
&\textit{such that} \;\; \mathbb{R}_{>0} \langle \tilde{N}_1 , \tilde{N}_2 \rangle \subset \tilde{L}^{0,0}(\bR)^+.N.
\end{split}
\end{equation}
It is clear that we may take $\tilde{N}_1 = N_1$ without loss of generality.
\begin{claim}
(a) $\mu_{22}$ is strongly admissible, and (b) $\mu_{11}$ is not.
\end{claim}
\begin{proof}
Relabel root spaces $\fg_{\alpha} = \mathbb{R}\langle X_{\alpha} \rangle$ of $\fg=\tilde{\fl}=\oplus \tilde{\fl}^{p,p}$ so that $\fg_{\alpha_1} \subset \tilde{\fl}^{0,0} = \mathbb{R}\langle X_{\alpha_2}, X_{\alpha_2 + \alpha_1}, X_{\alpha_2 + 2\alpha_1}, X_{\alpha_2 + 3\alpha_1} \rangle$, and $\tilde{L}^{0,0}(\mathbb{R})^+\cong \tGL_2(\mathbb{R})^+$ acts on $\fl^{-1,-1}_{\mathbb{R}}$ through $\mathrm{Sym}^3(\mathbb{R}^2)$.  We choose $\fl_{\rm{I}}$, $\fl_{\rm{II}}$ so that $\tilde{\fl}^{-1,-1}_{\rm{I}} = \langle X_{\alpha_2} \rangle$ and $\tilde{\fl}^{-1,-1}_{\rm{II}} = \langle X_{\alpha_2 + 2\alpha_1} \rangle$.

We begin with (b), taking $N_1 = N_2 = X_{\alpha_2}$.  A general $\tGL_2(\bR)^+$-conjugate of $N_2$ is
$$
\tilde{N}_2 = a^3 X_{\alpha_2} + 3a^2  b X_{\alpha_2 + \alpha_1} + 3ab^2 X_{\alpha_2 + 2\alpha_1} + b^3 X_{\alpha_2 + 3 \alpha_1} . 
$$
We require $0=[N_1,\tilde{N}_2]=b^3 X_{2\alpha_2 + 3\alpha_1}$, hence $b=0$; but then $\tilde{N}_2 = a^3 X_{\alpha_2}$ is not independent from $N_1$. So criterion \eqref{E:8.2*} is not satisfied.

For (a), we have $N_1 = N_2 =X_{\alpha_2 + 2\alpha_1}$.  A general $\tGL_2(\bR)^+$-conjugate of $N_2$ is $\tilde{N}_2 = \sum_{j=0}^3 A_j X_{\alpha_2 + j\alpha_1}$ where $A_0  = a^2 c$, $A_1 = 2abc + a^2 d$, $A_2 = 2abd + b^2 c$, and $A_3 = b^2 d$ with $ad-bc > 0$.  Such $[\underline{A}]$ satisfy\footnote{The equation defines the \emph{closure} of the orbit of $X_{\alpha_2 + 2\alpha_1}$ in $\bP \tilde{\fl}^{-1,-1}$; this includes the (twisted cubic) orbit of $X_{\alpha_2}$. }
\begin{equation} \label{E:8.2s}
(A_2 A_1 - 9A_0 A_3 )^2 = 4(A_2^2 - 3 A_1 A_3 )(A_1^2 - 3A_0 A_2 ),
\end{equation}
whose complement is the orbit in $\bP \tilde{\fl}^{-1,-1}$ of $N$ (the type $\rm{III}$ nilpotents).  In particular, if we take $(a,b,c,d)=(1,1,-\tfrac{1}{3},\tfrac{2}{3})$ $\implies$ $[\underline{A}]=[-\tfrac{1}{3}:0:1:\tfrac{2}{3}]$, then $\tilde{N}_2=-\tfrac{1}{3}X_{\alpha_2} + X_{\alpha_2 + 2\alpha_1} + \tfrac{2}{3}X_{\alpha_2 + 3\alpha_1}$ is independent of $N_1$, and commutes with it.  Moreover, we claim that any sum $r_1 N_1 +r_2 \tilde{N}_2$ ($r_1 , r_2 >0$) does not satisfy \eqref{E:8.2s}, hence is of type III (as required by \eqref{E:8.2*}). To see this, take $r_2 = 1$, $r_1 = r>0$ and write
$$
4\left((1+r)^2 - 3\cdot 0\cdot\tfrac{2}{3}\right) \left(0^2 - 3(\tfrac{-1}{3})(1+r)\right)-\left(0(1+r)-9(\tfrac{-1}{3})(\tfrac{2}{3})\right)^2 = 4(1+r)^3 - 4>0
$$
for $r>0$.
\end{proof}
Thus we have completely determined $\tilde{\Psi}^{\text{str}}$:
$$ 
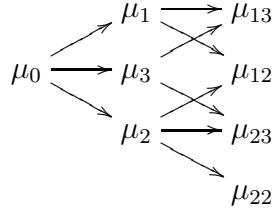
\xymatrixrowsep{0.3cm}
\xymatrixcolsep{0.8cm}
\xymatrix{
& \mu_1 \ar [r] \ar [rd] & \mu_{13} \\ 
\mu_0 \ar [ur] \ar [r] \ar [dr] & \mu_3 \ar [ur] \ar [dr] & \mu_{12} \\ 
& \mu_2 \ar [ur] \ar [r] \ar [dr] & \mu_{23} \\
& & \mu_{22} }
$$
whose only possible difference with $\tilde{\Psi}^{\text{pol}}$ is whether $\mu_{22}$ belongs to the latter.  That it does in fact belong, may be seen by a limiting argument. Begin by fixing $\sigma_0 = \bR_{>0} \langle N_1, N_2 (0)\rangle $ ($N_1 = X_{\alpha_2 + 2\alpha_1}, N_2 (0) = X_{\alpha_2}$) and $F^{\bullet} \in \partial D \subset \check{D}$ such that $(F^{\bullet},W_{\bullet}:=W(N_1 + N_2(0))_{\bullet})$ is $\bR$-split Hodge-Tate (guaranteed by $\mu_{12}\in \tilde{\Psi}^{\text{pol}}_D$), so that $e^{\bC\sigma_0}F^{\bullet}$ is a $\sigma_0$-nilpotent orbit. For $t<0$ set $N_2(t):=X_{\alpha_2} - 3t^2 X_{\alpha_2 + 2\alpha_1} - 2t^3 X_{\alpha_2 + 3\alpha_1}$, which corresponds to $(a,b,c,d)=(1,t,1,-2t)$ in the proof hence is of type II, and commutes with $N_1$.  Then $\sigma_t = \bR_{>0}\langle N_1 , N_2 (t)\rangle$ is of type III, and any $N_t \in \sigma_t$ induces $W(N_t )_{\bullet} = W_{\bullet}$. That $N_t$ polarizes $(F^{\bullet},W_{\bullet})$, so that $e^{\bC \sigma_t }F^{\bullet}$ is a $\sigma_t$-nilpotent orbit, is now immediate since $\sigma_t$ limits to $\sigma_0$ and this (positivity) statement holds for $\sigma_0$.
\end{example}

More generally, the methods of \cite[$\S$3]{BPR} may be useful for determining $\tilde{\Psi}^{\text{pol}}$ in some situations.  In the case of Calabi-Yau Hodge structures, there is another tool, which we will describe in the next section.

\subsection{Mirror symmetry and geometric realization} \label{S:matt}

We conclude by revisiting the period domain for ${\bf{h}} = (1,2,2,1)$ briefly treated in Example \ref{eg:CY0}, referring to Example \ref{eg:CY1} for the Hodge diamonds. By \cite{MR3217458}, the capacity of a given $[\fl]\in\Psi_D$ is the maximal dimension of an abelian subalgebra of $\tilde{\fl}$ contained in $\tilde{\fl}^{-1,-1}$. This easily allows us to determine that $\rm{cap}(\rm{I}_2)=2=\rm{cap}(\rm{II}_1)$ and $\rm{cap}(\rm{IV}_2)=3$, while the other four nontrivial elements have capacity $1$.

Suppose one wishes to determine the $2$-cubes of $\tilde{\Psi}_D^{\text{pol}}$.  (We shall say nothing about the $3$-cubes.) The first step would be to apply Example \ref{X:cubes}, which yields the following (partial) list of polarizable $2$-cubes:
\begin{description}
\item[(b)] $\langle \rm{I}_1\mid\rm{IV}_2\mid\rm{IV}_1\rangle$, $\langle \rm{III}_0\mid\rm{IV}_2\mid\rm{II}_1\rangle$, $\langle \rm{I}_1\mid\rm{II}_1\mid\rm{II}_0\rangle$, $\langle \rm{I}_1\mid\rm{I}_2\mid\rm{I}_1\rangle$ (the multi-$\rm{SL}(2)$'s: apply the algorithm of $\S$\ref{S:multvarSL2});
\item[(c)] one for each arrow not originating from $\rm{I}_0$: e.g. $\langle \rm{II}_0\mid \rm{II}_1\mid \rm{II}_1\rangle$;
\item[(d)] one for each type of capacity at least $2$: e.g. $\langle \rm{I}_2\mid \rm{I}_2\mid \rm{I}_2\rangle$ .
\end{description}
The remaining \emph{admissible} $2$-cubes are evidently:
\begin{itemize}
\item $\langle \rm{I}_1 \mid \rm{II}_1 \mid \rm{I}_1 \rangle$, $\langle \rm{II}_0 \mid \rm{II}_1 \mid \rm{II}_0\rangle $, $\langle \rm{I}_1 \mid \rm{IV}_2 \mid \rm{I}_1 \rangle$, $\langle \rm{II}_1 \mid \rm{IV}_2 \mid \rm{II}_1 \rangle$, $\langle \rm{II}_1 \mid \rm{IV}_2 \mid \rm{IV}_1 \rangle$
\end{itemize}
which one can show as in $\S$\ref{S:G2} (using criterion \eqref{E:8.2*}) are \emph{not} strongly admissible hence \emph{not} polarizable; and
\begin{itemize}
\item $\langle \rm{IV}_1 \mid \rm{IV}_2 \mid \rm{IV}_1 \rangle$, $\langle \rm{IV}_1 \mid \rm{IV}_2 \mid \rm{III}_0\rangle$, $\langle \rm{I}_1 \mid \rm{IV}_2 \mid \rm{III}_0 \rangle$, $\langle \rm{I}_1 \mid \rm{IV}_2 \mid \rm{II}_1 \rangle $, $\langle \rm{III}_0 \mid \rm{IV}_2 \mid \rm{III}_0 \rangle$
\end{itemize}
which \emph{are} strongly admissible.

\begin{claim}
$\langle \rm{IV}_1 \mid \rm{IV}_2 \mid \rm{III}_0 \rangle $ and $\langle \rm{I}_1 \mid \rm{IV}_2 \mid \rm{III}_0 \rangle$ are polarizable (in fact ``motivic'').
\end{claim}

\begin{proof}
We shall invoke Iritani's A-model $\mathbb{Z}$-VHS \cite{Iritani} in a special case described in \cite{BKV2}, first briefly recasting the latter in more Hodge-theoretic language.   Recall that for a unipotent VHS $\Phi$ over $(\Delta^*)^{\ell}$, the lift $\tilde{\Phi}: \mathfrak{H}^{\ell}\to D$ may be written uniquely in \emph{local normal form}
\begin{equation} \label{E:8.3*!}
\tilde{\Phi}(\underline{z})=e^{\sum z_j N_j} e^{\mu(\underline{q})}F^{\bullet}_{lim} ,
\end{equation}
where $q_j := \exp{(2 \pi {\bf i}z_j)}$ and $\mu:\Delta^{\ell}\to \oplus_{p<0;\,q} \fg^{p,q}_{lim}$ is holomorphic with $\mu(\underline{0})=0$. Here we take $\Phi$ to be of type $(1,\ell,\ell,1)$, with underlying local system $\mathbb{V}$ (with basis $\gamma_{\underline{z}}$); and $F^{\bullet}_{lim}$ to be expressed by writing the \emph{Hodge basis} $\omega = \{\omega^3; \omega^2_i ; \omega^1_i ; \omega^0 \}_{i=1}^{\ell}$ ($\omega_{\cdot}^p \in V_{lim}^{p,p}$) with respect to a basis $\tilde{\gamma}_0$ of $e^{-\sum z_i N_i}\mathbb{V}$ at $\underline{q}=\underline{0}$, as a matrix $\Omega_{lim} = {}_{\tilde{\gamma}_{\underline{0}}}[{\bf 1}]_{\omega}$. Likewise, \eqref{E:8.3*!} will be interpreted as an equality of matrices with $\tilde{\Phi}(\underline{z}) = {}_{\gamma_{\underline{z}}} [{\bf 1}]_{\omega}$, where ${\bf 1}$ is the identity transformation of the even cohomology of a certain CY 3-fold $\rm{X}^{\circ}$. In particular, $\omega$ will be a \emph{fixed} basis of $H^{\text{even}}(\rm{X}^{\circ})$, while the \emph{integral basis} $\gamma_{\underline{z}}$ varies.

Let $\hat{\Delta}^{\circ} \subset \bR^4$ be the convex hull of $(0,0,0,1)$, $(0,0,1,0)$, and $\Delta^{\circ} \times (-2,-3)$, where $\Delta^{\circ}\subset \bR^2$ is a reflexive polytope. A general anticanonical (CY 3-fold) hypersurface $\rm{X}^{\circ}$ in the associated toric variety $\PP_{\hat{\Delta}^{\circ}}$ has a natural (torically induced) elliptic fibration with section, over $\PP_{\Delta^{\circ}} = \mathbb{G}_m^2 \cup (\cup_{i=1}^r C_i )$; and $\{ C_i\}_{i=0}^{r-2} \subset H^4 (\rm{X}^{\circ})$ is a basis\footnote{so $\ell = r-1$; it turns out that $r$ is the number of integer points on the boundary of $\Delta \subset \bR^2$ (the dual of $\Delta^{\circ}$). } with dual $\{J_i \}_{i=0}^{r-2} \subset H^2 (\rm{X}^{\circ})$. Together these give rise to a basis $\cO:=\{\cO_{\rm{X}^{\circ}},\cO_{J_0},\ldots,\cO_{J_{r-2}},\cO_{C_0},\ldots ,\cO_{C_{r-2}},\cO_p \}$ of $K_0^{\text{num}}(\rm{X}^{\circ})$, and we set $$\psi(\underline{q}):=\sum_{\underline{k}\neq \underline{0}} \mathcal{N}_{\underline{k}} q_0^{k_0} \cdots q_{r-2}^{k_{r-2}}$$ where $\mathcal{N}_{\underline{k}}$ is the genus-$0$ Gromov-Witten invariant of class $\sum k_i [C_i]$.  (Write $\psi_i = \frac{\partial \psi}{\partial q_i}$ etc. for partial derivatives.) The Hodge basis is given by $\omega^3 = [\rm{X}^{\circ}]$, $\omega^2_i = [J_i]$, $\omega^1_i = [C_i]$, $\omega^0 = [p]$.

Now use the transcendental characteristic class
$$
\hat{\Gamma}({\rm X}^{\circ}) = 1 - \tfrac{1}{24} {\rm ch}_2({\rm X}^{\circ}) - \tfrac{2\zeta(3)}{(2\pi {\bf i})^3} {\rm ch}_3 ({\rm X}^{\circ}) \in K_0 ({\rm X}^{\circ})
$$
to define a transformation $\Gamma: K_0^{\text{num}}({\rm X}^{\circ}) \to H^{\text{even}}({\rm X}^{\circ},\mathbb{Q})$ by $\xi \mapsto [\hat{\Gamma}({\rm X}^{\circ} )] \cup {\rm ch}(\xi)$, with matrix $M:={}_{\omega}[\Gamma]_{\cO}$. Setting
$$
\Sigma_{\underline{q}} := \begin{pmatrix} 1 & 0 & 0 & 0 \\ 0 & 1 & 0&0 \\ -\psi_k & \psi_{kl} & 1 &0 \\ 2\psi & \psi_l & 0 & 1 \end{pmatrix}\, ,\;\;\;\;\;\; N_j :=\log [\cO(-J_j)\otimes \; ]_{\cO} \, ,
$$
the VHS over $(\Delta^*)^{r-1}$ defined (via \eqref{E:8.3*!}) by
$$
\Omega_{lim} :=M^{-1}\, ,\;\;\; \mu(\underline{q}) :=\log(M^{-1}\Sigma_{\underline{q}} M)
$$
is polarized by $Q(\alpha,\beta):=(-1)^{\frac{1}{2} \deg(\alpha)} \int_{{\rm X}^{\circ}} \alpha \cup \beta$.  Its motivic (geometric) realization as $H^3({\rm X}_t)$ (mirror family of CY 3-folds) is due to Iritani.  It is worked out in detail in \cite{BKV2}, where in particular one finds (in terms of intersection numbers on ${\rm X}^{\circ}$) that
\begin{equation} \label{E:8.3*!!}
N_j = \begin{pmatrix} 0 & 0 & 0 & 0 \\ -\delta_k^j & 0& 0& 0\\ -\tfrac{1}{2} J_j^2 J_k & -J_j J_k J_l &0 &0 \\ -\tfrac{1}{3}J_j^3 & -\tfrac{1}{2} J_j^2 J_l & -\delta_j^l & 0 \end{pmatrix} .
\end{equation}
Via the motivic interpretation, there is also a (codim. $1$) \emph{conifold} monodromy locus intersecting all the coordinate axes (as one variable leaves $\Delta^*$ whilst the others remain small), with monodromy $N_c$ of rank $1$.

Specializing \eqref{E:8.3*!!} to $\Delta^{\circ} = \rm{hull}\{ (2,-1),(-1,2),(-1,-1)\}$, so that $r=3$ and the Hodge numbers are $(1,2,2,1)$, we find 
$$
N_0 = \left( \begin{array}{c|cc|cc|c} 0&0&0&0&0&0 \\ \hline -1&0&0&0&0&0 \\ 0&0&0&0&0&0 \\ \hline -\tfrac{9}{2} & -9&-3&0&0&0 \\ -\tfrac{3}{2} & -3& -1& 0&0&0 \\ \hline -3& -\tfrac{9}{2}&-\tfrac{3}{2}& -1&0&0 \end{array} \right) \;\;\; \text{and} \;\;\; N_1 = \left( \begin{array}{c|cc|cc|c} 0&0&0&0&0&0 \\ \hline 0&0&0&0&0&0 \\ -1&0&0&0&0&0 \\ \hline -\tfrac{1}{2} & -3&-1&0&0&0 \\ 0 & -1& 0& 0&0&0 \\ \hline 0& -\tfrac{1}{2}& 0& 0&-1&0 \end{array} \right)
$$
The elements of ${\bf N}_D \simeq \Psi_D$ are distinguished by the list of ranks of $N$, $N^2$, $N^3$; for $N_0$, $N_1$, $N_c$, resp. $N_0+N_1$, these are $(3,2,1)$ [${\rm IV}_1$], $(4,2,0)$ [${\rm III}_0$], $(1,0,0)$ [${\rm I}_1$], resp. $(4,2,1)$ [${\rm IV}_2$]. Since ${\rm I}_1$ cannot degenerate to ${\rm III}_0$ or ${\rm IV}_1$, ${\rm III}_2$ is the only possibility for $N_0+N_c$ or $N_1+N_c$. Only the $(N_0,N_c)$ pair yields a case previously known to be polarizable.
\end{proof}

It turns out that two of the other strongly admissible 2-cubes have been shown to be polarizable: $\langle \rm{III}_0 \mid \rm{IV}_2 \mid \rm{III}_0 \rangle$ by \cite{grimm}; and $\langle \rm{IV}_1 \mid \rm{IV}_2 \mid \rm{IV}_1 \rangle$ in \cite[6.67]{BPR}. Determining the status of $\langle \rm{I}_1 \mid \rm{IV}_2 \mid \rm{II}_1 \rangle$ is left as an exercise to the reader!

\appendix

\section{Representation Theory Background} 

This is a laconic summary of representation theoretic results that are used in the paper.  For the material in \S\S\ref{S:rt} and \ref{S:wt} we recommend any standard reference, such as \cite{MR1153249, MR499562, MR1920389}; for the material in \S\S\ref{S:weyl}, \ref{S:ge} and \ref{S:levi} we recommend \cite{MR2532439}; for the material in \S\S\ref{S:st} and \ref{S:dge} we recommend \cite{\CoMc}; and for the material in \S\ref{S:hSL2} we recommend \cite{\schmid}.

\subsection{Roots} \label{S:rt}

Let $\fg_\bC$ be a complex semisimple Lie algebra.  A \emph{Cartan subalgebra} $\fh \subset \fg_\bC$ is a maximal abelian subalgebra consisting of semisimple elements.  There exist roots $\sR \subset \fh^*$ so that 
\[
  \fg_\bC \ = \ \fh \ \op \ \bigoplus_{\a\in\sR} \fg^\a \,,
\]
where 
\[
  \fg^\a \ := \ 
  \{ \xi \in \fg_\bC \ | \ [h , \xi] = \a(h) \ \forall \ h \in \fh \}
\]
is the one--dimensional \emph{root space} of $\a$.  One may always choose a basis $\sS = \{\a_1,\ldots,\a_r\}$ of $\fh^*$ with the property that $\sS \subset \sR$ and every root $\a \in \sR$ may be expressed as $\a = m^i \a_i$ with the $m^i$ either all nonnegative or all nonpositive; $\a$ is a \emph{positive} or \emph{negative root}, respectively.  The positive roots are denoted $\sR^+$.  Note that $\sR^+$ and $\sR^- := -\sR^+$ are disjoint and we have $\sR = \sR^+ \cup \sR^-$.

Let $\{\ttS^1,\ldots,\ttS^r\} \subset \fh$ denote the basis dual to the simple roots $\sS = \{\a_1,\ldots,\a_r\} \subset \fh^*$ of $\fg_\bC$.  

Every parabolic subgroup $P \subset G(\bC)$ may be realized as the stabilizer of a flag $F$ in a $G(\bC)$--homogeneous compact dual $\check D$.  Recall the decomposition \eqref{E:gphi}; when $P$ stabilizes a Hodge structure $\varphi \in D$, the Lie algebra of $P$ is 
\[
  \fp \ = \ \fg^{\ge0}_\varphi\,.
\]
A \emph{Borel} subgroup $B$ is a minimal parabolic subgroup.  The standard example of a Borel subalgebra is 
\begin{equation}\label{E:r1}
  \fb \ = \ \fh \ \op \ \bigoplus_{\a\in\sR^+} \fg^\a \,.
\end{equation}
Conversely, given a Borel subalgebra $\fb \subset \fg_\bC$, it is always possible to choose a Cartan $\fh \subset \fb$, and given such a Cartan, there is a unique choice of simple roots $\sS$ so that \eqref{E:r1} holds.

\subsection{Weyl group} \label{S:weyl}

Fix a complex semisimple Lie algebra $\fg_\bC$.  Given a Cartan subalgebra $\fh \subset \fg_\bC$, let $\sW \subset \tAut(\fh^*)$ denote the \emph{Weyl group} of $\fg_\bC$.\footnote{For a suitable realization of $\fg_\bC$ as a matrix subalgebra of $\fgl_n\bC$, we may identify $\fh$ with the subalgebra of diagonal matrices in $\fg_\bC$.}  Given a choice of simple roots $\sS=\{\a_1,\ldots,\a_r\} \subset \fh^*$, let $(i) \in \sW$ denote the simple reflection in the hyperplane orthogonal to $\a_i$, and let $(i_1\cdots i_\ell)$ denote the composition $(i_1)\circ\cdots\circ (i_\ell)$ of simple reflections.  Recall that $\sW$ is generated by the simple reflections $(i)$ subject to the following relations: $(i)^2 = \one$ for all $i$; and for all $i\not=j$:
\begin{center}
\begin{tabular}{ll}
$(ij)^2 = \one$, & if $\a_i$ and $\a_j$ are \emph{not} 
                   adjacent in the Dynkin diagram of $\fg_\bC$; \\
$(ij)^3 = \one$, & if $\a_i$ and $\a_j$ are joined by a single bond in the diagram;\\
$(ij)^4 = \one$, & if $\a_i$ and $\a_j$ are joined by a double bond;\\
$(ij)^6=\one$, & if $\a_i$ and $\a_j$ are joined by a triple bond.
\end{tabular}
\end{center}
See Figure \ref{f:dynkin} for the Dynkin diagrams.
\begin{figure}[!h] 
\begin{footnotesize}
\caption{Dynkin diagrams}
\begin{center}
\setlength{\unitlength}{3pt}
\begin{picture}(60,38)(-50,-5)
\multiput(-50,25)(7,0){3}{\circle*{1}}
\put(-50,25){\line(1,0){14}}
\multiput(-34.5,25)(2,0){3}{\circle*{0.3}}
\multiput(-29,25)(7,0){3}{\circle*{1}}
\put(-29,25){\line(1,0){14}}
\put(-51,26.3){{1}}\put(-44,26.3){{2}} 
\put(-25.7,26.3){{$r-1$}}\put(-16,26.3){{$r$}}
\put(-11,25){{$\fsl(r+1,\bC)$}}
\multiput(-50,16)(7,0){3}{\circle*{1}}
\put(-50,16){\line(1,0){14}}
\multiput(-34.5,16)(2,0){3}{\circle*{0.3}}
\multiput(-29,16)(7,0){3}{\circle*{1}}
\put(-29,16){\line(1,0){7}}
\multiput(-22,15.75)(0,0.5){2}{\line(1,0){7}}
\put(-18.7,15.15){$\rangle$}
\put(-51,17.3){{1}}\put(-44,17.3){{2}} 
\put(-25.7,17.3){{$r-1$}}\put(-16,17.3){{$r$}}
\put(-11,16){{$\fso(2r+1,\bC)$}}
\multiput(-50,7)(7,0){3}{\circle*{1}}
\put(-50,7){\line(1,0){14}}
\multiput(-34.5,7)(2,0){3}{\circle*{0.3}}
\multiput(-29,7)(7,0){3}{\circle*{1}}
\put(-29,7){\line(1,0){7}}
\multiput(-22,6.75)(0,0.5){2}{\line(1,0){7}}
\put(-18.7,6.15){$\langle$}
\put(-51,8.3){{1}}\put(-44,8.3){{2}} 
\put(-25.7,8.3){{$r-1$}}\put(-16,8.3){{$r$}}
\put(-11,7){{$\fsp(2r,\bC)$}}
\multiput(-50,0)(7,0){3}{\circle*{1}}
\put(-50,0){\line(1,0){14}}
\multiput(-34.5,0)(2,0){3}{\circle*{0.3}}
\multiput(-29,0)(7,0){2}{\circle*{1}}
\put(-29,0){\line(1,0){7}}
\multiput(-15,-2.33)(0,4.66){2}{\circle*{1}}
\put(-22,0){\line(3,1){7}}\put(-22,0){\line(3,-1){7}}
\put(-51,1.3){{1}}\put(-44,1.3){{2}} 
\put(-18,-3.7){{$r$}}\put(-18.7,3){{$r-1$}}
\put(-11,-0.5){{$\fso(2r,\bC)$}}
\end{picture}
\hspace{20pt}
\begin{picture}(50,38)(0,-5)
\multiput(0,25)(7,0){5}{\circle*{1}}
\put(0,25){\line(1,0){28}}
\put(14,30){\circle*{1}} \put(14,30){\line(0,-1){5}}
\put(-1,26.3){{1}}\put(6,26.3){{3}} \put(12,26.3){{4}}
\put(20,26.3){{5}}\put(27,26.3){{6}}
\put(14.7,29.4){{2}} \put(31,24.7){$E_6$}
\multiput(0,16)(7,0){6}{\circle*{1}}
\put(0,16){\line(1,0){35}}
\put(14,21){\circle*{1}} \put(14,21){\line(0,-1){5}}
\put(-1,17.3){{1}}\put(6,17.3){{3}} \put(12,17.3){{4}}
\put(20,17.3){{5}}\put(27,17.3){{6}}\put(34,17.3){{7}}
\put(14.7,20.4){{2}} \put(38,15.7){$E_7$}
\multiput(0,7)(7,0){7}{\circle*{1}}
\put(0,7){\line(1,0){42}}
\put(14,12){\circle*{1}} \put(14,12){\line(0,-1){5}}
\put(45,6.7){$E_8$} 
\put(-1,8.3){{1}}\put(6,8.3){{3}} \put(12,8.3){{4}}
\put(20,8.3){{5}}\put(27,8.3){{6}}\put(34,8.3){{7}}
\put(41,8.3){{8}} \put(14.7,11.4){{2}}
\multiput(0,0)(7,0){4}{\circle*{1}} 
\put(0,0){\line(1,0){7}}
\put(7,0.3){\line(1,0){7}} \put(7,-0.3){\line(1,0){7}}
\put(10,-0.85){$\rangle$}
\put(14,0){\line(1,0){7}} \put(25,-0.3){$F_4$}
\put(-1,1.3){{1}}\put(6,1.3){{2}}
\put(13,1.3){{3}}\put(20,1.3){{4}}
\multiput(40,0)(7,0){2}{\circle*{1}}
\put(40,0){\line(1,0){7}}\put(40,0.3){\line(1,0){7}}\put(40,-0.3){\line(1,0){7}}
\put(43.5,-0.85){$\langle$} \put(50,-0.3){$G_2$}
\put(39,1.3){{1}} \put(46,1.3){{2}}
\end{picture}
\label{f:dynkin}
\end{center}
\end{footnotesize}
\end{figure}
Note that $(i) \a_i = -\a_i$.  For $i\not=j$ the action of the simple reflection $(i)$ on the simple root $\a_j$ can be read off the Dynkin diagram as follows
\begin{eqnarray*}
  \setlength{\unitlength}{3pt}
  \begin{picture}(7,0)(0,-1)
  \put(0,0){\circle*{1}}\put(7,0){\circle*{1}}
  \put(-0.5,1){\footnotesize{$i$}} \put(6.5,1.5){\footnotesize{$j$}}
  \end{picture} \quad
  & \leadsto & (i)\a_j \, = \, \a_j \\
  \setlength{\unitlength}{3pt}
  \begin{picture}(7,0)(0,-1)
  \put(0,0){\circle*{1}}\put(7,0){\circle*{1}}
  \put(0,0){\line(1,0){7}}
  \put(-0.5,1){\footnotesize{$i$}} \put(6.5,1.5){\footnotesize{$j$}}
  \end{picture} \quad
  & \leadsto & (i)\a_j \, = \, \a_j + \a_i \\  
  \setlength{\unitlength}{3pt}
  \begin{picture}(7,0)(0,-1)
  \put(0,0){\circle*{1}}\put(7,0){\circle*{1}}
  \put(0,0.3){\line(1,0){7}} \put(0,-0.3){\line(1,0){7}}
  \put(3,-1){$\rangle$}
  \put(-0.5,1){\footnotesize{$i$}} \put(6.5,1.5){\footnotesize{$j$}}
  \end{picture} \quad
  & \leadsto & (i)\a_j \, = \, \a_j + \a_i \tand
  (j)\a_i \,=\, \a_i+2\a_j \\  
  \setlength{\unitlength}{3pt}
  \begin{picture}(7,0)(0,-1)
  \put(0,0){\circle*{1}}\put(7,0){\circle*{1}}
  \put(0,0.3){\line(1,0){7}} \put(0,-0.3){\line(1,0){7}}
  \put(0,0){\line(1,0){7}}
  \put(3,-1){$\rangle$}
  \put(-0.5,1){\footnotesize{$i$}} \put(6.5,1.5){\footnotesize{$j$}}
  \end{picture} \quad
  & \leadsto & (i)\a_j \, = \, \a_j + \a_i \tand
  (j)\a_i \,=\, \a_i+3\a_j 
\end{eqnarray*}


\begin{example}[Special Linear Algebra]
Let $\fg_\bC \simeq \fsl_n\bC$ be the algebra of trace--free linear maps $\bC^n\to\bC^n$.  Fix a basis of $\{e_1,\ldots,e_n\}$ of $\bC^n$ and let $\{e^1,\ldots,e^n\}$ be the dual basis of $(\bC^n)^*$ so that $\{ e_i\ot e^j \ | \ i\not=j\} \cup \{ e_i\ot e^i - e_{i+1}\ot e^{i+1} \ | \ 1 \le i \le n-1\}$ is a basis of $\fsl_n\bC$.  Then the diagonal subalgebra $\fh = \{ h = \sum_i h_i \,e_i\ot e^i \ | \ \sum_i h^i=0 \}$ is a Cartan subalgebra.  Define $\e_i \in \fh^*$ by $\e_i(h) = h_i$.  Then the roots are $\Delta = \{\e^i-\e^{j} \ | \ i\not= j\}\subset\fh^*$, and the Weyl group is the symmetric group $\cS_n$ permuting the $\e^i$.
\end{example}

\begin{remark}\label{R:WnS}
Fix a set of simple roots $\sS \subset \fh^*$ for $\fg_\bC$.  If $\sS' \subset \fh^*$ is a second set of simple roots, then there exists a unique $w \in \sW$ such that $\sS' = w\sS$.
\end{remark}

\subsection{Grading elements} \label{S:ge}

A \emph{grading element} is any semisimple element $\ttZ \in \fg_\bC$ acting on $\fg_\bC$ (via the adjoint action) by integer eigenvalues.  That is, 
\[
  \fg_\ell \ = \ \{ \xi \in \fg_\bC \ | \ [\ttZ,\xi] = \ell\,\xi \}
\]
is nonzero only if $\ell \in \bZ$, and $\fg_\bC$ admits a $\ttZ$--eigenspace decomposition
\[
  \fg_\bC \ = \ \fg_k \, \op\cdots\op\,\fg_{-k} \,.
\]

Note that the Jacobi identity implies
\[
  [\fg_a,\fg_b] \ \subset \ \fg_{a+b} \,.
\]
In fact, $\fg_{\ge0}$ is a parabolic subalgebra.  Every parabolic subalgebra $\fp$ may be realized in this fashion.  Two distinct grading elements may determine the same parabolic.  However, given a parabolic $\fp$, and a choice of Cartan and Borel subalgebra $\fh \subset \fb \subset\fp$, there exists a unique grading element $\ttZ \in \fh$ with $\fp = \fg_{\ge0}$ and such that $\fg_1$ generates $\fg_+$ as an algebra. 

We may always choose a Cartan subalgebra $\fh$ of $\fg_\bC$ so that 
\[
  \ttZ \ \in \ \fh \ \subset \ \fg_0\,.
\]  
We may further choose simple roots $\sS \subset \fh^*$ of $\fg_\bC$ so that 
\[
  \a(\ttZ) \ge 0 \quad\hbox{for all} \quad \a \in \sS\,.
\] 

For further discussion of grading elements in the context of Hodge theory, see \cite[\S2.2--2.3]{MR3217458}.

\subsection{Levi subalgebras} \label{S:levi}

A \emph{Levi subalgebra} $\fl_\bC \subset \fg_\bC$ arises as the commutator 
\[
  \fl_\bC \ = \ \{ \xi \in \fg_\bC \ | \ [\ttZ,\xi] = 0 \}
\]
of a semisimple element $\ttZ \in \fg_\bC$ acting on $\fg_\bC$ with integer eigenvalues.  Note that both $\fh$ and $\fg_\bC$ are Levi subalgebras.  More generally, Levi subalgebras are reductive subalgebras and as such 
\begin{subequations}\label{SE:proj}
\begin{equation} \label{E:ldecomp}
  \fl_\bC \ = \ \fz \,\op\, \fl^\tss_\bC
\end{equation}
decomposes as a direct sum of its center $\fz$ with the semisimple factor $\fl^\tss_\bC = [\fl_\bC,\fl_\bC]$.  Moreover, the decomposition \eqref{E:ldecomp} is orthogonal with respect to the Killing form on $\fl_\bC$.  (The Killing form on $\fl_\bC$ may be identified with restriction to $\fl_\bC$ of the Killing form on $\fg_\bC$.)  Let
\begin{equation}
  \pi_\fl^\tss : \fl_\bC \ \to \ \fl_\bC^\tss
\end{equation}
\end{subequations}
denote the projection to the semisimple factor.  

Any Levi $\fl_\bC$ contains a Cartan subalgebra $\fh \ni \ttZ$ of $\fg_\bC$.  Moreover, any set of simple roots $\sS'$ for $\fl_\bC^\tss$ may be realized as a subset of simple roots $\sS \subset \fh^*$ for $\fg_\bC$.  More precisely, $\fh' = \fh \cap \fl^\tss_\bC$ is a Cartan subalgebra of $\fl_\bC^\tss$, and any set of simple roots $\sS' \subset (\fh')^*$ for $\fl_\bC^\tss$ may be realized as $\left.\tilde\sS\right|_{\fh'}$ for some subset $\tilde\sS \subset \sS$ of simple roots $\sS \subset \fh^*$ of $\fg_\bC$.  In general, we will abuse notation and write $\sS' \subset \sS$.

\begin{remark}\label{R:finitelevis}
It follows from Remark \ref{R:WnS} and the discussion above that the number of Levi subalgebras containing a \emph{fixed} Cartan $\fh$ is \emph{finite}.
\end{remark}

Given a real form $\fg_\bR$ of $\fg_\bC$, we will say that a real subalgebra $\fl_\bR \subset \fg_\bR$ is a \emph{real Levi subalgebra} if the complexification $\fl_\bC = \fl_\bR \ot \bC$ is a Levi subalgebra of $\fg_\bC$.

\subsection{Fundamental weights} \label{S:wt}

Given a Cartan and Levi subalgebra $\fh \subset \fl_\bC$ there is a simple test to determine when an element $\z \in \fh$ lies in $\fl^\tss_\bC$.  A choice of simple roots $\sS = \{\a_1,\ldots,\a_r\}$ for $\fg_\bC$ determines a set of \emph{fundamental weights} $\{\w_1,\ldots,\w_r\}$.  We will need only two elementary properties of fundamental weights.  
\begin{a_list}
\item
If $\sS' \subset \sS$ is a set of simple roots for the semisimple factor $\fl^\tss_\bC$ of a Levi subalgebra $\fl_\bC$, then $\z\in \fh$ lies in $\fl^\tss_\bC$ if and only if $\w_i(\z) = 0$ for every $\a_i \in \sS \backslash \sS'$.
\item
Each $\w_i = q^i_j\,\a_j$ is a linear combination of the simple roots with \emph{positive} coefficients $0 < q^i_j \in \bQ$.  
\end{a_list}

\subsection{Standard triples} \label{S:st}

A \emph{standard triple} in $\fg$ is a set of three elements $\{ N^+ , Y , N\} \subset \fg$ such that 
$$
  [Y , N^+] \ = \ 2\,N^+ \,,\quad
  [N^+,N] \ = \ Y \tand
  [Y,N] \ = \ -2\,N \,.
$$
The elements $N^+ , Y , N$ are, respectively, the \emph{nilpositive}, \emph{neutral} and \emph{nilnegative} elements of the triple.  The neutral element $Y$ of a standard triple is a grading element \cite{\CoMc}.

\begin{example} \label{eg:stdtri_1}
The matrices 
\begin{equation} \label{E:stdtri_sl2}
  \bn^+ \ = \ \left(\begin{array}{cc} 0 & 1 \\ 0 & 0 \end{array}\right) \,,\quad 
  \by \ = \ \left(\begin{array}{cc} 1 & 0 \\ 0 & -1 \end{array}\right) \tand
  \bn \ = \ \left(\begin{array}{cc} 0 & 0 \\ 1 & 0 \end{array}\right)
\end{equation}
form a standard triple in $\fsl(2,\bR)$; while the matrices
\begin{equation} \label{E:stdtri_su11}
  \overline\be \ = \ \half \left(\begin{array}{cc} \bi & 1 \\
            1 & -\bi \end{array}\right) \,,\quad
  \bz \ = \ \left(\begin{array}{cc} 0 & \bi \\ 
           -\bi & 0 \end{array}\right) \tand
  \be = \half \left(\begin{array}{cc} -\bi & 1 \\ 
            1 & \bi \end{array}\right)
\end{equation}
form a standard triple in $\fsu(1,1)$.  The one-dimensional subalgebra spanned by $\bi\bz$ is a maximal compact Cartan subalgebra of $\fg_\bR = \fsl(2,\bR)$, and \eqref{E:stdtri_su11} is a DKS--triple (page \pageref{p:DKS}).
\end{example}

\subsection{Jacobson--Morosov filtrations} \label{S:JMf}

Given a nilpotent $N \in \tEnd(V)$, the Jacobson--Morosov Theorem asserts that $N$ may be realized as the nilnegative element of a standard triple $\{ N^+ , Y , N\} \subset \tEnd(V)$.  The vector space decomposes as a direct sum 
\[
  V \ = \ \bigoplus_{\ell\in\bZ} V_\ell
\]
of $Y$--eigenspaces with integer eigenvalues \cite[II.7]{MR499562}.  The \emph{Jacobson--Morosov filtration} $W(N)$ of $V$ is defined by 
\begin{equation} \label{E:W(N)}
  W_\ell(N) \ := \ \bigoplus_{m\le \ell} V_m \,.
\end{equation}
It is the unique increasing filtration of $V$ with the properties that 
\begin{i_list}
  \item $N (W_\ell(N)) \subset W_{\ell-2}(N)$, and 
  \item the induced map $N^\ell : W_\ell(N)/W_{\ell-1}(N) \to W_{-\ell}(N)/W_{-\ell-1}W(N)$ is a vector space isomorphism for all $\ell \ge 0$.
\end{i_list}
In particular, $W(N)$ does not depend on our choice of standard triple.

Given $k \in \bZ$, we define $W(N)[-k]$ to be the filtration 
\[
  W_\ell(N)[-k] \ = \ W_{\ell-k}(N) \,.
\]
When there exists $F \in \check D$ of weight $k$ (viewed as a filtration on $V_{\mathbb{C}}$) such that $z \mapsto e^{zN}F$ is a nilpotent orbit, $W_\ell(N)[-k]$ is the \emph{monodromy weight filtration}.

\subsection{Distinguished grading elements} \label{S:dge}
A grading element $Y \in \fg_\bC$ is \emph{distinguished} if the $Y$--eigenspace decomposition $\fg_\bC = \op_\ell \fg_\ell$, given by
\[
  \fg_\ell \ = \ \{ \xi \in \fg_\bC \ | \ [Y,\xi] = \ell\xi \}\,,
\]
satisfies two conditions:
\begin{a_list}
\item $\tdim\,\fg_0 = \tdim\,\fg_2$, and 
\item $\fg_2$ generates $\fg_+$ as an algebra.
\end{a_list}

Bala and Carter \cite{MR0417306, MR0417307} showed that a distinguished grading element $Y$ can be realized as the neutral element of a standard triple; see also \cite{MR1251060}.  However, it is not the case that every neutral element of a standard triple is a distinguished grading element.  In fact, the neutral element $Y$ of a standard triple $\{N^+,Y,N\} \subset \fg_\bC$ is a distinguished grading element if and only if there exists no proper Levi subalgebra $\fl_\bC \subsetneq \fg_\bC$ containing $N$ (equivalently, containing the standard triple).

Note that we include the ``trivial'' case $\fg=\{ 0\}$, $Y=0$, with the trivial DKS- and standard triples; this means in particular that $\mathcal{L}_{\varphi,\ft}$ in \eqref{E:dfncL} contains $\ft$ as an element.

\subsection{Horizontal $\tSL(2)$s} \label{S:hSL2}

Recall the notation \eqref{E:stdtri_su11}, and the decomposition \eqref{E:gphi}.  A \emph{horizontal $\tSL(2)$ at $\varphi \in D$} is given by a representation $\upsilon : \tSL(2,\bC) \to G(\bC)$ such that 
\begin{subequations}\label{SE:hTDS}
\begin{equation}
  \upsilon(\tSL(2,\bR)) \ \subset \ G(\bR)^+
\end{equation}
and 
\begin{equation} \label{E:hTDSb}
  \upsilon_* \overline\be \ \in\  \fg^1_\varphi \,,\quad
  \upsilon_* \bz \ \in\  \fg^0_\varphi \,,\quad
  \upsilon_* \be \ \in\  \fg^{-1}_\varphi \,.
\end{equation}
\end{subequations}
Note that \eqref{E:hTDSb} completely determines $\upsilon$, and 
\[
   \{ \overline\sE , \sZ , \sE  \} \ = \ 
   \upsilon_*\{ \overline\be , \bz , \be \}
\]
is a DKS--triple (page \pageref{p:DKS}).

\bibliographystyle{alpha}
\bibliography{}

\end{document}